%% file: imp_sampling_tau_leap_manuscript.tex
\documentclass[11pt]{article}

\usepackage{imp_samp_SRNS}
\usepackage[section]{placeins}
\usepackage{tabularx,ragged2e,booktabs,caption}
\usepackage{hyperref}
\usepackage{algorithm}
\usepackage{footnote}
\makesavenoteenv{algorithm}
\numberwithin{equation}{section}
\numberwithin{figure}{section}
\numberwithin{table}{section}
\numberwithin{algorithm}{section}

\newcommand{\expt}[1]{\mathrm{E}\left[#1\right]}
\newcommand{\exptpibar}[1]{\mathrm{E}_{\bar{\pi}_{\ell}}\left[#1\right]}

\newcommand{\rset}{\mathbb{R}}
\newcommand{\nset}{\mathbb{N}}
\newcommand{\zset}{\mathbb{Z}}

\newcommand{\PERIOD}{.}
\newcommand{\COMMA}{,}

\newcommand{\Ordo}[1]{{\mathcal{O}}\left(#1\right)}
\newcommand{\ordo}[1]{{o}\left(#1\right)}

\makeatletter
\def\BState{\State\hskip-\ALG@thistlm}
\makeatother

\title{Importance sampling for a robust  and efficient multilevel Monte Carlo estimator for  stochastic reaction networks} 

\author{Chiheb Ben Hammouda\thanks{King Abdullah University of Science and Technology (KAUST), Computer, Electrical and Mathematical Sciences \& Engineering Division (CEMSE), Thuwal $23955-6900$, Saudi Arabia ({\tt chiheb.benhammouda@kaust.edu.sa}).}
        \and Nadhir Ben Rached \thanks{Chair of Mathematics for Uncertainty Quantification, RWTH Aachen University, Aachen $52072$, Germany. ({\tt benrached@uq.rwth-aachen.de}).} 
\and  Ra\'ul Tempone\thanks{King Abdullah University of Science and Technology (KAUST), Computer, Electrical and Mathematical Sciences \& Engineering Division (CEMSE), Thuwal $23955-6900$, Saudi Arabia ({\tt raul.tempone@kaust.edu.sa}).} \thanks{Alexander von Humboldt Professor in Mathematics for Uncertainty Quantification, RWTH Aachen University, Aachen $52072$, Germany.}}

\begin{document}
	\date{}
\maketitle
\begin{abstract}
The multilevel Monte Carlo (MLMC) method for continuous-time Markov chains, first introduced by
Anderson and Higham \cite{Anderson2012}, is a highly efficient simulation technique that can be used to estimate various  statistical quantities for stochastic reaction networks (SRNs), in particular for stochastic biological systems. Unfortunately, the robustness and performance of the multilevel method can be affected  by the high kurtosis, a phenomenon  observed at the deep levels of MLMC, which leads to inaccurate estimates of  the sample variance. In this work, we address cases where the high-kurtosis phenomenon is due to \textit{catastrophic coupling} (characteristic of pure jump processes where coupled consecutive paths are identical in most of the simulations, while  differences only appear in a tiny proportion) and introduce a pathwise-dependent importance sampling (IS) technique that improves the robustness and efficiency of the multilevel method.   Our  theoretical results, along with the conducted numerical experiments, demonstrate that our proposed method significantly reduces  the kurtosis of the deep levels of MLMC,  and also improves the strong convergence rate from $\beta=1$ for the standard case (without IS), to $\beta=1+\delta$, where $0<\delta<1$ is a user-selected parameter in our  IS algorithm. Due to the complexity theorem of MLMC, and given  a pre-selected tolerance, $\text{TOL}$, this results in an improvement of the complexity from  $\Ordo{\text{TOL}^{-2} \log(\text{TOL})^2}$ in the standard case to $\Ordo{\text{TOL}^{-2}}$, which is the optimal complexity of the MLMC estimator. We achieve all these improvements with a  negligible additional cost since our IS algorithm is only applied a few times across each simulated path.
 
\

\textbf{Keywords} Multilevel Monte Carlo. Continuous-time Markov chains. Stochastic reaction networks. Stochastic biological systems. Importance sampling.

\textbf{2010 Mathematics Subject Classification} 		60H35. 60J27. 60J75. 92C40.
\end{abstract}

\thispagestyle{plain}

\setcounter{tocdepth}{1}

 \section{Introduction}
In this work, we propose a novel importance sampling (IS) algorithm that can be combined with the multilevel Monte Carlo (MLMC) estimator to numerically solve stochastic differential equations (SDEs) driven  by Poisson random measures \cite{li2007analysis,cinlar2011probability}. 

We  focus on a particular class of continuous-time Markov chains known as stochastic reaction networks (SRNs) (see Section \ref{sec:pjp} for a short introduction). SRNs describe the time evolution of biochemical reactions, epidemic processes \cite{brauer2001mathematical,anderson2015stochastic}, and transcription and translation in genomics and virus kinetics \cite{srivastava2002stochastic,hensel2009stochastic},  among other important applications. 

Let $\mathbf{X}$ be an SRN taking values in $\nset^d$ and defined in the time-interval $[0,T]$, where $T>0$ is a user-selected final time. We aim to provide accurate MLMC estimations of the expected value, $\expt{g(\mathbf{X}(T))}$,  where $g:\rset^d\to\rset$ is a given   scalar observable of $\mathbf{X}$. 

The main goal of our new proposed method is to improve the robustness  and performance of the MLMC estimator by i) solving the  high-kurtosis phenomenon encountered when using the multilevel method in the context of continuous-time Markov chains (see Section \ref{sec:The large kurtosis issue}), and ii) improving the complexity of the MLMC estimator by  increasing   the strong  convergence rate.

Many methods have been developed to simulate exact sample paths of SRNs;
for instance, the {stochastic simulation algorithm} (SSA) was
introduced by Gillespie in \cite{gillespie1976general} and the {modified next reaction method} (MNRM) was proposed by Anderson in \cite{anderson2007modified}.  Pathwise exact realizations of SRNs may be  computationally very costly when some reaction channels have high reaction rates. To overcome this issue, Gillespie \cite{gillespie2001approximate} and Aparicio and Solari  \cite{aparicio2001population} independently proposed the explicit tau-leap (TL) method  (see Section \ref{sec:exp_tau})  to simulate approximate paths of  $\mathbf{X}$ by evolving the process with fixed time steps, keeping the reaction rates fixed within each time step. Furthermore, other simulation schemes have been proposed to deal with situations with well-separated fast and slow time scales \cite{cao2005trapezoidal,rathinam2007reversible,abdulle2010chebyshev,ahn2013implicit,moraes2016multilevel_splitting,hammouda2017multilevel}.

To reduce the computational work needed to estimate  $\expt{g(\mathbf{X}(T))}$, Anderson and Higham \cite{Anderson2012} introduced the MLMC method \cite{giles2008multilevel,giles2015multilevel} based on the explicit TL scheme in the context of SRNs. Many extensions of the MLMC method have since been introduced to address other challenges. For instance, adaptive multilevel estimators \cite{lester2015adaptive,moraes2016multilevel,moraes2016multilevel_splitting} were proposed to improve the performance of non-adaptive estimators \cite{Anderson2012} to simulate SRNs with markedly different time scales. \cite{hammouda2017multilevel} extended \cite{Anderson2012} to  systems with slow and fast time scales, and introduced a hybrid multilevel estimator that uses an  implicit scheme for levels where explicit TL cannot be used due to numerical instability.

One important challenge encountered when using MLMC in the context of SRNs is the high-kurtosis phenomenon (see Section  \ref{sec:The large kurtosis issue} for more details), which may occur due to  either \textit{catastrophic coupling} (characteristic of pure jump processes where coupled consecutive paths are identical in most of the simulations, while  differences only appear in a tiny proportion; see Section \ref{sec:Catastrophic coupling} for more details) or \textit{catastrophic decoupling} (observed for general stochastic processes where terminal values of the sample paths of both coarse and fine levels become very different from each other; see Section  \ref{sec:Catastrophic decoupling} for more details). This  poor  behavior of the kurtosis affects the accurate estimation of the sample variance needed for the MLMC algorithm.  Consequently,  it affects the robustness and performance  of the multilevel estimator in many cases   (see Section \ref{sec:The large kurtosis issue} for the illustration of this issue). As of today,  few  works have addressed this issue; for instance, the authors in \cite{moraes2016multilevel} mentioned this issue and developed a more accurate estimator for the multilevel variance based on dual-weighted residual expansion techniques. In \cite{lester2018robustly}, a new method has been proposed to address the high-kurtosis phenomenon when it is due to \textit{catastrophic decoupling}, and  introduced a new approach of coupling consecutive levels of MLMC called  the \textit{common process method} (CPM), instead of using the  \textit{split propensity method} (SPM) proposed in \cite{Anderson2012}. The CPM is based on the use of common inhomogeneous Poisson processes for both coarse and fine sample paths.  Although the CPM improves the robustness and reliability of the multilevel estimator by dramatically decreasing  the kurtosis, it nonetheless incurs remarkable additional computational and memory costs  because  for each level  it requires  i) running the TL algorithm twice, and ii) storing the total number of times each Poisson process has fired over each time step. 

In the work presented here, compared to \cite{lester2018robustly}, we address cases of  high kurtosis observed in the MLMC estimator due to \textit{catastrophic coupling} and   propose a novel method that  provides a more robust multilevel estimator. We introduce a  pathwise-dependent IS technique to dramatically  decrease the high kurtosis caused by the  SPM strategy for coupling the paths of two consecutive levels. We should note that other IS  methods were proposed, in the context of biochemical systems and SRNs, but for the efficient estimation of rare events \cite{kuwahara2008efficient,daigle2011automated,cao2013adaptively}. Furthermore,  these IS  methods were combined with the MC method instead of the MLMC method that we present here.

We show that our  proposed method  not only improves the robustness of the multilevel estimator by significantly reducing  the kurtosis, but also  improves the strong convergence rate from $\beta=1$ for the standard case (without IS), to $\beta=1+\delta$, where $0<\delta<1$ is  a user-selected parameter in our  IS algorithm. Due to the complexity theorem of MLMC \cite{cliffe2011multilevel}, and given  a pre-selected tolerance, $\text{TOL}$, this results in an improvement of the complexity of MLMC from  $\Ordo{\text{TOL}^{-2} \log(\text{TOL})^2}$  to the optimal complexity, \ie,  $\Ordo{\text{TOL}^{-2}}$. We achieve all these improvements with a  negligible additional cost since our IS algorithm is only applied a few times across each simulated path.

Alternatively, the optimal MLMC complexity of order $\Ordo{\text{TOL}^{-2}}$ can be achieved by using  (i) MC with an exact scheme (for instance SSA), or (ii) an unbiased MLMC estimator \cite{Anderson2012}, where the deepest level is simulated with an exact scheme, or (iii) a biased hybrid MLMC estimator  \cite{moraes2016multilevel}, where the paths are simulated in a hybrid fashion that switches adaptively, based on the relative computational cost, between the TL and an exact method. Both approaches (i) and (ii)  incur a substantial  additional cost by introducing an exact scheme. This significant additional cost is not manifested in the rate exponent but in a large constant that deteriorates the actual complexity. Although our method is based on a biased MLMC estimator, without steps simulated with an exact scheme as in \cite{moraes2016multilevel}, it still achieves   a complexity of order $\Ordo{\text{TOL}^{-2}}$ with a smaller constant than those produced by   the methods (i), (ii) and (iii)  mentioned above. Compared to \cite{Anderson2012}, we suggest an orthogonal approach of lowering the complexity rate by improving the strong convergence rate, instead of removing the bias (weak error). Similarly to our work, the authors in \cite{moraes2016multilevel}  improve the strong convergence rate to reach the complexity of order $\Ordo{\text{TOL}^{-2}}$. However, compared to \cite{moraes2016multilevel}, we use a different strategy based on a  pathwise-dependent IS coupled with the TL scheme, instead of using a hybrid approach that switches between an exact and the TL scheme. 

We also propose a new approach to overcome the high-kurtosis phenomenon, which affects the robustness and  reliability of the MLMC estimator introduced in \cite{Anderson2012}. Although this issue can be addressed differently,  using the dual-weighted residual expansion techniques developed in \cite{moraes2016multilevel} in order to estimate more accurately the sample variance and bias on the deepest levels of MLMC, we believe that our approach has two main advantages over the approach in \cite{moraes2016multilevel}: first, our method is much simpler and  easier to generalize to other schemes, such as the split-step implicit TL scheme \cite{hammouda2017multilevel} where it is difficult  to get estimates using  the dual-weighted residual expansion techniques. Furthermore, although the approach in \cite{moraes2016multilevel} provides  a more accurate estimate of the variance than the sample variance estimate, there is still no clear analysis of how accurate (biased) those estimates are. The difficulty of establishing such analysis is  mainly  due to  the lack of sharp concentration inequalities for linear combinations of independent Poisson random variables (rdvs), as stated in Remark 4 in \cite{moraes2016multilevel}.  Finally, we should emphasize that  the hybrid scheme in \cite{moraes2016multilevel}   is an efficient algorithm that avoids   the simulated paths to take negative values, which is an undesirable consequence of the TL approximation. In this case, for problems where we are close to the boundary, combining the two approaches (our approach and the approach in \cite{moraes2016multilevel}) may lead to  more  efficient results.

This work is structured  as follows: we start by giving an overview of concepts used in this work such as SRNs (Section \ref{sec:pjp}), explicit TL approximation (Section \ref{sec:exp_tau}), and the MLMC method (Section \ref{sec:MLMC)}). Then, in Section \ref{sec:The large kurtosis issue}, we explain the high-kurtosis phenomenon along with its  leading causes in the context of SRNs. In Sections  \ref{sec:Idea}, \ref{sec:main results} and, \ref{sec:About the choice of the new measure and the importance sampling algorithm}, we present the details of our IS algorithm that we combine with the MLMC method. We start by presenting in Section \ref{sec:Idea} the motivation of our idea by the sampling under an optimal measure for simulating SRNs. Then, in Section \ref{sec:main results}, we present a summary of the main results of this work, and  in Section \ref{sec:About the choice of the new measure and the importance sampling algorithm}, we analyze our proposed IS algorithm and state the main convergence theorems related to the kurtosis and the variance estimates of our approach. Furthermore, we  present, in the same section, a cost analysis of the MLMC methods presented in this work, with and without IS.  Before concluding,  we show,  in Section \ref{sec:num_experiments}, the results obtained through the  numerical experiments conducted across different examples of SRNs.
\input{Stochastic_reaction_networks.tex}

\input{TL_approx.tex}

\input{mlmc.tex}

\input{kurt_issue.tex}

\input{const_imp_sampling_algorithm.tex}

\section{Numerical Experiments}\label{sec:num_experiments}
\input{Num_experiments.tex}

\section{Conclusions and Future Work}\label{sec:Conclusions and future work}
In the work presented here, we address  the   high-kurtosis phenomenon related to catastrophic coupling, and observed in MLMC estimators when applied  in the context of SRNs and pure jumps. We propose a novel path-dependent IS algorithm to be used  with MLMC, in order to improve  robustness and computational performance.

Our theoretical results and  numerical experiments show that our  proposed method  not only improves the robustness of the multilevel estimator by dramatically reducing  the kurtosis, but also  improves the strong convergence rate, which results in an improvement of the complexity of the MLMC method, from  $\Ordo{\text{TOL}^{-2} \log(\text{TOL})^2}$  to $\Ordo{\text{TOL}^{-2}}$, with  $\text{TOL}$ being a pre-selected tolerance. We achieve all these improvements with a  negligible additional cost since our IS algorithm is only applied a few times across each simulated path.

Here, we limit ourselves to the use of the IS technique with an explicit TL scheme. In  a future study, we intend to investigate the potential of our proposed algorithm when using a split-step implicit TL scheme, as proposed in \cite{hammouda2017multilevel}, which is required for systems with the presence of slow and fast timescales (stiff systems). To overcome the  \textit{catastrophic coupling} issue, the authors in \cite{hammouda2017multilevel} used extrapolation      to estimate the sample variance when using MLMC. We believe that our new IS  technique may help to obtain accurate estimates of the  sample variances needed  by the MLMC estimator. Another potential  research direction may be to  investigate a more optimal IS scheme to be used for MLMC; for instance, we may try to use a hierarchy of  $\delta_{\ell}$, where the parameter $\delta$ used in our proposed method would depend on the level of discretization. Furthermore, we may explore the possibility of introducing a new IS scheme for MLMC based on SPM coupling, to address the \textit{catastrophic decoupling} issue, which is  the second cause of the  high-kurtosis phenomenon in the context of SRNs when using MLMC. Finally, we can combine the strengths of our method and the hybrid approach  in \cite{moraes2016multilevel} to improve the performance of the MLMC estimator.

\textbf{Acknowledgments} This work was supported by the KAUST Office of Sponsored Research (OSR) under Award No. URF/1/2584-01-01 and the Alexander von Humboldt Foundation. C. Ben Hammouda and R. Tempone are members of the KAUST SRI Center for Uncertainty Quantification in Computational Science and Engineering.  The authors would like to thank Dr. Alvaro Moraes and Sophia  Franziska Wiechert for their helpful and constructive comments. The authors are also  very grateful to the anonymous referees for their valuable comments and suggestions  that greatly contributed to shape the final version of the work.


\bibliographystyle{plain}
\bibliography{imp_samp_SRNS} 
\appendix
\input{appendix.tex}

\end{document}

%% file: Stochastic_reaction_networks.tex
\subsection{Stochastic Reaction Networks (SRNs)}
\label{sec:pjp}
We are interested in the time evolution of a homogeneously mixed  chemical reacting system described by the Markovian pure jump process,  $\mathbf{X}:[0,T]\times \Omega \to \nset^d$, where ($\Omega$, $\mathcal{F}$, $P$) is a probability space. In this framework, we assume that $d$ different species interact through $J$ reaction channels. 
The $i$-th component,  $X^{(i)}(t)$, describes the abundance of the $i$-th species present in the chemical system at time $t$. This work aims to  study the time evolution of the state vector, 
\begin{equation*}
 \mathbf{X}(t) = (X^{(1)}(t), \ldots, X^{(d)}(t)) \in
  \nset^d \PERIOD
\end{equation*}
Each reaction channel, $\mathcal{R}_j$, is a pair $(a_j, \boldsymbol{\nu}_{j})$ defined by its propensity function, $a_{j}:\rset^{d} \rightarrow \rset_{+}$, and its state change vector, $ \boldsymbol{\nu}_{j}=( \nu_{j,1},\nu_{j,2},..., \nu_{j,d})$, satisfying\footnote{Hereafter,  we use $\text{Prob}\left(A;B\right)$ and  $\expt{A;B}$ to denote the conditional probability and  conditional  expectation of $A$ given $B$, respectively.}
\begin{align}\label{reaction_channel}
 \text{Prob}\left(\mathbf{X}(t+ \Delta t)=\mathbf{x}+ \boldsymbol{\nu}_{j} ; \mathbf{X}(t)=\mathbf{x}\right)=a_{j}(\mathbf{x})\Delta t + \ordo{\Delta t}, \: j=1,2,...,J  \PERIOD 
\end{align}
Formula   \eqref{reaction_channel} states that the probability of observing a jump in the process, $\mathbf{X}$, from  state $\mathbf{x}$ to  state $\mathbf{x} +  \boldsymbol{\nu}_{j}$, a consequence of the firing of  reaction $\mathcal{R}_{j}$ during a small time interval, $(t, t + \Delta  t]$, is proportional to the length of the time interval, $\Delta  t$, with $a_{j}(\mathbf{x})$  as the constant of proportionality. 

We set $a_j(\mathbf{x}){=}0$ for  $\mathbf{x}$ such that $\mathbf{x}{+}\boldsymbol{\nu}_j\notin \nset^d$ (\emph{the non-negativity assumption}: the system can never produce negative population values). 

As a consequence of relation (\ref{reaction_channel}), the process $\mathbf{X}$  is a continuous-time, discrete-space Markov chain that can be characterized by the random time change representation of Kurtz \cite{kurtz_2005}
\begin{equation}
  \label{eq:exact_process}
\mathbf{X}(t)= \mathbf{x}_{0}+\sum_{j=1}^{J} Y_j \LP \int_0^t  a_{j}(\mathbf{X}(s)) \, \ud s \RP \boldsymbol{\nu}_j  \COMMA
\end{equation}
where $Y_j:\rset_+{\times} \Omega \to \nset$ are independent unit-rate Poisson
processes. 
Conditions on the  reaction channels  can be imposed to ensure uniqueness  \cite{anderson2015stochastic} and to avoid explosions in finite time \cite{engblom2012stability,rathinam2013moment,gupta2014scalable}.

We emphasize that, by using the \textit{stochastic mass-action kinetics}
principle,  we assume that the propensity function, $a_j(.)$, for a reaction channel $\mathcal{R}_j$, represented by the following  diagram\footnote{$\alpha_{j,i}$ molecules of the species $S_i$ are consumed and $\beta_{j,i}$ are produced. Thus, $(\alpha_{j,i},\beta_{j,i}) \in \nset^2$ but $\beta_{j,i}-\alpha_{j,i}$, can be a negative integer, constituting the vector $\boldsymbol{\nu}_j=\left(\beta_{j,1}-\alpha_{j,1},\dots,\beta_{j,d}-\alpha_{j,d}\right) \in \zset^d$.}
\begin{equation*}
\alpha_{j,1} S_1+\dots+\alpha_{j,d} S_d \overset{\theta_j}{\rightarrow}\beta_{j,1} S_1+\dots+\beta_{j,d} S_d \COMMA
\end{equation*}
obeys the following  relation 
\begin{equation}\label{eq:prop_dynamics}
a_j(\mathbf{x}):=\theta_j \prod_{i=1}^d \frac{x_i!}{(x_i-\alpha_{j,i})!} \mathbf{1}_{\{x_i\ge \alpha_{j,i}\}}\COMMA
\end{equation}
where $\{\theta_j\}_{j=1}^J$ are positive constant reaction rates, $x_i$ is the counting number of the species $S_i$, and $\mathbf{1}_{\mathcal{A}}$ is the indicator function of the set $\mathcal{A}$.

%% file: TL_approx.tex
\subsection{The Explicit Tau-Leap (Explicit-TL) Approximation}
\label{sec:exp_tau}
The explicit-TL scheme is a pathwise-approximate method independently introduced in \cite{gillespie2001approximate} and \cite{aparicio2001population} to overcome the computational drawback of exact methods, \ie, when many reactions fire during a short time interval. This scheme can be derived from the random time change representation of Kurtz  \eqref{eq:exact_process} by approximating the integral $\int_{t_i}^{t_{i+1}} a_{j}(\mathbf{X}(s)) \ud s $ by $a_j(\mathbf{X}(t_i))\,(t_{i+1}-t_i)$, \ie, using the forward-Euler method with a time mesh  $\{t_{0}=0, t_{1},...,t_{N}= T\}$. In this way, the  \nameexp approximation of $\mathbf{X}$ should satisfy for $k\in\{1,2,\ldots,N\}$
\begin{equation*}\label{approx}
\mathbf{Z}(t_{k}) = \mathbf{x}_{0}+\sum_{j=1}^{J} Y_{j} \LP  \sum_{i=0}^{k-1} a_{j}(\mathbf{Z}(t_i))(t_{i+1}-t_{i}) \RP   \boldsymbol{\nu}_{j} \PERIOD
\end{equation*}
Given a uniform time mesh of size $\Delta t$ and $\mathbf{Z}(t_0) := \mathbf{x}_{0}$, we  simulate a path of $\mathbf{Z}$ as follows
\begin{equation*}
\mathbf{Z}(t_k):=\mathbf{z}+\sum_{j=1}^{J} \mathcal{P}_{j}(a_{j}(\mathbf{z}) \Delta t)  \boldsymbol{\nu}_{j} \COMMA \: 1 \le k \le N,
\end{equation*}
iteratively, where $\mathbf{z}=\mathbf{Z}(t_{k-1})$ and $\{\mathcal{P}_{j}(r_j)\}_{j=1}^J $ are independent Poisson rdvs with respective rates, $r_j$. Note that the  \nameexp path, $\mathbf{Z}$, is defined only at the points of the time mesh, but it can be naturally extended to $[0,T]$  as a piecewise constant path.

%% file: mlmc.tex
\subsection{The Multilevel Monte Carlo (MLMC) Method}
\label{sec:MLMC)}
Let $\mathbf{X}$ be a stochastic process and $g: \rset ^{d} \rightarrow \rset$ a  scalar observable. 
Let us assume that we want to approximate $\expt{g(\mathbf{X}(T))}$, but instead of sampling directly from $\mathbf{X}(T)$, we sample from $\mathbf{Z}_{\Delta t}(T)$, which are rdvs generated by an approximate method with step size $\Delta t$.  
Let us also assume  that the variates $\mathbf{Z}_{\Delta t}(T)$ are generated with an algorithm with weak order, $\Ordo{\Delta t}$, \ie, $\expt{g(\mathbf{X}(T))- g(\mathbf{Z}_{\Delta t}(T) )}= \Ordo{\Delta t}$.\footnote{We  refer to \cite{li2007analysis} for the underlying assumptions and proofs of this statement, in the context of the TL scheme.}
 
Let $\mu_{M}$ be the standard Monte Carlo estimator  of $\expt{g(\mathbf{Z}_{\Delta t}(T))}$ defined by
\begin{equation*}
\mu_{M} :=\frac{1}{M}\sum_{m=1}^{M} g(\mathbf{Z}_{\Delta t,[m]}(T))\COMMA
\end{equation*}
where $\{\mathbf{Z}_{\Delta t,[m]}(T)\}_{m=1}^M$ are independent and distributed as $\mathbf{Z}_{\Delta t}(T)$. 

We define the global error of the MC estimator as  $ \left(\expt{\left(\expt{g(\mathbf{X}(T))} - \mu_{M}\right)^2}\right)^{\frac{1}{2}}$. Then, we write the following error decomposition
\begin{equation*}
\expt{\left(\expt{g(\mathbf{X}(T))} - \mu_{M}\right)^2} = \underset{ \text{squared bias}}{\underbrace{\left( \expt{g(\mathbf{X}(T))- g(\mathbf{Z}_{\Delta t}(T) )} \right)^2 }}+ \underset{\text{Variance}}{\underbrace{\left(\expt{g(\mathbf{Z}_{\Delta t}(T) )} - \mu_{M} \right)^2}} \PERIOD
\end{equation*}
To achieve the desired accuracy, $\text{TOL}$, it is sufficient to take  $\Delta t=\Ordo{\text{TOL}}$ so that the bias  is $\Ordo{\text{TOL}}$ and  impose $M=\Ordo{\text{TOL}^{-2}}$ so that the variance is $\Ordo{\text{TOL}}$ \cite{duffie1995efficient}.  
As a consequence, the expected total computational work is $\Ordo{\text{TOL}^{-3}}$.

The MLMC estimator, introduced by Giles \cite{giles2008multilevel} (see also \cite{kebaier2005statistical} for the two-level construction), allows us to reduce the total computational work  up to  $\Ordo{	\text{TOL}^{-2-	\max\left(0, \frac{\gamma-\beta}{\alpha}\right)} \log\left(\text{TOL}\right)^{2 \times \mathbf{1}_{\{\beta=\gamma\}}}}$, where  $(\alpha,\beta,\gamma)$ are weak, strong, and work rates, respectively (see Theorem \ref{thm:MLMC_comlexity} for more details).  The basic idea of MLMC is to generate,  and couple in a clever manner, paths with different step sizes.
We can construct the MLMC estimator as follows: consider a hierarchy of nested meshes of the time interval $[0,T]$, indexed by $\ell=0, 1,\dots, L$. We denote by $\Delta t_{0}$ the step size used at level $\ell=0$. The size of the subsequent time steps for levels $\ell \geq 1$ is given by $\Delta t_{\ell}=K^{-\ell} \Delta t_{0}$, where $K{>}1$ is a given integer constant. In this work, we take $K = 2$. Furthermore, we denote by $M_{\ell}$ the number of samples per level in the MLMC estimator. To simplify the notation,  hereafter $\mathbf{Z}_{\ell}$ denotes the approximate process generated using a step size of $\Delta t_{\ell}$.

Consider now the following telescoping decomposition of $\expt{g(\mathbf{Z}_{L}(T))}$
\begin{align}\label{MLMC1}
\expt{g(\mathbf{Z}_{L}(T))}&= \expt{g(\mathbf{Z}_{0}(T))}+ \sum_{\ell=1}^{L} \expt{g(\mathbf{Z}_{\ell}(T))- g(\mathbf{Z}_{\ell-1}(T))} \\
&\: \text{Var}[g(\mathbf{Z}_{0}(T))]  \gg   \text{Var}[g(\mathbf{Z}_{\ell}(T))- g(\mathbf{Z}_{\ell-1}(T))] \searrow \: \text{as} \: \ell \nearrow \nonumber \\ 
	&\quad M_0   \: \quad  \quad \quad   \quad \gg M_{\ell} \searrow \: \text{as} \: \ell \nearrow \PERIOD \nonumber 
\end{align} 
Then, by defining 
\begin{align}\label{eq:MLMC_MC_estimators}
\begin{cases} 
\hat{Q}_{0}:= \frac{1}{M_{0}} \sum\limits_{m_{0}=1}^{M_{0}} g(\mathbf{Z}_{0,[m_{0}]}(T)) \\ 
\hat{Q}_{\ell}:= \frac{1}{M_{\ell}} \sum\limits_{m_{\ell}=1}^{M_{\ell}}  \LP g(\mathbf{Z}_{\ell,[m_{\ell}]}(T))-g(\mathbf{Z}_{\ell-1,[m_{\ell}]}(T)) \RP \COMMA \\
\end{cases}
\end{align}
we arrive at the unbiased MLMC estimator, $\hat{Q}$, of  $\expt{g(\mathbf{Z}_{L}(T))}$
\begin{equation}\label{eq:MLMC_estimator}
\hat{Q}:= \sum\limits_{\ell=0}^{L} \hat{Q}_{\ell}\PERIOD
\end{equation}
We note that the key point here is that both $\mathbf{Z}_{\ell,[m_{\ell}]}(T)$ and $\mathbf{Z}_{\ell-1,[m_{\ell}]}(T)$ are sampled using different time discretizations but with the same generated randomness. 

Theorem \ref{thm:MLMC_comlexity} from \cite{cliffe2011multilevel} states the computational complexity of the MLMC estimator for different scenarios:
\begin{theorem}[MLMC complexity]\label{thm:MLMC_comlexity}
Let $g:=g\left(\mathbf{X}\right)$ denote a rdv, and let $g_{\ell}:=g\left(\mathbf{Z}_{\ell}\right)$ denote the corresponding level $\ell$ numerical approximation. If there exist independent estimators $\hat{Q}_{\ell}$ based on $M_{\ell}$ Monte Carlo samples, each with expected cost $W_{\ell}$ and variance $V_{\ell}$, and positive constants $\alpha$ (weak convergence rate), $\beta$ (strong convergence rate), $\gamma$ (work rate), $c_1$, $c_2$, $c_3$ such that $\alpha \ge \min(\beta,\gamma)$ and
\begin{enumerate}
\item[i)] $\abs{\expt{g_{\ell}-g}}\le c_1 2^{-\alpha \ell}$
\item[ii)] $\expt{\hat{Q}_{\ell}}=\begin{cases}
               \expt{g_{0}},\quad \ell=0\\
               \expt{g_{\ell}-g_{\ell-1}},\quad \ell>0
            \end{cases}$
\item[iii)] $V_{\ell}:= \text{Var}\left[g_{\ell}-g_{\ell-1}\right] \le c_2 2^{-\beta \ell}$
\item[iv)]  $W_{\ell} \le c_3 2^{\gamma \ell}$,
\end{enumerate}
then there exists a positive constant $c_4$ such that for any $\text{TOL}<e^{-1}$, there are
values $L$ and $M_{\ell}$ for which the multilevel estimator
\begin{equation*}
\hat{Q}=\sum_{\ell=0}^L \hat{Q}_{\ell},
\end{equation*}
has a mean-square-error with bound 
\begin{align*}
\expt{\left(\hat{Q}-\expt{g}\right)^2} < \text{TOL}^2,
\end{align*}
with  a computational complexity $W$ with bound
\begin{align}
\expt{W}=\begin{cases}
               c_4 \text{TOL}^{-2},\quad  &\beta>\gamma, \nonumber \\
               c_4 \text{TOL}^{-2} \left(\log(\text{TOL})\right)^2,\quad  &\beta=\gamma, \nonumber \\
               c_4 \text{TOL}^{-2-\frac{\gamma-\beta}{\alpha}},\quad  &\beta<\gamma.
            \end{cases}
\end{align}
\end{theorem}
We emphasize that Theorem \ref{thm:MLMC_comlexity} still applies to our approach, proposed in Section \ref{sec:About the choice of the new measure and the importance sampling algorithm},  since  we only  modify the way we sample coupled paths in this context, by combining the standard way of coupling two tau-leap paths with  our IS algorithm. Our  proposed IS technique  does not change the weak rate but improves   the  strong convergence rate, $\beta$,  thus leading to an improvement of the MLMC complexity rate, to reach the optimal rate.

%% file: kurt_issue.tex
\subsection{The High-Kurtosis Phenomenon}\label{sec:The large kurtosis issue}
Let $g$ denote a rdv, and let $g_{\ell}$ denote the corresponding level $\ell$ numerical approximation. We also define $Y_{\ell}:= g_{\ell}-g_{\ell-1}$. The standard deviation of the sample variance for the rdv $Y_{\ell}$ is given by 
\begin{equation}\label{eq:approx_var}
\sigma_{\mathcal{S}^2(Y_{\ell})} =\frac{\text{Var}[Y_{\ell}]}{\sqrt{M}}  \sqrt{(\kappa-1)+\frac{2}{M-1}},
\end{equation}
where the kurtosis $\kappa = \frac{\expt{\left(Y_{\ell}-\expt{Y_{\ell}}\right)^4}}{\left(\text{Var}\left[Y_{\ell}\right]\right)^2}$. 

For the setting of the MLMC method,  accurate estimates of $V_{\ell}=\text{Var}[Y_{\ell}]$ are required since the optimal number of samples per level, $M^{
\ast}_{\ell} $, for the MLMC estimator is given by (see \cite{giles2015multilevel} for more details)
\begin{equation}\label{eq:optimal_number_samples}
M^{
\ast}_{\ell} = \left[2 \text{TOL}^{-2} \sqrt{V_{\ell} W^{-1}_{\ell}  } \sum_{\ell=L_0}^L \sqrt{V_{\ell} W_{\ell} }  \right],
\end{equation}
where $[x]:=\text{ceil}(x)$,  $W_{\ell}$ is the cost per sample path,  $\text{TOL}$ is the accuracy of the MLMC estimator, and  $L_0$ \footnote{We set $L_0=0$ unless otherwise stated. In our numerical experiments, we select $L_0$ such that $\var{g_{L_0+1}{-}g_{L_0}} \ll \var{g_{L_0}}$, in order to  ensure the stability of the variance of the coupled paths of our MLMC estimator.} and $L$ are the coarsest and the deepest levels of the MLMC estimator, respectively.

The high kurtosis makes it challenging to estimate  $V_{\ell}$ accurately, since $\Ordo{\kappa}$ samples are required to obtain a reasonable estimate of the variance (see \ref{eq:approx_var}). Two possible consequences of the high kurtosis may occur, and deteriorate the robustness and the performance of the MLMC estimator
\begin{itemize}
\item The sample variance, $V_{\ell}$, is an under-estimate. The effect is that the required confidence interval semi-length is not faithfully attained,  due to $\sigma_{\mathcal{S}^2(Y_{\ell})}$ given by \eqref{eq:approx_var}.
\item The sample variance, $V_{\ell}$, is an over-estimate. In this case, too many sample paths are generated, and the algorithm takes substantially more time to run.
\end{itemize}
Several studies  \cite{giles2015multilevel,gou2016estimating,moraes2016multilevel,hammouda2017multilevel,lester2018robustly,bayer2020numerical} discussed the issue of high kurtosis when using MLMC, for different applications. In the context of SRNs, there are mainly two causes of  the high-kurtosis phenomenon: i) \textit{Catastrophic coupling} or ii) 
\textit{Catastrophic decoupling}. In the following subsections, we explain these two causes.
\subsubsection{Catastrophic Coupling}\label{sec:Catastrophic coupling}
The high-kurtosis phenomenon, in this case, is caused by  \textit{catastrophic coupling} (see Section 1.7 of \cite{moraes2016multilevel}), which is a characteristic of pure jump processes that motivates this work.  When using the MLMC estimator in this context, the following issue is usually encountered: When $\ell$ (MLMC level) becomes large, due to the used  coupling strategy (see Section \ref{sec:Idea}), $Y_{\ell}$ is different from zero only in a tiny  proportion of the simulated coupled paths (see Figures \ref{fig:catastrophic_coupling_illustration_decay_X_1}, \ref{fig:catastrophic_coupling_illustration_example2_X_1} \ref{fig:catastrophic_coupling_illustration_example4_X_3}). This behavior is one of the leading causes  of the high-kurtosis phenomenon (see Figures \ref{fig:MLMC_decay_X_0_10}, \ref{fig:MLMC_exp2_X_1} and \ref{fig:MLMC_exp4_X_3}), resulting in inaccurate estimates of the sample variance (see \eqref{eq:approx_var}). 

As an illustration of  \textit{catastrophic coupling}, consider an example  when $g$ takes values in $\{0,1\}$, and let $g_{\ell}$ denote the corresponding level $\ell$ numerical approximation in the MLMC estimator. In this case, we have
\begin{align}\label{eq:exp_high_kurt}
Y_{\ell}=g_{\ell}-g_{\ell-1}=\begin{cases}
               1,\quad  \text{with probability} \quad p_{\ell}\\
               -1,\quad \text{with probability}\quad q_{\ell}\\
               0,\quad  \text{with probability} \quad 1-p_{\ell}-q_{\ell}.
            \end{cases}
\end{align}
Observe that this example is a true illustration of the SRNs that we consider in this work. For instance, by observing the histograms in Figures  \ref{fig:catastrophic_coupling_illustration_decay_X_1}, \ref{fig:catastrophic_coupling_illustration_example2_X_1} and \ref{fig:catastrophic_coupling_illustration_example4_X_3}, we can check that we usually encounter the situation manifested by   \eqref{eq:exp_high_kurt}, with $p_{\ell},q_{\ell} \ll 1$, and  $\text{Prob}\left(g_{\ell}-g_{\ell-1}=0\right) \rightarrow 1$ as $\ell$ increases. 

If $p_{\ell},q_{\ell} \ll 1$, then $\expt{Y_{\ell}} \approx0$  and $\kappa_{\ell} \approx (p_{\ell}+q_{\ell})^{-1}\gg 1$. Therefore, many samples are required for an accurate estimate of $V_{\ell}=\text{Var}[Y_{\ell}]$, since using \eqref{eq:approx_var}, we need $M_{\ell} \gg \kappa_{\ell} \overset{\ell \rightarrow \infty} {\longrightarrow}\infty  $; otherwise, we may get all samples $Y_{\ell}=0$, which  gives an estimated variance of zero. Furthermore, the kurtosis  becomes worse as $\ell \rightarrow \infty$ since $p_{\ell},q_{\ell} \rightarrow0$ due to weak convergence.

\subsubsection{Catastrophic Decoupling}\label{sec:Catastrophic decoupling}
The high-kurtosis phenomenon can also occur because of \textit{catastrophic decoupling}, as explained in \cite{lester2016extending} and observed in \cite{lester2018robustly}.  \textit{Catastrophic decoupling} occurs when  the terminal values of the sample paths of both coarse and fine levels   become
very different from each other. In fact, due to the  SPM coupling  strategy (see Section \ref{sec:Idea}), all reactions  start immediately in the fine level and not  in the coarse level, since reactions cannot happen  until the reaction propensities are updated.  We note that this issue  becomes more severe when dealing with large scales of species count.

We emphasize that we do not treat \textit{catastrophic decoupling} with our novel proposed method, but rather we address the case of  \textit{catastrophic coupling}. Nonetheless,  \textit{catastrophic decoupling} can be addressed  by using a different coupling,  such as CPM coupling \cite{lester2018robustly}. In a future work, to address the issue of  \textit{catastrophic decoupling}, we intend to explore the possibility of introducing a new IS scheme for MLMC based on SPM coupling.

\begin{remark}
As proposed in  \cite{gillespie2000chemical},  SRNs paths can be approximated  using the chemical Langevin equation (CLE), which is only valid when  the  expected number of occurrences of each reaction channel $R_j$
in $[t, t+\Delta t)$ is much larger than $1$, \ie,
\begin{equation}\label{eq:CLE_cond}
a_j(\mathbf{x}_t) \Delta t \gg1,\quad \forall j \in \{0,1\dots,J\}.
\end{equation}
Assumption \eqref{eq:CLE_cond}, implicitly implies that the system has large molecular population numbers.  In this work, we do not impose this restriction on the examples  we consider.  Moreover, such an assumption does not hold in our setting and more precisely in the deepest level of MLMC estimator ($\Delta t$  very small). 
\end{remark}


%% file: const_imp_sampling_algorithm.tex
\section{Motivation}\label{sec:Idea}
\subsection{Characterization of the Original Coupling Measure}
Let us use the notations of Section \ref{sec:MLMC)}, and denote  $g_{\ell}:=g\left(\mathbf{Z}_{\ell}(T)\right)$. Then, we can rewrite \eqref{MLMC1} as
\begin{equation}\label{eq:MLMC}
\expt{g_{L}}=\sum_{\ell=1}^{L} \expt{g_\ell-g_{\ell-1}}+\expt{g_0},
\end{equation}
where each term in \eqref{eq:MLMC} can be written as
\begin{equation}\label{eq:expect_under_P}
\expt{g_0}=\int g_0 d \mathbb{P}_0, \quad  \expt{g_\ell- g_{\ell-1}}=\int (g_\ell-g_{\ell-1}) d \mathbb{P}_\ell,
\end{equation}
where $\mathbb{P}_\ell$ is the coupling measure and  $\mathbb{P}_0$ is the single level measure.

To characterize the original coupling measure $\mathbb{P}_\ell$ in the  context of SRNs, we  define the pure jump process $X$ by the Kurtz representation, as in (\ref{eq:exact_process}). For the sake of simplicity, let  us consider $X$ to be one-dimensional (only one species), only one reaction  $(J=1)$ (in this case we denote the state change scalar by $\nu_1$; see \eqref{reaction_channel}), and $g(x)=x,\: x \in \rset$. We denote $\overline{X}_{\ell-1}$, $\overline{X}_{\ell}$ the two TL approximations of the true process $X$ based on two consecutive grid levels $(\ell-1,\ell)$ and recall that $\Delta t_{\ell-1}=2\Delta t_{\ell}$ (equivalently, we denote by $N_{\ell}-1$ and $N_{\ell}$ the number of times steps used at levels $\ell-1$ and $\ell$, respectively). Let $0 \le n \le N_{\ell-1}-1$. If we consider two consecutive time-mesh points for  $\overline{X}_{\ell-1}$, $\{t_n, t_{n+1}\}$, and three consecutive time-mesh points for $\overline{X}_{\ell}$, $\{t_n, t_n+\Delta t_{\ell}, t_{n+1}\}$, then we have
\begin{small}
\begin{align}\label{eq:coupling_decomp}
\overline{X}_{\ell-1}(t_{n+1})&=\overline{X}_{\ell-1}(t_{n})+ \nu_1 \mathcal{Y}_{1,n}\left(a\left(\overline{X}_{\ell-1}(t_{n})\right) \Delta t_{\ell-1}\right)\nonumber \\
\overline{X}_{\ell}(t_{n}+\Delta t_{\ell})&=\overline{X}_{\ell}(t_{n})+ \nu_1 \mathcal{Q}_{1,n}\left(a\left(\overline{X}_{\ell}(t_{n})\right) \Delta t_{\ell}\right)\nonumber \\
\overline{X}_{\ell}(t_{n+1})&=\overline{X}_{\ell}(t_{n}+\Delta t_{\ell})+ \nu_1 \mathcal{R}_{1,n}\left(a\left(\overline{X}_{\ell}(t_{n}+\Delta t_{\ell})\right) \Delta t_{\ell} \right),
\end{align}
\end{small}
where $\mathcal{Y}_{1,n}, \mathcal{Q}_{1,n}, \mathcal{R}_{1,n}$ are  conditionally independent Poisson rdvs.

To couple the $\overline{X}_{\ell-1}$ and $\overline{X}_{\ell}$ processes, we first decompose $\mathcal{Y}_{1,n}\left(a\left(\overline{X}_{\ell-1}(t_{n})\right) \Delta t_{\ell-1}\right)$ as the sum of two conditionally independent Poisson rdvs, $\mathcal{P}_{1,n}\left(a\left(\overline{X}_{\ell-1}(t_{n})\right) \Delta t_{\ell}\right)+\mathcal{P}_{2,n}\left(a\left(\overline{X}_{\ell-1}(t_{n})\right) \Delta t_{\ell}\right)$. Then, by applying this decomposition in \eqref{eq:coupling_decomp}, we obtain
\begin{small}
\begin{align*}
\overline{X}_{\ell-1}(t_{n+1})&=\overline{X}_{\ell-1}(t_{n})+ \nu_1 \mathcal{P}_{1,n}\left(a\left(\overline{X}_{\ell-1}(t_{n})\right) \Delta t_{\ell}\right) + \nu_1 \mathcal{P}_{2,n}\left(a \left(\overline{X}_{\ell-1}(t_{n})\right) \Delta t_{\ell}\right)\nonumber \\
\overline{X}_{\ell}(t_{n+1})&=\overline{X}_{\ell}(t_{n})+ \nu_1 \mathcal{Q}_{1,n}\left(a\left(\overline{X}_{\ell}(t_{n})\right) \Delta t_{\ell}\right) + \nu_1 \mathcal{R}_{1,n}\left(a\left(\overline{X}_{\ell}(t_{n}+\Delta t_{\ell})\right) \Delta t_{\ell}\right).
\end{align*}
\end{small}
Furthermore, by using the same reasoning of coupling strategy as in \cite{Anderson2012} , we can show that for the first  time interval $[t_n, t_n+\Delta t_{\ell}]$, we have 
\begin{small}
\begin{align}\label{eq:coupled_levels_first_interval}
\overline{X}_{\ell-1}(t_n+\Delta t_{\ell})&= \overline{X}_{\ell-1}(t_n)+ \left ( \mathcal{P}^\prime_{n}\left(m^1_{\ell,n} \Delta t_{\ell}\right)+\mathcal{P}^{''}_{n} \left(\left(a\left(\overline{X}_{\ell-1}(t_n)\right)-m^1_{\ell,n}\right)\Delta t_{\ell}\right) \right ) \nu_1 \nonumber \\
\overline{X}_{\ell}(t_n+\Delta t_{\ell})&= \overline{X}_{\ell}(t_n)+ \left ( \mathcal{P}^\prime_{n}\left(m^1_{\ell,n} \Delta t_{\ell}\right)+\mathcal{P}^{'''}_{n} \left(\left(a\left(\overline{X}_{\ell}(t_n)\right)-m^1_{\ell,n}\right)  \Delta t_{\ell} \right)\right )\nu_1, 
\end{align}
\end{small}
where $m^1_{\ell,n}=\min \left(a\left(\overline{X}_{\ell}(t_n)\right),a\left(\overline{X}_{\ell-1}(t_n)\right)\right)$, and  $\mathcal{P}^\prime_{n}, \mathcal{P}^{''}_{n}, \mathcal{P}^{'''}_{n}$ are  conditionally independent Poisson rdvs.

For the time interval $[t_n+\Delta t_{\ell}, t_{n+1}]$, we have 
\begin{small}
\begin{align}\label{eq:coupled_levels_second_interval}
 \overline{X}_{\ell-1}(t_{n+1})&= \overline{X}_{\ell-1}(t_n+\Delta t_{\ell})+ \left ( \mathcal{Q}^\prime_{n}\left(m^2_{\ell,n} \Delta t_{\ell}\right)+\mathcal{Q}^{''}_{n} \left(\left(a\left(\overline{X}_{\ell-1}(t_n)\right)-m^2_{\ell,n}\right)\Delta t_{\ell}\right) \right ) \nu_1 \nonumber \\
 \overline{X}_{\ell}(t_{n+1}) &= \overline{X}_{\ell}(t_n+\Delta t_{\ell})+ \left ( \mathcal{Q}^\prime_{n}\left(m^2_{\ell,n} \Delta t_{\ell}\right)+\mathcal{Q}^{'''}_{n} \left(\left(a\left(\overline{X}_{\ell}(t_n+\Delta t_{\ell})\right)-m^2_{\ell,n}\right)  \Delta t_{\ell} \right)\right )\nu_1, 
\end{align}\end{small}
where $m^2_{\ell,n}=\min \left(a\left(\overline{X}_{\ell}(t_n+\Delta t_{\ell})\right),a\left(\overline{X}_{\ell-1}(t_n)\right)\right)$, and  $\mathcal{Q}^\prime_{n}, \mathcal{Q}^{''}_{n}, \mathcal{Q}^{'''}_{n}$ are  conditionally independent Poisson rdvs.

\eqref{eq:coupled_levels_first_interval} and \eqref{eq:coupled_levels_second_interval} imply that 
\begin{small}
\begin{align}\label{eq: local_error_expression}
\overline{X}_{\ell}(t_{n+1})-\overline{X}_{\ell-1}(t_{n+1})&=\overline{X}_{\ell}(t_{n})-\overline{X}_{\ell-1}(t_{n})\nonumber\\
&+\nu_1 \left(  \mathcal{P}^{'''}_n  \left( \Delta a^1_{\ell-1,n} \Delta t_{\ell}\right) \mathbf{1}_{\Delta a^1_{\ell-1,n}>0} -  \mathcal{P}^{''}_n  \left( -\Delta a^1_{\ell-1,n} \Delta t_{\ell}\right) \mathbf{1}_{\Delta a^1_{\ell-1,n}<0} \right)\nonumber\\
&+\nu_1 \left(  \mathcal{Q}^{'''}_n  \left( \Delta a^2_{\ell-1,n} \Delta t_{\ell}\right) \mathbf{1}_{\Delta a^2_{\ell-1,n}>0} -  \mathcal{Q}^{''}_n  \left( -\Delta a^2_{\ell-1,n} \Delta t_{\ell}\right) \mathbf{1}_{\Delta a^2_{\ell-1,n}<0} \right)\COMMA
\end{align}
\end{small}
where $\Delta a^1_{\ell-1,n}=a\left(\overline{X}_{\ell}(t_{n})\right)-a\left(\overline{X}_{\ell-1}(t_{n})\right)$ and  $\Delta a^2_{\ell-1,n}=a\left(\overline{X}_{\ell}(t_{n}+\Delta t_{\ell})\right)-a\left(\overline{X}_{\ell-1}(t_{n})\right)$. 

In the following, we denote, for $0 \le n \le N_{\ell-1}-1$ (note that  $N_{\ell}=2 N_{\ell-1}$),
\begin{small}
\begin{align}\label{eq:abs_delta_a}
\begin{cases}
              \Delta a_{\ell,2 n}= |\Delta a^1_{\ell-1,n}|,\quad  \text{in} \quad  [t_n, t_n+\Delta t_{\ell}]. \\
                \Delta a_{\ell,2 n+1}=|\Delta a^2_{\ell-1,n}|,\quad  \text{in} \quad [t_n+\Delta t_{\ell}, t_{n+1}].
            \end{cases}
\end{align}
\end{small}
Note that in \eqref{eq: local_error_expression}, not only are  $\mathcal{P}^{''}_n, \mathcal{P}^{'''}_n, \mathcal{Q}^{''}_n, \mathcal{Q}^{'''}_n$  rdvs, but   $\Delta a_{\ell,2n}$ and $\Delta a_{\ell,2n+1}$  (defined in \eqref{eq:abs_delta_a}) are also   rdvs, because of their dependence on $\overline{X}_{\ell-1}(t_{n})$, $\overline{X}_{\ell}(t_{n})$, and $\overline{X}_{\ell}(t_n+\Delta t_{\ell})$. Therefore,  to derive some of the following formulas for analyzing our IS  algorithm, we need to consider a sigma-algebra, $\mathcal{F}_{n_{\ell}}$, with $0\le n_{\ell}\le N_{\ell}-1$, such that  $\Delta a_{\ell,n_{\ell}}$, conditioned on $\mathcal{F}_{n_{\ell}}$, is deterministic, \ie,   $\Delta a_{\ell,n_{\ell}}$ is measurable with respect to $\mathcal{F}_{n_{\ell}}$. This way, the only randomness being considered comes from the Poisson rdvs used for updating the states of $\overline{X}_{\ell}(t_{n+1})$ and $\overline{X}_{\ell-1}(t_{n+1})$. For  this purpose, we consider  for a fixed $n_{\ell}$, $\mathcal{F}_{n_{\ell}}$ as the sigma algebra
\begin{equation}\label{eq:conditioned_filtration}
\mathcal{F}_{n_{\ell}}:=\sigma\left( \left(\Delta a_{\ell,k}\right)_{k=0,\dots,n_{\ell}}  \right), \quad n_{\ell}=0,\dots,N_{\ell}-1.
\end{equation}
In what follows,  the terms $\{\Delta a_{\ell,n}\}_{n=0}^{N_{\ell}-1}$, defined in  \eqref{eq:abs_delta_a}, will be denoted for the multi-channel case, by  $\{\Delta a^j_{\ell,n}\}_{n=0}^{N_{\ell}-1}$, where $j \in \{1,\dots, J\}$   corresponds to the index of the reaction channel.

\subsection{Characterization of the Optimal Change of  Measure}
It is known that, the optimal change of measure, $\pi_{\ell}$,  the one that achieves the minimum variance,   satisfies
\begin{equation}\label{eq:expect_under_Pi}
d \pi_0 \propto \abs{g_0} d \mathbb{P}_0, \quad d \pi_\ell \propto \abs{g_\ell-g_{\ell-1}} d \mathbb{P}_\ell.
\end{equation}
Observe that the optimal measure, $\pi_\ell$,   removes the probability mass at  zero, where most of $\mathbb{P}_{\ell}$ is concentrated due to catastrophic coupling (explained in Section \ref{sec:Catastrophic coupling}). We emphasize that, in this work, we aim to  perform a change of measure with respect to  $\mathbb{P}_\ell$, while keeping  the single level measure  $ \mathbb{P}_0$ unchanged.

The minimum variance is given by 
 \begin{align}\label{eq:IS_optimal}
 \text{Var}_{\pi_{\ell}}\left[\left(g_\ell-g_{\ell-1}\right) \frac{d\mathbb{P}_\ell}{d\pi_{\ell}}\right]&= \left(E_{\mathbb{P}_{\ell}}\left[\abs{g_\ell-g_{\ell-1}}\right]\right)^2 \left(1 - \left(E_{\pi_{\ell}}\left[sgn \left(g_\ell-g_{\ell-1}\right)\right]\right)^2 \right)\nonumber\\
& =\left(E_{\mathbb{P}_{\ell}}\left[\abs{g_\ell-g_{\ell-1}}\right]\right)^2 \left( 1 - \left(\pi_{\ell}\left( g_\ell-g_{\ell-1}>0\right)- \pi_{\ell} \left(g_\ell-g_{\ell-1}<0\right)\right)^2\right),
 \end{align}
 where $\text{sgn}(.)$ is the sign function.
 
Interestingly,  using   Theorem 3.2 in   \cite{anderson2011error} in the context of the  explicit TL scheme for pure jump processes, we conclude that $E_{\mathbb{P}_{\ell}}\left[\abs{g_\ell-g_{\ell-1}}\right]=\Ordo{\Delta t_{\ell}}$ for any Lipschitz function $g$. Therefore, we clearly observe  that the optimal  IS improves the strong convergence rate, and hence leads to the optimal complexity rate of the MLMC estimator (see Theorem \ref{thm:MLMC_comlexity}). 

Unfortunately, it is unfeasible to sample from $\pi_{\ell}$; therefore, our goal in the following sections is to propose a practical IS algorithm with a sub-optimal change of measure, $\bar{\pi}_{\ell}$.

\section{Main Results}\label{sec:main results}
Our analysis  and theoretical estimates in  Section \ref{sec:About the choice of the new measure and the importance sampling algorithm}, and numerical experiments in Section \ref{sec:num_experiments}  show that
\begin{enumerate}
\item  For  $\ell=1,\dots,L$, and  $g_{\ell}:=g\left(\overline{\mathbf{X}}_{\ell}\right)$: the change of measure is performed at each time step, for  $0 \le n \le N_{\ell}-1$, by going forward in time, and is only applied when
\begin{small}
\begin{equation}\label{eq: imp_sampling_cond_general}
 \text{i)}\quad  j \in \mathcal{J}_1:= \{1\le j\le J; \quad g(\mathbf{X}+\boldsymbol{\nu}_j) \neq g(\mathbf{X})\} \quad \& \quad \text{ii)} \quad  \Delta a^j_{\ell,n} \neq 0 \quad \& \quad \text{iii)}\quad  \Delta g_{\ell} (t_n)=0
\end{equation}
where $\Delta g_{\ell}:=g_{\ell}-g_{\ell-1}$.
\end{small}
\item If  \eqref{eq: imp_sampling_cond_general} is fulfilled: instead of using $\Delta a^j_{\ell,n} \Delta t_{\ell}$  in  \eqref{eq:abs_delta_a},   we propose to use $\lambda^j_{\ell,n} \Delta t_{\ell}$, with   $\lambda^j_{\ell,n}$ is  given by
\begin{small}
\begin{equation}\label{eq:choice_c_l_step1}
\lambda^j_{\ell,n} =c_{\ell} \Delta a^j_{\ell,n}=\Delta t_{\ell}^{-\delta} \Delta a^j_{\ell,n} ,\quad 0<\delta<1 \COMMA
\end{equation}
\end{small}
where $\delta$ is a scale parameter in our  IS algorithm.
\item We show that  our proposed method (MLMC with IS) significantly reduces  the kurtosis at the deep levels of MLMC (small $\Delta t_{\ell}$) (see Theorem \ref{thm:kurtosis_high_dim} and the numerical experiments in Section \ref{sec:num_experiments}),  and also improves the strong convergence rate from $\beta=1$,  for the standard case (without IS), to $\beta=1+\delta$, where $0<\delta<1$ is a user-selected parameter in our IS algorithm (see Theorem  \ref{thm:variance_high_dim}, and the numerical experiments in Section \ref{sec:num_experiments}). Due to  Theorem \ref{thm:MLMC_comlexity},   and given  a pre-selected tolerance, $\text{TOL}$, this results in an improvement of the complexity from  $\Ordo{\text{TOL}^{-2} \log(\text{TOL})^2}$, in the standard case, to the optimal complexity, \ie,  $\Ordo{\text{TOL}^{-2}}$.  These improvements come with a negligible additional cost since   we show in Section \ref{sec:Cost analysis} that $W^{\text{without IS}}_{\ell,\text{sample}} \approx W^{\text{with IS}}_{\ell,\text{sample}}$ ($W_{\ell,\text{sample}}$ denotes the average cost of simulating coupled  MLMC paths at level $\ell$). We show a summary of these results in Table \ref{table:Main results of our work.}; see Sections \ref{sec:About the choice of the new measure and the importance sampling algorithm} and \ref{sec:num_experiments} for more details.
\FloatBarrier
\begin{small}
\begin{table}[h!]
	\centering
	    \begin{tabular}{c|c|c}
\toprule[1.5pt]
 \multicolumn{1}{c|}{Quantity of Interest} & \multicolumn{1}{c|}{MLMC Without IS (standard case)} & \multicolumn{1}{c}{MLMC With IS ($0<\delta<1$)} \\
\midrule	      
		   $\kappa_{\ell}$ &  $\Ordo{\Delta t_{\ell}^{-1}} $  &  $\Ordo{\Delta t_{\ell}^{\delta-1}}$  \\
		   \hline
	$V_{\ell}$ & $\Ordo{\Delta t_{\ell}}$ &  $\Ordo{\Delta t_{\ell}^{1+\delta}}$ \\
	 	\hline
		  $W_{\ell,\text{sample}}$  & $\approx 2 \times J \times C_p \times \Delta t_{\ell}^{-1}$ &  $\approx 2 \times J \times C_p \times \Delta t_{\ell}^{-1}$  \\
	\hline
		  $\text{Work}_{\text{MLMC}}$ & $\Ordo{\text{TOL}^{-2} \log\left(\text{TOL}\right)^2}$ &  $\Ordo{\text{TOL}^{-2}}$ \\
		 
		\bottomrule[1.25pt]
	\end{tabular}
	\caption{Main results for the comparison of MLMC combined with  our IS algorithm, and  standard MLMC. $\kappa_{\ell}$ denotes the kurtosis of the coupled MLMC paths at level $\ell$.  $V_{\ell}$ denotes the variance of the coupled MLMC paths at level $\ell$. $C_p$ is the  cost of generating one Poisson rdv.}
	\label{table:Main results of our work.}
\end{table}
\end{small}
\FloatBarrier
\end{enumerate}

\section{Construction of the  IS Measure and   Convergence Estimates}\label{sec:About the choice of the new measure and the importance sampling algorithm}

\subsection{Construction of the  IS Measure: The One-Dimensional Case}\label{sec:One dimensional case}
We start with the one-dimensional case (only one species), where the number of reactions is $(J=1)$.  Instead of using $\Delta a_{\ell,n} \Delta t_{\ell}$ as the rate parameter of the Poisson rdvs used in each time step to update the states of the coupled paths (\eqref{eq:coupled_levels_first_interval}, \eqref{eq:coupled_levels_second_interval}) where $\Delta a_{\ell,n}$ is given by \eqref{eq:abs_delta_a},  we suggest using $\lambda_{\ell,_n} \Delta t_{\ell}$, with the parameter  $\lambda_{\ell,n}$, which will be determined given some constraints that we impose to ensure that our change of measure is i) reducing the kurtosis of the MLMC estimator at the deep levels, ii) reducing the variance of the MLMC levels and increasing the strong convergence rate. In the following, we denote   $g_{\ell}:=g\left(\overline{\mathbf{X}}_{\ell}\right)$, and $\Delta g_{\ell}:=g_{\ell}-g_{\ell-1}$, for  $\ell=1,\dots,L$.

The change of measure is performed at each time step by going forward in time, and is only applied when it is needed, \ie, we impose the following   condition for applying the change of measure 
\begin{align}\label{eq: imp_sampling_cond}
\Delta a_{\ell,n} \neq 0 \quad \& \quad \Delta g_{\ell} (t_n)=0,\quad  0 \le n \le N_{\ell}-1, \quad \ell=1,\dots,L.
\end{align}
Condition \eqref{eq: imp_sampling_cond} is motivated by the fact that i) we need to change the measure only in cases where the coupled paths at the $n$th time step are equal and ii) for cases where the rates of the Poisson rdvs are non zero, so we do not have the issue of the likelihood being equal to zero.

Whenever \eqref{eq: imp_sampling_cond} holds, the change of measure is given by changing the rate of the Poisson rdvs (see \eqref{eq: local_error_expression} and \eqref{eq:abs_delta_a})  in the tau-leap approximation from $\Delta a_{\ell,n}\Delta t_{\ell}$ to  $\lambda_{\ell,n}\Delta t_{\ell}$. Hence, the conditional likelihood is then given by the ratio between the probability mass functions of two Poisson rdvs with rates $\Delta a_{\ell,n}\Delta t_{\ell}$ and  $\lambda_{\ell,n}\Delta t_{\ell}$. Through a simple computation, this leads to
\begin{align}\label{eq:lik_each_step}
L_{\ell,n}&=\frac{e^{-\Delta a_{\ell,n} \Delta t_{\ell}}}{e^{-\lambda_{\ell,n} \Delta t_{\ell}}} \left(\frac{\Delta a_{\ell,n}}{\lambda_{\ell,i}}\right)^{k_n}, \quad n \in \mathcal{I}^s_{\ell}   \nonumber\\
&=e^{-\Delta t_{\ell}(\Delta a_{\ell,n}-\lambda_{\ell,n})} \left(\frac{\Delta a_{\ell,n}}{\lambda_{\ell,n}}\right)^{k_n}\COMMA
\end{align}
where $k_n$ is the number of jumps that occurs at the $n$th time step  where we apply the change of measure, and  $\mathcal{I}_{\ell}^{s}$ is the random  set including the time steps at level $\ell$  where we simulate the Poisson rdvs under the new measure.

Thus, across one path, the  likelihood ratio is given by
\begin{equation}\label{eq: lik_level_ell_def}
L_{\ell}=e^{-\Delta t_{\ell} \sum_{n \in \mathcal{I}_{\ell}^{s}}(\Delta a_{\ell,n}-\lambda_{\ell,n})} \prod_{n  \in \mathcal{I}_{\ell}^{s}} \left(\frac{\Delta a_{\ell,n}}{\lambda_{\ell,n}}\right)^{k_n}\PERIOD
\end{equation}
Furthermore, if we impose that $\lambda_{\ell,n}=c_{\ell} \Delta a_{\ell,n}$ then we obtain
\begin{equation}\label{eq: lik_level_ell}
L_{\ell}=\left(e^{(c_{\ell} -1) \Delta t_{\ell} \sum_{n \in \mathcal{S}}\Delta a_{\ell,n}}\right) \left(c_{\ell}^{-\sum_{n \in \mathcal{S}} k_n}\right).
\end{equation}
We note that imposing $\lambda_{\ell,n}=c_{\ell} \Delta a_{\ell,n}$ can be motivated by the fact that we want to keep the same physical structure of the rate of the Poisson process driving the state changes, \ie, depending on $\Delta a_{\ell,n}$. However, we try to introduce a scaled constant $c_{\ell}$ that depends on $\Delta t_{\ell}$ so that we reduce the probability of having $\Delta g_{\ell}(T)=0$, under the new measure. A reasonable choice of  $c_{\ell}$  is given by  
\begin{align}\label{eq:choice_c_l_step1}
c_{\ell}=\Delta t_{\ell}^{-\delta}\COMMA
\end{align}
where $\delta>0$ is the scale parameter to be determined. Note that the case $\delta=0$ is similar to the case of using the old measure in all time steps. 
\begin{remark}\label{rem:proba_0event_delta}
Observe that for  $\delta \in (0,1)$ and $g(x)=x,\: x \in \rset$, we have $\bar{\pi}_{\ell}\left(\abs{\Delta \overline{X}_{\ell}(T)}=0 \right)$ is still approaching $1$  as  $\Delta t_{\ell}$ decreases,  but compared to the initial situation (without IS), we decrease the rate of convergence with respect to $\Delta t_{\ell}$ (compare Figure \ref{fig:catastrophic_coupling_illustration_decay_X_1}, for the case without  IS, and Figure \ref{fig:catastrophic_coupling_illustration_decay_X_1_imp_sampling} for the case with IS  with $\delta=\frac{3}{4}$). 
\end{remark}

\subsection{Construction of the  IS Measure:: The Multi-Channels and High Dimensional States Case}\label{sec:Generalization of our method to multi-channels and high dimensional states}
Extending  our method to a higher dimension in the number of reaction channels, $J$,   and in the state vector $\mathbf{X}$ is straightforward with slight modifications. We first define the set $\mathcal{J}_1$ as 
\begin{align*}
\mathcal{J}_1=\{1\le j\le J; \quad g(\mathbf{X}+\boldsymbol{\nu}_j) \neq g(\mathbf{X})\}.
\end{align*}
In the multi-channel case, we are only interested in changing the measure for reactions whose stoichiometric vector, $\boldsymbol{\nu}_j$,  changes  the state of the quantity of interest, \ie, for reactions with index $j \in \mathcal{J}_1$. In Algorithm \ref{alg:exp_TL_imp_sampling}, we summarize our methodology for simulating two coupled explicit TL paths with IS.
\begin{small}
\begin{algorithm}[h!]
\caption{Simulates two coupled explicit TL paths with IS, and computes the likelihood ratio. }
\label{alg:exp_TL_imp_sampling}
\begin{algorithmic}[1]
	\State Fix $ \Delta t_{\ell} > 0$  and set $\Delta t_{\ell-1} = 2 \times \Delta t_{\ell}$. 
	\State Set $\mathbf{Z}_{\ell}(0) = \mathbf{Z}_{\ell-1}(0) = \mathbf{x}_{0}$, $t_{\ell} = t_{\ell-1}=0$, $n=0$.
	\State  Set $c_{\ell}=\Delta t_{\ell}^{-\delta}$, $\delta \in (0,1)$
		\While{$t_{\ell} < T$}
		\State $n=n+1$
		\For {$j{=}1$ \textbf{to} $J$}
         \If {$\left( a_{j}(\mathbf{Z}_{\ell}(t_{\ell})) \neq a_{j}(\mathbf{Z}_{\ell-1}(t_{\ell-1})) \quad  \& \quad g(\mathbf{Z}_{\ell}(t_{\ell}))=g(\mathbf{Z}_{\ell-1}(t_{\ell})) \quad  \& \quad  j \in \mathcal{J}_1\right)$}
	\State $A_{3(j-1)+1}= \min\left(a_{j}(\mathbf{Z}_{\ell}(t_{\ell})), a_{j}(\mathbf{Z}_{\ell-1}(t_{\ell-1}))\right)$
	\State $A_{3(j-1)+2}=  c_{\ell} \left(a_{j}\left(\mathbf{Z}_{\ell}(t_{\ell})\right)-A_{3(j-1)+1}\right) $
	\State $A_{3(j-1)+3}=c_{\ell} \left( a_{j}\left(\mathbf{Z}_{\ell-1}(t_{\ell-1})\right)-A_{3(j-1)+1}\right) $
	\State Compute $L_{\ell,n}$ using \eqref{eq:lik_each_step} and update  the likelihood terms, $L_{\ell}^j$ and  $L_{\ell}$   using  \eqref{eq:lik_fact_each_reaction} and \eqref{eq:lik_high_dim_total}.
	\Else
		\State $A_{3(j-1)+1}= \min \left(a_{j}(\mathbf{Z}_{\ell}(t_{\ell})),  a_{j}(\mathbf{Z}_{\ell-1}(t_{\ell-1})) \right)$
	\State $A_{3(j-1)+2}= a_{j}(\mathbf{Z}_{\ell}((t_{\ell}))-A_{3(j-1)+1} $
	\State $A_{3(j-1)+3}=a_{j}(\mathbf{Z}_{\ell-1}(t_{\ell-1}))-A_{3(j-1)+1} $
	  \EndIf
	\State $\Lambda_{3(j-1)+1}= \: \text{Poisson} \: (A_{3(j-1)+1} \Delta t_{\ell}) $
	\State $\Lambda_{3(j-1)+2}= \: \text{Poisson} \: (A_{3(j-1)+2} \Delta t_{\ell})$
	\State $\Lambda_{3(j-1)+3}= \: \text{Poisson} \: (A_{3(j-1)+3} \Delta t_{\ell}) $
	\EndFor
	\State State updating 
	\begin{itemize}
	\item[i)] Set $\boldsymbol{\Gamma}_{\ell}= \boldsymbol{\nu} \otimes [1 \: 1 \: 0]$ and $\boldsymbol{\Gamma}_{\ell-1}= \boldsymbol{\nu} \otimes [1 \: 0 \: 1]$ ($A \otimes B$ refers to the Kronecker product of the matrices $A$ and $B$).	
	\item[ii)] Update $\mathbf{Z}_{\ell}(t_{\ell}+ \Delta t_{\ell}) = \mathbf{Z}_{\ell}(t_{\ell})+ \Delta t_{\ell} \boldsymbol{\Gamma}_{\ell} \Lambda$
	\item[iii)] Update $\mathbf{Z}_{\ell-1}(t_{\ell}+ \Delta t_{\ell}) = \mathbf{Z}_{\ell-1}(t_{\ell})+ \Delta t_{\ell} \boldsymbol{\Gamma}_{\ell-1} \Lambda$
	\end{itemize}
	  \If{$(n \: mod \: 2) =0 $}  		
				$t_{\ell-1}=t_{\ell-1}+\Delta t_{\ell-1}$
	    \EndIf
	  
	\State $t_{\ell}=t_{\ell}+\Delta t_{\ell}$
	\EndWhile
\end{algorithmic}
\end{algorithm}
\end{small}

We consider  a number of reactions $J>1$,  $\mathbf{X} \in \nset^d, d \ge 1$ and  $g: \rset^d \rightarrow \rset$.  Hereafter, we denote by $\{\nu_{j,i}\}_{j=1}^J$ the  coordinates in the stoichiometric vectors, $ \{\boldsymbol{\nu}_j\}_{j=1}^J$,  corresponding to the state change of the $i$th species.
For a fixed $0\le n \le N_{\ell}-1$, we define  $\mathcal{F}_n$ to be the sigma algebra given by
\begin{align}\label{eq:conditioned_filtration_high_dim}
\mathcal{F}_{n}:=\sigma\left( \left(\Delta a^j_{\ell,k}\right)_{j=1,\dots,J;k=0,\dots,n}  \right), \quad n=0,\dots,N_{\ell}-1.
\end{align}
The  likelihood ratio for each reaction channel $j \in \mathcal{J}_1$  has a similar expression to \eqref{eq: lik_level_ell}, and is given by
\begin{align}\label{eq:lik_fact_each_reaction}
L_{\ell}^j =\left(e^{(c_{\ell} -1) \Delta t_{\ell} \sum_{n \in \mathcal{I}_{\ell,j}^{s}} \Delta a^j_{\ell,n}}\right) \left(c_{\ell}^{-\sum_{n \in \mathcal{I}_{\ell,j}^{s}} k^j_n}\right), \quad  j \in \mathcal{J}_1,
\end{align}
where $k_n^j$ is the number of jumps associated with the $j$th reaction channel that occurs at the $n$th time step where we apply the change of measure, and  $\mathcal{I}_{\ell,j}^{s}$ is the random set including the time steps at level $\ell$,  where we simulate the Poisson rdvs under the new measure for the $j$th reaction channel.

Thus, across one path, the  likelihood ratio is given by
\begin{small}
\begin{align}\label{eq:lik_high_dim_total}
L_{\ell}=\prod_{j \in \mathcal{J}_1} L_{\ell}^j=\left(e^{(c_{\ell} -1) \Delta t_{\ell}\sum_{j \in \mathcal{J}_1} \sum_{n \in \mathcal{I}_{\ell,j}^{s}} \Delta a^j_{\ell,n}}\right) \left(c_{\ell}^{-\sum_{j \in \mathcal{J}_1}\sum_{n \in \mathcal{I}_{\ell,j}^{s}} k^j_n}\right).
\end{align}
\end{small}
Similarly to Section \ref{sec:One dimensional case}, we choose  $c_{\ell}$ to be  given by \eqref{eq:choice_c_l_step1}, with $\delta>0$. Remark \ref{rem:proba_0event_delta} holds for the high dimensional case. In particular, compare Figures \ref{fig:catastrophic_coupling_illustration_example2_X_1} and \ref{fig:catastrophic_coupling_illustration_example4_X_3}, for the case without  IS, and Figures \ref{fig:catastrophic_coupling_illustration_exampl2_X_1_imp_sampling} and  \ref{fig:catastrophic_coupling_illustration_exampl3_X_3_imp_sampling} for the case using IS  with $\delta=\frac{3}{4}$, for $g(\mathbf{X})=X^{(i)}$, \ie, the projection on the $i$th coordinate of the state vector $\mathbf{X}$.

\subsection{Convergence Estimates of MLMC combined with IS}\label{sec:Convergence Estimates and Optimal IS Parameter_high_dim}

In this section, we aim to derive convergence estimates of the kurtosis and the variance.  We start by  stating the main two assumptions (Assumptions \ref{assumption1_high_dim} and \ref{assumption2_high_dim}), needed to derive the main results in this section.  For the ease of presentation, we consider $g(\mathbf{X})=X^{(i)}$, the projection on the $i$th coordinate of the state vector $\mathbf{X}$.
\begin{assumption}\label{assumption1_high_dim}
For a small $\Delta t_{\ell}$, and conditioning on $\mathcal{F}_{N_{\ell}-1}$ and $\left(\mathcal{I}^s_{\ell,j}=\mathcal{S}_j\right)_{j \in \mathcal{J}_1}$,  we denote, for  $j \in \mathcal{J}_1 $, $K_j=  \sum_{n \in \mathcal{S}_j} k_n^j$, and  we assume that, for $0 \le \delta<1$,

(a) for all $\overline{\mathbf{K}} \in \nset^{\# \mathcal{J}_1}$ such that $\sum_{j \in \mathcal{J}_1}  \nu_{j,i} \overline{K}_j \neq 0$, we have
\begin{small}
\begin{equation*}
\frac{\bar{\pi}_{\ell}\left( \abs{\Delta g_{\ell}(T)}=\abs{ \sum_{j \in \mathcal{J}_1} \nu_{j,i} \overline{K}_j}, \cap_{j \in \mathcal{J}_1} \{K_j=\overline{K}_j\} ; \left(\mathcal{F}_{N_{\ell}-1}, \left(\mathcal{I}^s_{\ell,j}=\mathcal{S}_j\right)_{j \in \mathcal{J}_1}\right)\right)}{\sum_{j \in \mathcal{J}_1}\bar{\pi}_{\ell}\left( \abs{\Delta g_{\ell}(T)}=\abs{\nu_{j,i} }, \{K_j= 1\}  \cap_{q \neq j} \{K_q =0\}; \left(\mathcal{F}_{N_{\ell}-1}, \left(\mathcal{I}^s_{\ell,j}=\mathcal{S}_j\right)_{j \in \mathcal{J}_1}\right)\right)} \le 1.
\end{equation*}
\end{small} 

(b)  for all $q \ge 1$, there exists $\eta_{q,\ell}>0$ such that,
 \begin{small}
\begin{equation*}
\bar{\pi}_{\ell} \left\lbrace \abs{\Delta g_{\ell}(T)}= q ;\left(\mathcal{F}_{N_{\ell}-1}, \left(\mathcal{I}^s_{\ell,j}=\mathcal{S}_j\right)_{j \in \mathcal{J}_1}\right)  \right\rbrace \le \eta_{q,\ell} \Delta t_{\ell}^{(1-\delta) q}, \quad \text{ with} \: \frac{\eta_{q+1,\ell}}{\eta_{q,\ell}}\le \eta,
\end{equation*}
\end{small}
with $\eta$ independent of $\ell$ and $q$.

(c) for all  $j \in \mathcal{J}_1$ there  exist a single  $n_j^\ast \in \mathcal{S}_j$ such that
\begin{small}
\begin{equation*}
\bar{\pi}_{\ell} \left\lbrace \abs{\Delta g_{\ell}(T)}=\abs{\nu_{j,i} },  \{K_j= 1\}  \cap_{q \neq j} \{K_q =0\} ; \left(\mathcal{F}_{N_{\ell}-1}, \left(\mathcal{I}^s_{\ell,j}=\mathcal{S}_j\right)_{j \in \mathcal{J}_1}\right)  \right\rbrace = e^{-\Delta t_{\ell}^{1-\delta} \Delta a^j_{\ell,n_j^\ast}} \left(\Delta t_{\ell}^{1-\delta} \Delta a^j_{\ell,n_j^\ast}\right) \left(1+\ordo{1}\right).
\end{equation*}
\end{small}

\end{assumption}
\begin{assumption}\label{assumption2_high_dim}
There exists $ \ell_0\ge 0$ such that for all  $ \ell \ge \ell_0$,  we have 
\begin{align*}
0<C_1\le \sum_{j \in \mathcal{J}_1} \exptpibar{e^{- \Delta a^j_{\ell,n_j^\ast}} \Delta a^j_{\ell,n_j^\ast}}  \le \exptpibar{e^{ \sum_{j \in \mathcal{J}_1} \sum_{n \in \mathcal{I}^s_{\ell,j}} \Delta a^j_{\ell,n_j^\ast}}} \le  C_2 <\infty,
\end{align*}
where $ C_1$, $C_2$ are independent of $\ell$, and $n_j^\ast$ are defined in Assumption \ref{assumption1_high_dim} (c).
\end{assumption}
We emphasize that Assumption \ref{assumption1_high_dim} (c) is   motivated by our numerical observations, which suggest that for small values of $\Delta t_{\ell}$, we sample at most one single step using our IS algorithm, which 
separates the two paths (see Figures \ref{fig:number_k_n>0_decay}, \ref{fig:number_k_n>0_example2} and  \ref{fig:number_k_n>0_example4} in Appendix \ref{appendix:Numerical evidence of   Assumptions}).   Furthermore, by observing that $\Delta a^j_{\ell,n}=\Ordo{1},\: \forall j \in \mathcal{J}_1$, Assumption \ref{assumption2_high_dim} is   motivated by our numerical observations (see Figures \ref{fig: IS steps decay} and  \ref{fig: IS steps example4}), which show that $\expt{\sum_{j\in \mathcal{J}_1}\#\mathcal{I}^s_{\ell,j}}=\Ordo{1}$.

Now, we state the main results of this section through Theorems \ref{thm:kurtosis_high_dim} and \ref{thm:variance_high_dim}. The proof of these theorems  are identical   to the one dimensional proofs (one species and one reaction ($J=1$)) with slight differences. Consequently, for ease of presentation, we present in Appendix \ref{appendix:Proofs of Lemma}  the one dimensional proofs.  The  key result for these proofs is Lemma \ref{lemma: moments} which is proven   in Appendix \ref{appendix:Proofs of Lemma}. In the following and  without loss of generality, we also assume that $\abs{\nu_1}=1$. 
\begin{lemma}[Conditional $L^p$ moments estimates]\label{lemma: moments}
Let $J=1$, $p \ge 1$ and $0 \le \delta<1$, and suppose that Assumptions  \ref{assumption1_high_dim} and \ref{assumption2_high_dim} hold, then, for   $\Delta t_{\ell} \rightarrow 0$, we have 
\begin{small}
\begin{equation}\label{eq:moments}
\exptpibar{\abs{\Delta g_{\ell}}^p(T) L_{\ell}^p ; \left(\mathcal{F}_{N_{\ell}-1}, \mathcal{I}^s_{\ell}=\mathcal{S}\right)}=  \Delta t_{\ell}^{(p -1)\delta+1}  \left(\Delta a_{\ell,n^\ast}\right)  e^{p(\Delta t_{\ell}^{1-\delta} -\Delta t_{\ell})  \sum_{n \in \mathcal{S}} \Delta a_{\ell,n}}  e^{- \left( \Delta t_{\ell}^{1-\delta}  \Delta a_{\ell,n^\ast} \right)}  \left(1+h_{p,\ell}\right),
\end{equation}
\end{small}
such that $h_{p,\ell} \underset{\Delta t_{\ell} \rightarrow 0}{\longrightarrow 0}$.
\end{lemma}
\begin{remark}
Note that for $J \ge 1$, Lemma \ref{lemma: moments} is extended to the multi-channel case by expressing the right-hand side of \ref{eq:moments} as  a summation over the set $\mathcal{J}_1$ of similar terms but involving $\Delta a^j_{\ell,n}$ instead of $\Delta a_{\ell,n}$.  These terms correspond to only one jump occurring under the new measure and due to the firing of only one reaction channel $j \in  \mathcal{J}_1$.
\end{remark}
Finally a further assumption (Assumption \ref{assump:integrability_assump}) is needed to prove   Theorems \ref{thm:kurtosis_high_dim} and \ref{thm:variance_high_dim}
\begin{assumption}\label{assump:integrability_assump}
For a sufficiently large $\ell$, we assume that there exists a constant $C_p$, independent of $\ell$,  such that $h_{p,\ell}$ in Lemma \ref{lemma: moments}  fulfills  $\exptpibar{h_{p,\ell}^2}\le C_p <\infty$.
\end{assumption}
Theorem \ref{thm:kurtosis_high_dim} shows that the kurtosis at level $\ell$ of the MLMC estimator combined with our IS algorithm, $\kappa_{\ell}$, is $\Ordo{\Delta t_{\ell}^{\delta-1}}$.
\begin{theorem}\label{thm:kurtosis_high_dim}
Let $J\ge 1$ and  let us denote $Y_{\ell}:=\Delta g_{\ell}(T) L_{\ell}$. Suppose that Assumptions  \ref{assumption1_high_dim}, \ref{assumption2_high_dim}  and \ref{assump:integrability_assump} hold. Then, for $0\le \delta<1$ and $\Delta t_{\ell} \rightarrow 0$, we have 
\begin{equation}\label{eq:kurtosis_approximate_formula_high_dim}
\kappa_{\ell}:=\frac{\exptpibar{\left(Y_{\ell}-\exptpibar{Y_{\ell}}\right)^4}}{\left(\text{Var}_{\bar{\pi}_{\ell}}\left[Y_{\ell}\right]\right)^2}=\Ordo{\Delta t_{\ell}^{\delta-1}}.
\end{equation}
\end{theorem}
Theorem  \ref{thm:kurtosis_high_dim} clearly shows  the effect seen for the limiting case $\delta=0$ where we do not apply the IS algorithm and thus, the kurtosis $\kappa_{\ell}$ increases at a rate of $\Delta t ^{-1}_{\ell}$.   Compared to  the case without  IS, we reduce the kurtosis by a factor of $\Delta t^{-\delta}_{\ell}$. Theorem  \ref{thm:kurtosis_high_dim} is confirmed by our numerical experiments in Section \ref{sec:num_experiments}.

Let us fix $0\le \delta<1$. We show in Theorem \ref{thm:variance_high_dim} that the strong convergence rate is $\beta=\delta+1$.
\begin{theorem}\label{thm:variance_high_dim}
Let   $0\le \delta<1$ and  $J\ge 1$ and  let us denote $Y_{\ell}:=\Delta g_{\ell}(T) L_{\ell}$. Suppose that Assumptions  \ref{assumption1_high_dim}, \ref{assumption2_high_dim}  and \ref{assump:integrability_assump} hold. Then, for $\Delta t_{\ell} \rightarrow 0$,  we have 
\begin{equation*}
\text{Var}_{\bar{\pi}_{\ell}}\left[Y_{\ell} \right] =\Ordo{\Delta t_{\ell}^{1+\delta}}.
\end{equation*}
\end{theorem}
The result in Theorem \ref{thm:variance_high_dim} is confirmed by the  numerical experiments in Section \ref{sec:num_experiments},  which demonstrate that our IS algorithm improves the strong convergence rate from $\beta=1$ (see Figures \ref{fig:MLMC_decay_X_0_10}, \ref{fig:MLMC_exp2_X_1}, and  \ref{fig:MLMC_exp4_X_3}) to $\beta=1+\delta$ with $\delta>0$ (see Figures   \ref{fig: MLMC estimator_imp_sampling choice_1delta_05_exp1} and \ref{fig: MLMC estimator_imp_sampling choice_1delta_075_exp1} for   Example \ref{exp:decay},  \ref{fig: MLMC estimator_imp_sampling choice_1delta_05_exp2} and \ref{fig: MLMC estimator_imp_sampling choice_1delta_075_exp2} for  Example \ref{exp:Gene transcription and translation}, and Figures  \ref{fig: MLMC estimator_imp_sampling choice_1delta_05_exp3} and \ref{fig: MLMC estimator_imp_sampling choice_1delta_075_exp3} for  Example \ref{exp:Michaelis–Menten enzyme kinetics}) for the case with  IS.   Due to  Theorem \ref{thm:MLMC_comlexity}, and given that $\gamma=1$ (work rate) for both cases, with and without  IS, we  improve the complexity of the MLMC method from $\Ordo{\text{TOL}^{-2} \log(\text{TOL})^2}$ for the case without IS  to the optimal complexity, \ie,  $\Ordo{\text{TOL}^{-2}}$, for the case with  IS, where  $\text{TOL}$ is  a pre-selected tolerance.
\begin{remark}[More general observable $g$]
For ease of presentation, we formulate our assumptions and show our  proofs for an observable $g$ in the class of projections. However, our results can be easily extended to include  linear maps, and linear combination of indicator functions.
\end{remark}

\subsection{Cost Analysis}\label{sec:Cost analysis}
In this section, we  analyze briefly  the computational costs when using MLMC with our IS technique compared to  standard MLMC, in the context of SRNs. Let   $M_{\ell}$  denote the number of samples at level $\ell$, and $W_{\ell,\text{sample}}$  the expected cost per sample path at level $\ell$. Observe that the expected computational cost of the MLMC estimator is given by 
\begin{align*}
W_{\text{MLMC}}:=\sum_{\ell=0}^L  M_{\ell} W_{\ell,\text{sample}},
\end{align*}
If we denote  by   $W^{\text{with IS}}_{\ell,\text{sample}}$ and $W^{\text{without IS}}_{\ell,\text{sample}}$  the expected costs of simulating one sample path  at level $\ell$ with  and without IS, respectively, then we have 
\begin{align}\label{eq:cost_MLMC_IS}
W^{\text{without IS}}_{\ell,\text{sample}} &\approx 2 \times J \times C_p \times \Delta t_{\ell}^{-1}\nonumber\\
 W^{\text{with IS}}_{\ell,\text{sample}}&\approx 2 \times J \times C_p \times \Delta t_{\ell}^{-1}+ \underset{\ll C_p} {\underbrace{C_{\text{lik}}}}  \times \underset{\ll \Delta t_{\ell}^{-1}}{\underbrace{\overline{\sum_{j \in \mathcal{J}_1} \# \mathcal{I}^s_{\ell,j}}}} \approx W^{\text{without IS}}_{\ell,\text{sample}} \COMMA
\end{align}
where $C_p$ is the  cost of generating one Poisson rdv, $C_{\text{lik}}$ is the cost of computing the likelihood ratio, and  $\overline{\sum_{j \in \mathcal{J}_1}\# \mathcal{I}^s_{\ell,j}}$ is  the average number of time steps at level $\ell$, where we simulate under the new measure  the $j$th reaction channel. We note that the inequality $\overline{\sum_{j \in \mathcal{J}_1}\# \mathcal{I}^s_{\ell,j}}  \ll \Delta t_{\ell}^{-1}$ is motivated  and justified by  the construction of our IS algorithm, where IS is only applied a few times across each simulated path. This is also confirmed by Figure \ref{fig: IS steps}.  Furthermore, we refer to  Figure \ref{fig: cost_comparison} for  evidence of  the observation made by \eqref{eq:cost_MLMC_IS}.
\FloatBarrier
\begin{figure}[h!]
	\centering
	\begin{subfigure}{.47\textwidth}
		\centering
		\includegraphics[width=1\linewidth]{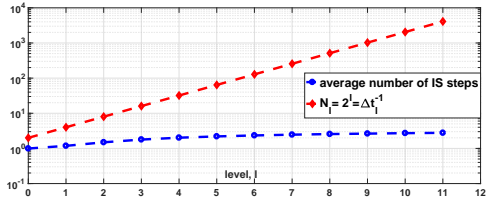}
		\caption{Example \ref{exp:decay}.}
		\label{fig: IS steps decay}
	\end{subfigure}
	\begin{subfigure}{.47\textwidth}
		\centering
		\includegraphics[width=1\linewidth]{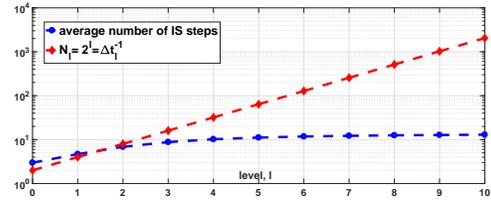}
		\caption{Example \ref{exp:Michaelis–Menten enzyme kinetics}.}
		\label{fig: IS steps example4}
	\end{subfigure}
	\caption{Average number of time steps for different MLMC levels, $\overline{\sum_{j \in \mathcal{J}_1}\# \mathcal{I}^s_{\ell,j}}$, with  IS (with $\delta=\frac{3}{4}$),  with $10^5$ samples.  This Figure shows  that $\overline{\sum_{j \in \mathcal{J}_1}\# \mathcal{I}^s_{\ell,j}} = \Ordo{1} \ll \Delta t_{\ell}^{-1}$.}
	\label{fig: IS steps}
\end{figure}
\FloatBarrier
\begin{figure}[h!]
	\centering
	\begin{subfigure}{.47\textwidth}
		\centering
		\includegraphics[width=1\linewidth]{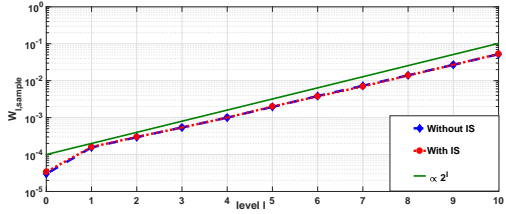}
		\caption{Example \ref{exp:Gene transcription and translation}.}
		\label{fig: cost_comparison example2}
	\end{subfigure}
	\begin{subfigure}{.47\textwidth}
		\centering
		\includegraphics[width=1\linewidth]{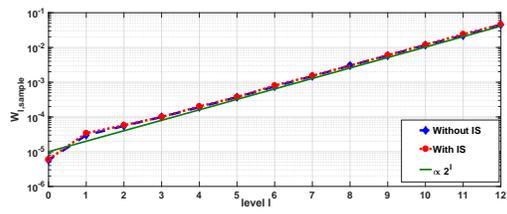}
		\caption{Example \ref{exp:Michaelis–Menten enzyme kinetics}.}
		\label{fig: cost_comparison example3}
	\end{subfigure}
	\caption{Comparison of the average cost per sample path per level (in CPU time and estimated with $10^6$ samples).  MLMC combined with  IS (with $\delta=\frac{3}{4}$) has  the same  average cost per sample path per level, as  standard MLMC.}
	\label{fig: cost_comparison}
\end{figure}
\FloatBarrier
Furthermore, if we denote by $V_{\ell}=\text{Var}\left[g_{\ell}-g_{\ell-1}\right]$, then   from our analysis in Section \ref{sec:Generalization of our method to multi-channels and high dimensional states}, and  our numerical experiments in Section \ref{sec:num_experiments}, it is shown that  $V_{\ell}^{\text{with IS}} \ll V_{\ell}^{\text{without IS}}$ implying that  $\Rightarrow$ $M_{\ell}^{\text{with IS}} \ll M_{\ell}^{\text{without IS}}$ (see Figure \ref{fig:  samples_comparison}).

Therefore, we conclude that combining our pathwise IS with the MLMC estimator not only improves its robustness and  convergence behavior, but also  significantly reduces  the cost. 
\FloatBarrier
\begin{figure}[h!]
	\centering
	\begin{subfigure}{.47\textwidth}
		\centering
		\includegraphics[width=1\linewidth]{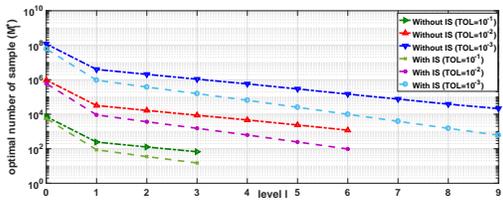}
		\caption{Example \ref{exp:Gene transcription and translation}.}
		\label{fig: samples_comparison example2}
	\end{subfigure}
	\begin{subfigure}{.47\textwidth}
		\centering
		\includegraphics[width=1\linewidth]{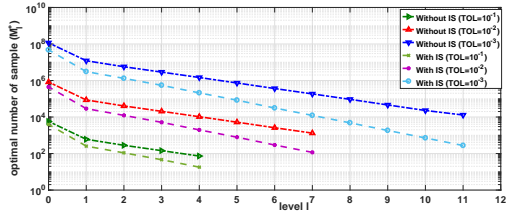}
		\caption{Example \ref{exp:Michaelis–Menten enzyme kinetics}.}
		\label{fig: samples_comparison example3}
	\end{subfigure}
	\caption{Comparison of the optimal number  of samples per level, using \eqref{eq:optimal_number_samples}, for different pre-selected tolerances $\text{TOL}$ for the MLMC estimator. For a fixed $\text{TOL}$ , the optimal number of samples per level of  MLMC with IS (with $\delta=\frac{3}{4}$) is significantly smaller than the one  of MLMC without IS.}
	\label{fig:  samples_comparison}
\end{figure}
\FloatBarrier

%% file: Num_experiments.tex
In the following, we illustrate the main benefits of the  MLMC-based method, when used in combination with  our IS algorithm (explained in Section \ref{sec:About the choice of the new measure and the importance sampling algorithm}), compared to the standard MLMC   used in \cite{Anderson2012}. We consider three different examples of SRNs, given by Examples \ref{exp:decay}, \ref{exp:Gene transcription and translation}, and \ref{exp:Michaelis–Menten enzyme kinetics}, where we use  the MLMC method  to estimate $E[g\left(\mathbf{X}(T)\right)]$, where  $\mathbf{X}$ is the state vector representing the counting number of each species in the system, $g:  \rset^d \rightarrow \rset$ is a given  scalar observable of $\mathbf{X}$, and $T>0$ is a user-selected final time.   We note that our numerical results were  obtained using an Intel(R) Xeon(R) CPU E5-2680 architecture. Furthermore, the computer code is written in the MATLAB programming language (version R2019a), and it can be downloaded from  \blue{\url{https://github.com/hammouc/MLMC_IS_SRNs}}.

\begin{example}[Decay example]\label{exp:decay}
This model has one reaction,
\begin{equation*}
X \overset{\theta_1}{\rightarrow} \emptyset,
\end{equation*}
with $\theta_1=1$, $T=1$, and $X_0=10$. The stoichiometric scalar $\nu=-1$ and the propensity function $a(x)=\theta_1 x$. The quantity of interest in this example is $E[X(T)]$. 
\end{example}
\begin{example}[Gene transcription and translation \cite{Anderson2012}]\label{exp:Gene transcription and translation}
This model has five reactions,
\begin{align*}
\emptyset &\overset{\theta_1}{\rightarrow} R, \quad R \overset{\theta_2}{\rightarrow} R+P \nonumber\\
2P &\overset{\theta_3}{\rightarrow} D, \quad R \overset{\theta_4}{\rightarrow} \emptyset \\
 P &\overset{\theta_5}{\rightarrow} \emptyset \nonumber
\end{align*}
with  $\boldsymbol{\theta}=(25, 10^3, 0.001, 0.1, 1)$, $T=1$, $\mathbf{X}(t)=(R(t), P(t), D(t))$ and $\mathbf{X}_0=(0, 0, 0)$. The stoichiometric matrix  and the propensity functions are given by
\begin{small}
\begin{equation*}
\boldsymbol{\nu}= 
\left({\begin{array}{ccc} 1 & 0 & 0  \\ 0 & 1 & 0 \\ 0 & -2 & 1 \\-1 & 0 & 0 \\ 0  & -1  & 0 \end{array}}\right) ,\quad
a(\mathbf{X})=\left(\begin{array}{ccc} \theta_1    \\ \theta_2 R \\ \theta_3 P (P-1) \\ \theta_4 R  \\ \theta_5 P  \end{array}\right)
\end{equation*} 
\end{small}
The quantity of interest is $E[X^{(1)}(T)]$. We note that the choice of $X^{(1)}$ as the target species  was determined by  selecting the $i$th species with the highest probability of having $\overline{X}^{(i)}_{\ell}(T)-\overline{X}^{(i)}_{\ell-1}(T)=0$ on the deep levels, resulting in  the most severe \textit{catastrophic coupling} explained in Section \ref{sec:Catastrophic coupling}.  In this example, the coarsest level of the MLMC estimator is  $L_0=2$.
\end{example}
\begin{example}[Michaelis-Menten enzyme kinetics \cite{rao2003stochastic}]\label{exp:Michaelis–Menten enzyme kinetics}
The catalytic conversion of a substrate, $S$, into a product, $P$, via an enzymatic reaction involving enzyme, $E$. This
is described by Michaelis-Menten enzyme kinetics with  three  reactions,
\begin{align*}
E+S &\overset{\theta_1}{\rightarrow} C, \quad C \overset{\theta_2}{\rightarrow} E+S \nonumber\\
C &\overset{\theta_3}{\rightarrow} E+P, 
\end{align*}
with  $\boldsymbol{\theta}=(0.001, 0.005, 0.01)$, $T=1$, $\mathbf{X}(t)=(E(t), S(t), C(t), P(t))$ and $\mathbf{X}_0=(100,100, 0, 0)$. The stoichiometric matrix  and the propensity functions are given by
\begin{small}
\begin{equation*}
\boldsymbol{\nu}= 
\left({\begin{array}{cccc} -1 & -1 & 1 & 0  \\  1 & 1 & -1 & 0 \\  1 & 0 & -1 & 1  \end{array}}\right) ,\quad
a(\mathbf{X})=\left(\begin{array}{ccc} \theta_1 E S   \\ \theta_2 C \\  \theta_3 C \end{array}\right)
\end{equation*} 
\end{small}
The quantity of interest in this example is $E[X^{(3)}(T)]$.
\end{example}
We show in Table \ref{table:Comparison of  convergence rates for the different numerical examples with and without importance sampling algorithm.} the summarized results, related  to the convergence rates,  for the different scenarios without/with IS, and for the different examples that we consider in our numerical experiments.  We also show several cases depending on the parameter, $\delta$, used in the IS algorithm. From this table, we can see that our IS  algorithm, besides dramatically reducing  the kurtosis,  improves the strong convergence rates from $1$ to $1+\delta$, which then improves the total complexity of the MLMC estimator from $\Ordo{\text{TOL}^{-2} \log(\text{TOL})^2)}$ to  $\Ordo{\text{TOL}^{-2}}$, where $\text{TOL}$ is a pre-selected tolerance. This improvement is confirmed by Figure \ref{fig: complexity_comparison}, which shows that MLMC, when used in combination  with our IS  algorithm, achieves the same numerical complexity, $\Ordo{\text{TOL}^{-2}}$ as MC with an exact method (SSA), but with a significantly smaller  constant.   The detailed convergence plots for each example are presented  in Sections \ref{sec:umerical Results of MLMC Without IS} and \ref{sec: Numerical Results of MLMC With IS}.

Figure \ref{fig: complexity_comparison} illustrates  the improvement of the complexity rate compared to standard MLMC. For both examples \ref{exp:Gene transcription and translation} and \ref{exp:Michaelis–Menten enzyme kinetics},  MLMC, when used in combination  with our IS  algorithm,  significantly outperforms  the standard MLMC. In particular, to achieve a desired accuracy of $\text{TOL}=10^{-3}$ in example \ref{exp:Gene transcription and translation}, MLMC with IS ($\delta=\frac{3}{4}$)  requires  around $30\%$ of the total work (in CPU time) of MLMC without IS. To achieve the same accuracy  in example \ref{exp:Michaelis–Menten enzyme kinetics},  MLMC with IS ($\delta=\frac{3}{4}$) requires around $18\%$ of the total work  of MLMC without IS. We note that the different parameters of the MLMC estimator such as i) the coarsest level, $L_0$, the deepest level, $L$, and the optimal number of samples $\{M_{\ell}\}_{\ell=L_0}^L$, were selected using a similar procedure to the procedure in \cite{hammouda2017multilevel}.
\FloatBarrier
\begin{small}
\begin{table}[h!]
	\centering
	\begin{tabular}{l*{6}{c}r}
		\toprule[1.5pt]
	 Example  &  $\alpha$  & $\beta$ &  $\gamma$   &$\kappa_{L}$ \\
	\hline
	      Example \ref{exp:decay} without IS   & $1.04$ & $1.03$ & $1$ &  $2191$ \\
		 Example \ref{exp:decay} with IS ($\delta=1/4$)  & $1.04$  & $1.27$ & $1$  & $275$\\
		 Example \ref{exp:decay} with IS ($\delta=1/2$)  & $1.04$  & $1.57$ & $1$  & $34.2$\\
		  Example \ref{exp:decay} with IS ($\delta=3/4$)  & $1.04$  & $1.93$ & $1$ & $5.1$ \\
		 \hline
	  Example \ref{exp:Gene transcription and translation} without IS  & $1$ & $0.99$ & $1$   &  $3290$ \\
	   Example \ref{exp:Gene transcription and translation} with  IS  ($\delta=1/4$) & $1$  & $1.23$ & $1$ & $409$  \\
		 Example \ref{exp:Gene transcription and translation} with IS  ($\delta=1/2$) & $1$  & $1.47$ & $1$  & $50$  \\
		  Example \ref{exp:Gene transcription and translation} with IS ($\delta=3/4$) & $1$  & $1.72$ & $1$ & $5.8$  \\
		  \hline
		   Example \ref{exp:Michaelis–Menten enzyme kinetics} without IS   & $1.02$ & $1.03$ & $1$  &  $1130$   \\
	 Example \ref{exp:Michaelis–Menten enzyme kinetics} with IS  ($\delta=1/4$) & $1.02$  & $1.26$ & $1$ & $208$   \\
		 Example \ref{exp:Michaelis–Menten enzyme kinetics} with  IS  ($\delta=1/2$) & $1.02$  & $1.5$ & $1$   & $36.7$   \\
		  Example \ref{exp:Michaelis–Menten enzyme kinetics} with IS  ($\delta=3/4$) & $1.03$  & $1.75$ & $1$  & $5.9$  \\
		\bottomrule[1.25pt]
	\end{tabular}
	\caption{Comparison of  convergence rates ($\alpha$,  $\beta$,  $\gamma$), and  the kurtosis at the deepest levels of MLMC, $\kappa_L$,  for the different  examples with and without the IS algorithm. $\alpha, \beta,\gamma$ are the estimated rates of weak convergence, strong convergence and computational work, respectively, with a number of samples $M=10^6$. The detailed  convergence plots  are presented   in Sections \ref{sec:umerical Results of MLMC Without IS} and \ref{sec: Numerical Results of MLMC With IS}.}
	\label{table:Comparison of  convergence rates for the different numerical examples with and without importance sampling algorithm.}
\end{table}
\end{small}
\FloatBarrier
\begin{figure}[h!]
	\centering
	\begin{subfigure}{.5\textwidth}
		\centering
		\includegraphics[width=1\linewidth]{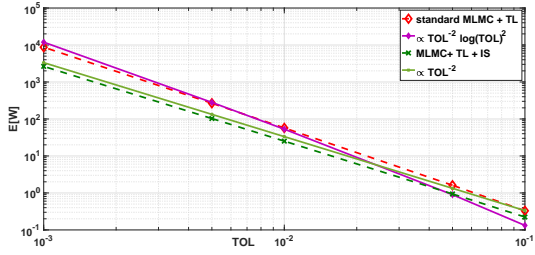}
		\caption{Example \ref{exp:Gene transcription and translation}.}
		\label{fig: complexity_comparison example2_1}
	\end{subfigure}
	\begin{subfigure}{.5\textwidth}
		\centering
		\includegraphics[width=1\linewidth]{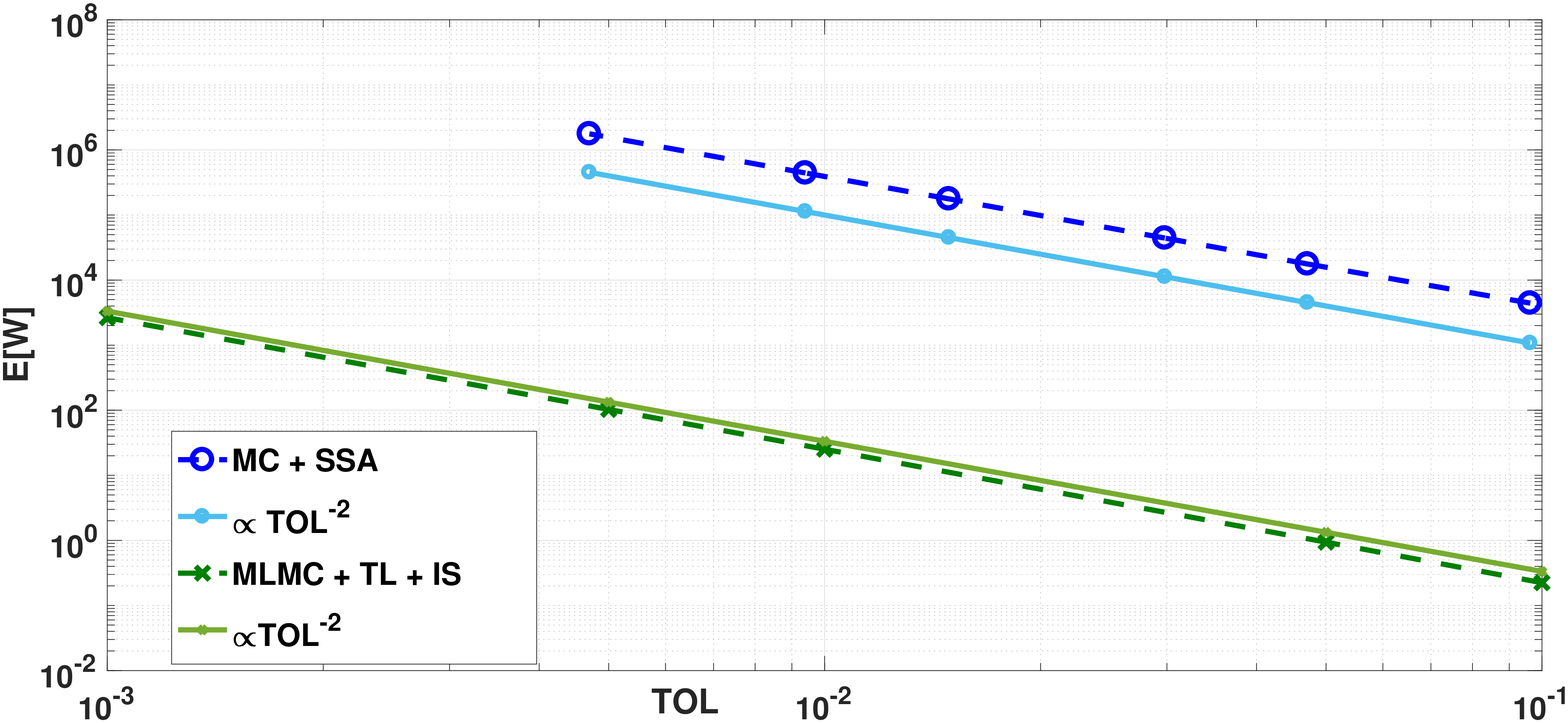}
		\caption{Example \ref{exp:Gene transcription and translation}}
		\label{fig: complexity_comparison example2_1}
	\end{subfigure}
	\begin{subfigure}{.5\textwidth}
		\centering
		\includegraphics[width=1\linewidth]{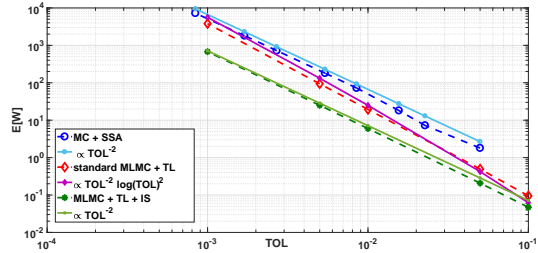}
		\caption{Example \ref{exp:Michaelis–Menten enzyme kinetics}.}
		\label{fig: complexity_comparison example3}
	\end{subfigure}
	\caption{Comparison of the numerical complexity of the different methods i) MC with SSA, ii) standard MLMC with TL, and iii) MLMC with TL in combination with IS  ($\delta=\frac{3}{4}$).  MLMC combined with our IS  algorithm achieves the same numerical complexity, $\Ordo{\text{TOL}^{-2}}$, as MC with an exact method (SSA), but with a significantly smaller  constant. On the other hand, MLMC in combination with IS  algorithm significantly outperforms  standard MLMC.}
	\label{fig: complexity_comparison}
\end{figure}
\FloatBarrier

\subsection{Numerical Results of MLMC Without IS}\label{sec:umerical Results of MLMC Without IS}
In Figures \ref{fig:MLMC_decay_X_0_10}, \ref{fig:MLMC_exp2_X_1} and \ref{fig:MLMC_exp4_X_3} , we show the convergence plots for the MLMC method without  IS for Examples \ref{exp:decay}, \ref{exp:Gene transcription and translation} and \ref{exp:Michaelis–Menten enzyme kinetics}, respectively. In these figures, and specifically from the right plot in the second row, we can see that for deep levels of MLMC, the kurtosis increases dramatically with respect to level $\ell$ of the MLMC method. This poor behavior of the kurtosis is mainly due to the \textit{catastrophic coupling} issue (explained in Section \ref{sec:Catastrophic coupling}), as illustrated by Figures \ref{fig:catastrophic_coupling_illustration_decay_X_1}, \ref{fig:catastrophic_coupling_illustration_example2_X_1} and \ref{fig:catastrophic_coupling_illustration_example4_X_3}. 
\FloatBarrier
\begin{figure}[h!]
		\centering
		\includegraphics[width=1\linewidth]{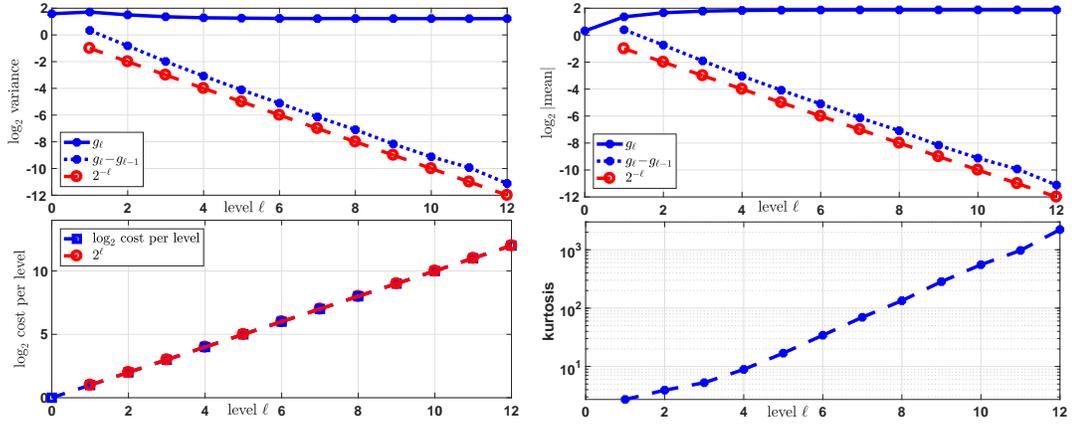}
	\caption{MLMC without IS  for Example \ref{exp:decay}: Convergence plots  with $g_{\ell}=\overline{X}_{\ell}(T)$.}
	\label{fig:MLMC_decay_X_0_10}
\end{figure}
\FloatBarrier
\begin{figure}[h!]
	\centering
	\begin{subfigure}{0.48\textwidth}
		\centering
		\includegraphics[width=1\linewidth]{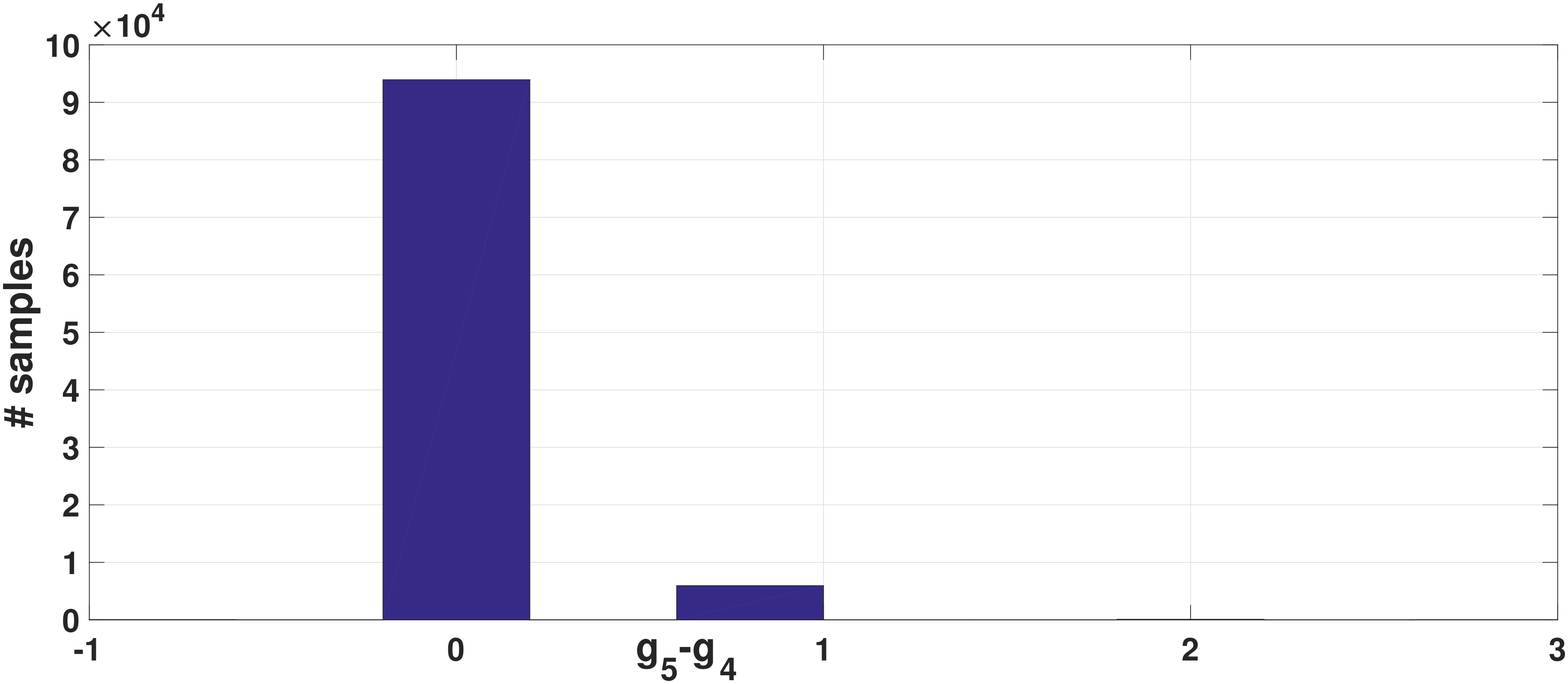}
		\caption{}
		\label{fig:sub3}
	\end{subfigure}
	\begin{subfigure}{0.48\textwidth}
		\centering
		\includegraphics[width=1\linewidth]{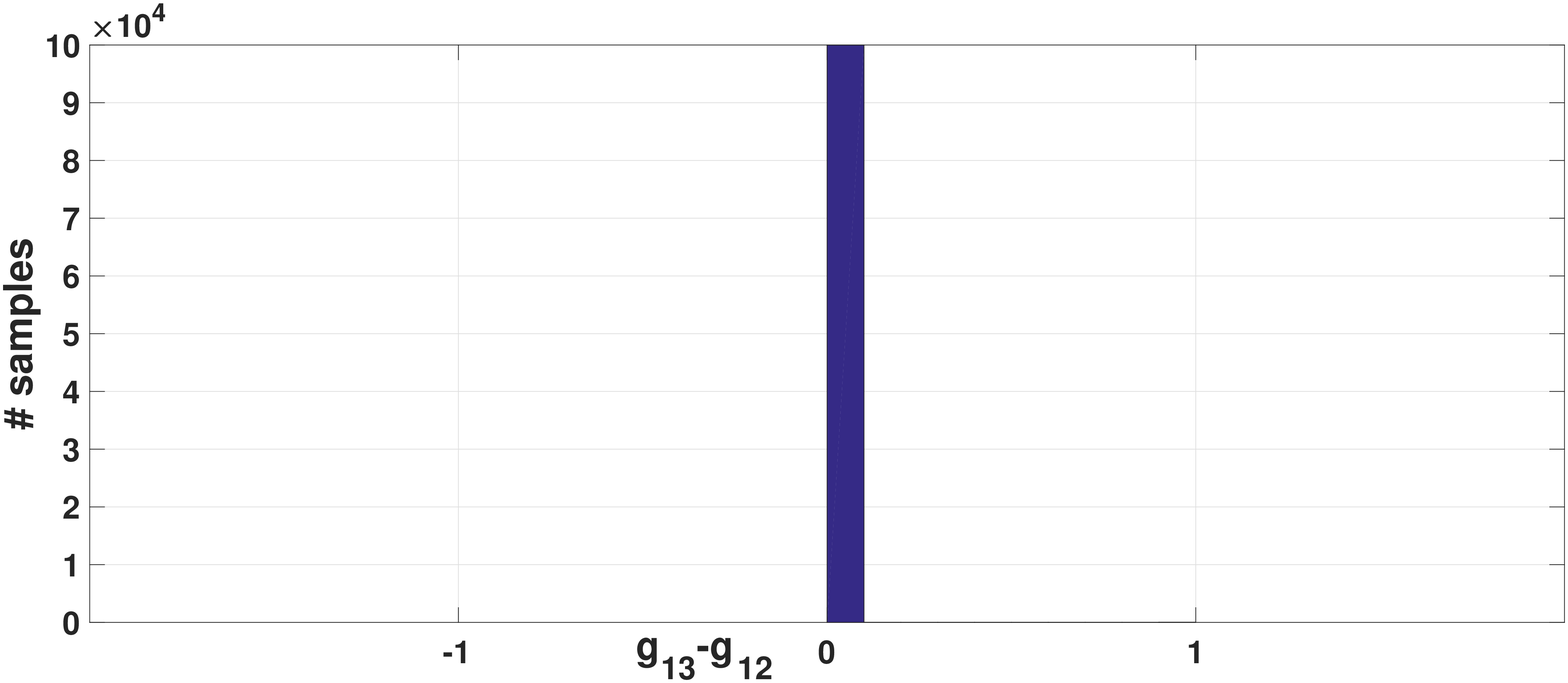}
		\caption{}
		\label{fig:sub4}
	\end{subfigure}
	\caption{Example \ref{exp:decay} without IS: Histogram of  $g_{\ell}-g_{\ell-1}$ ($g_{\ell}=\overline{X}_{\ell}(T)$), for number of samples $M_{\ell}=10^5$. The proportion of samples $\{g_{\ell}-g_{\ell-1}=0\}$ is an increasing function of the level, $\ell$,  of the MLMC estimator, to reach almost $100\%$ for $\ell=13$. a) $\ell=5$. b)  $\ell=13$.}
	\label{fig:catastrophic_coupling_illustration_decay_X_1}
\end{figure}
\FloatBarrier
\begin{figure}[h!]
		\centering
		\includegraphics[width=1\linewidth]{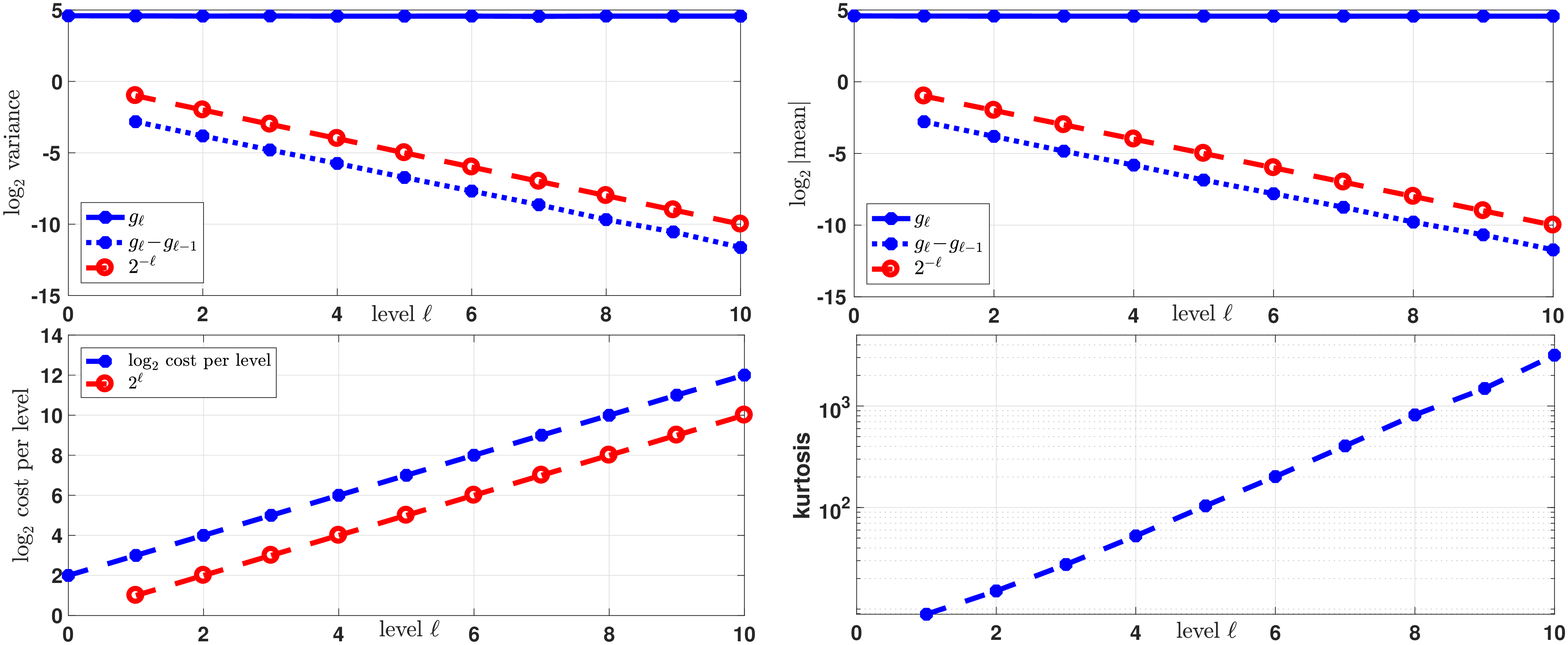}
	\caption{Convergence plots of MLMC without IS for Example \ref{exp:Gene transcription and translation}.}
	\label{fig:MLMC_exp2_X_1}
\end{figure}
\FloatBarrier
\begin{figure}[h!]
	\centering
	\begin{subfigure}{.48\textwidth}
		\centering
		\includegraphics[width=1\linewidth]{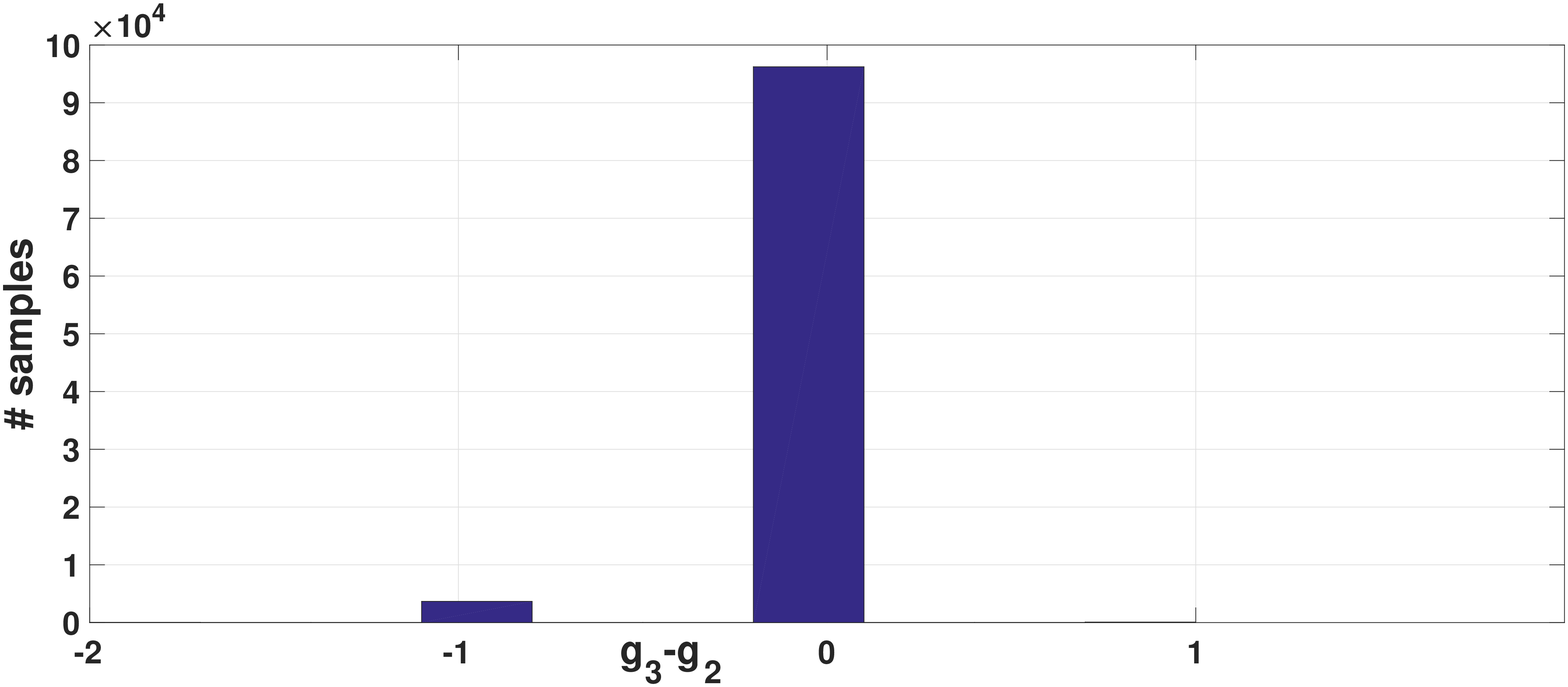}
		\caption{}
		\label{fig:sub3}
	\end{subfigure}%
	\begin{subfigure}{.48\textwidth}
		\centering
		\includegraphics[width=1\linewidth]{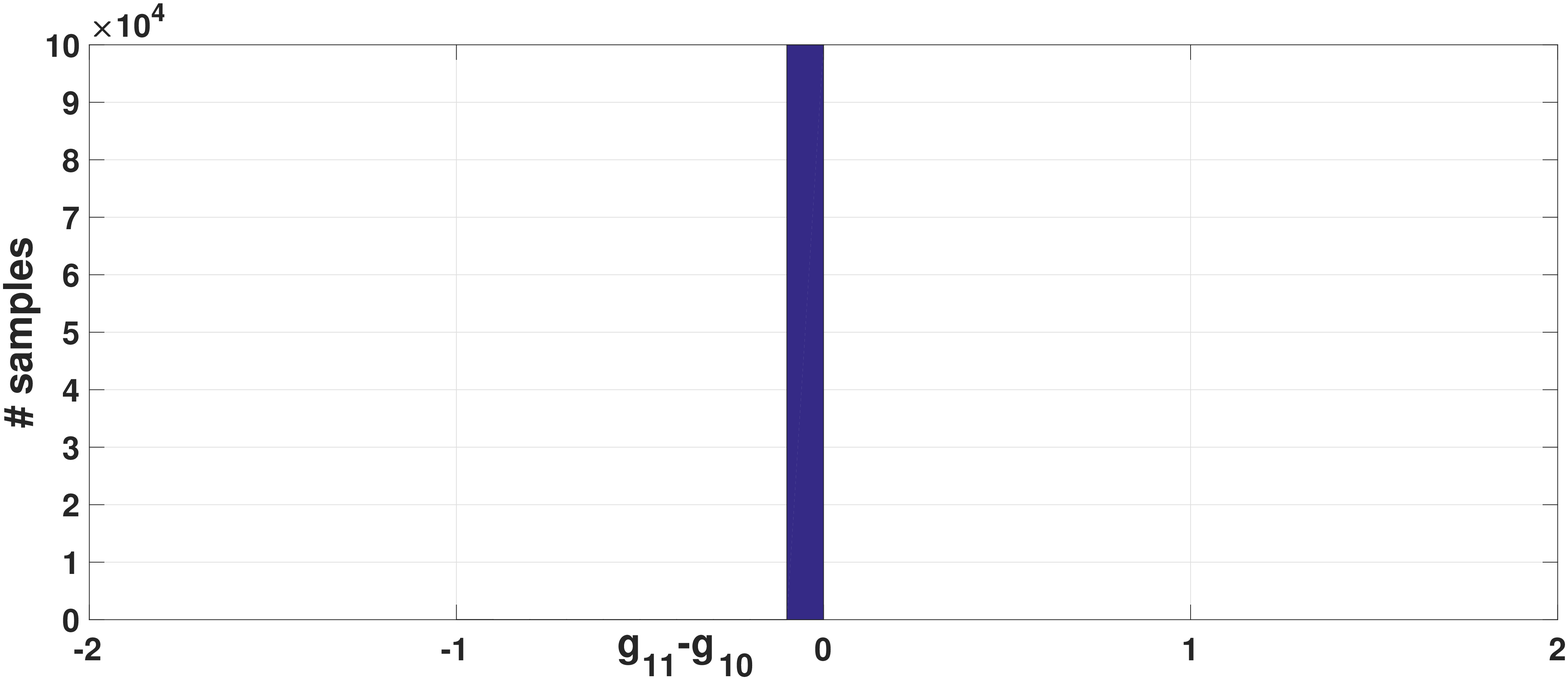}
		\caption{}
		\label{fig:sub4}
	\end{subfigure}
	
	\caption{Example \ref{exp:Gene transcription and translation} without IS: Histogram of  $g_{\ell}-g_{\ell-1}$ ($g_{\ell}=\overline{X}^{(1)}_{\ell}(T)$),  for number of samples $M_{\ell}=10^5$. The proportion of samples $\{g_{\ell}-g_{\ell-1}=0\}$ is an increasing function of the level, $\ell$,  of the MLMC estimator, to reach almost $100\%$ for $\ell=11$.  a) $\ell=3$. b)  $\ell=11$.}
	\label{fig:catastrophic_coupling_illustration_example2_X_1}
\end{figure}
\FloatBarrier
\FloatBarrier
\begin{figure}[h!]
		\centering
		\includegraphics[width=1\linewidth]{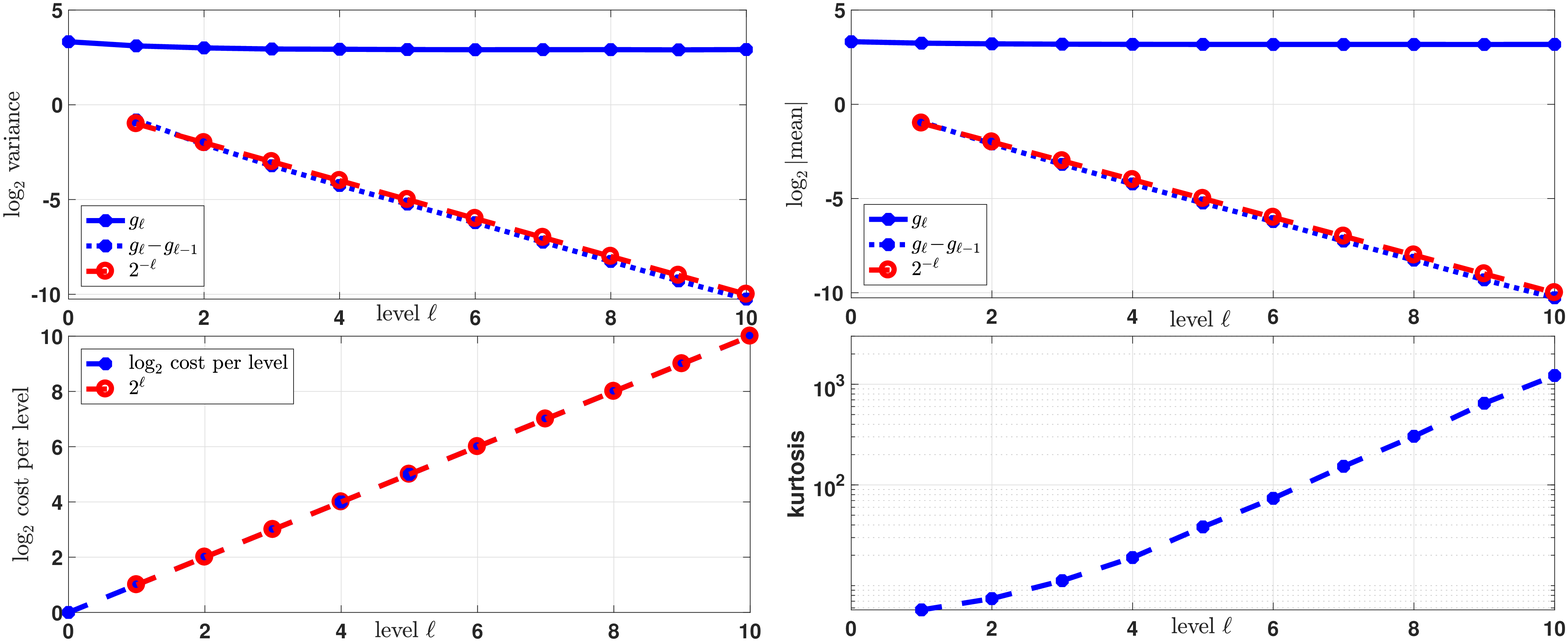}
	\caption{Convergence plots of MLMC without IS for Example \ref{exp:Michaelis–Menten enzyme kinetics}.}
	\label{fig:MLMC_exp4_X_3}
\end{figure}
\FloatBarrier
\begin{figure}[h!]
	\centering
	\begin{subfigure}{.48\textwidth}
		\centering
		\includegraphics[width=1\linewidth]{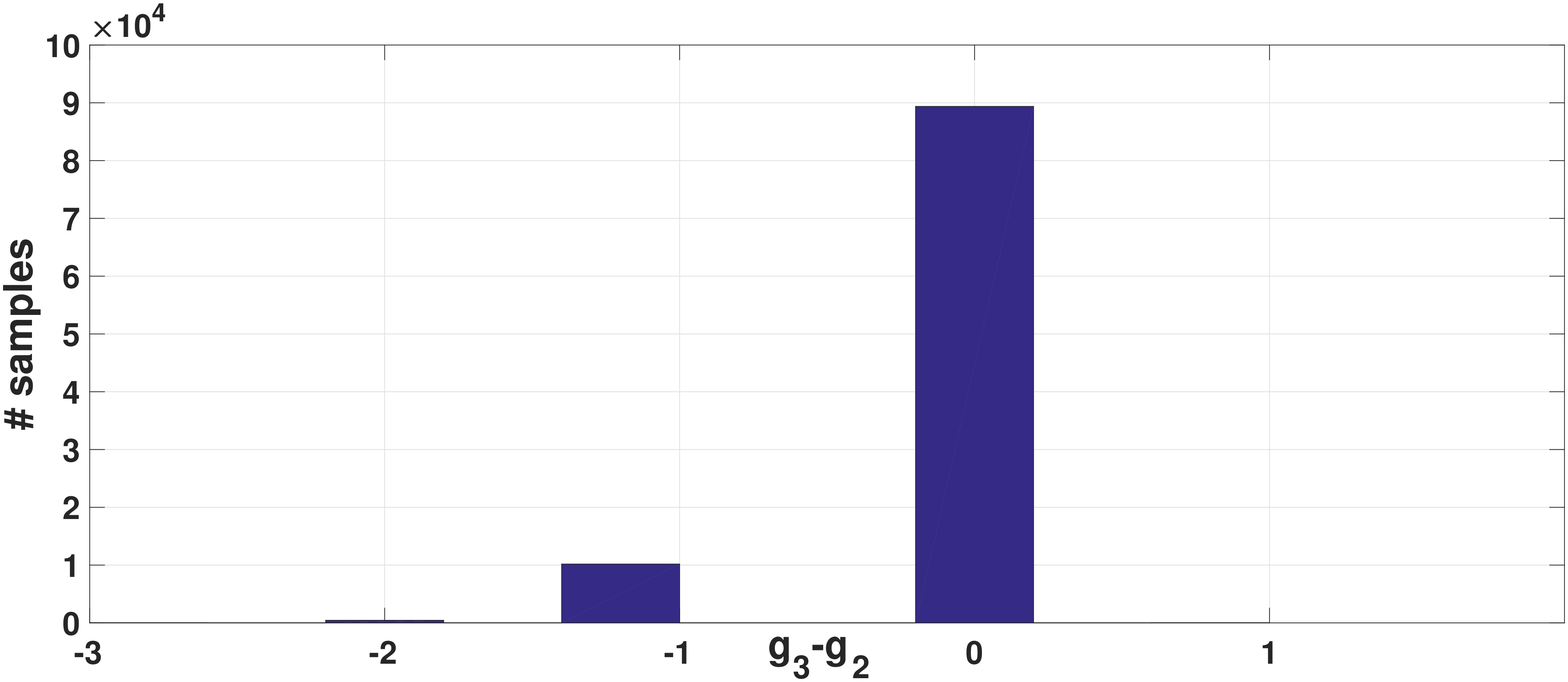}
		\caption{}
		\label{fig:sub3}
	\end{subfigure}\hspace{0.1em}%
	\begin{subfigure}{.48\textwidth}
		\centering
		\includegraphics[width=1\linewidth]{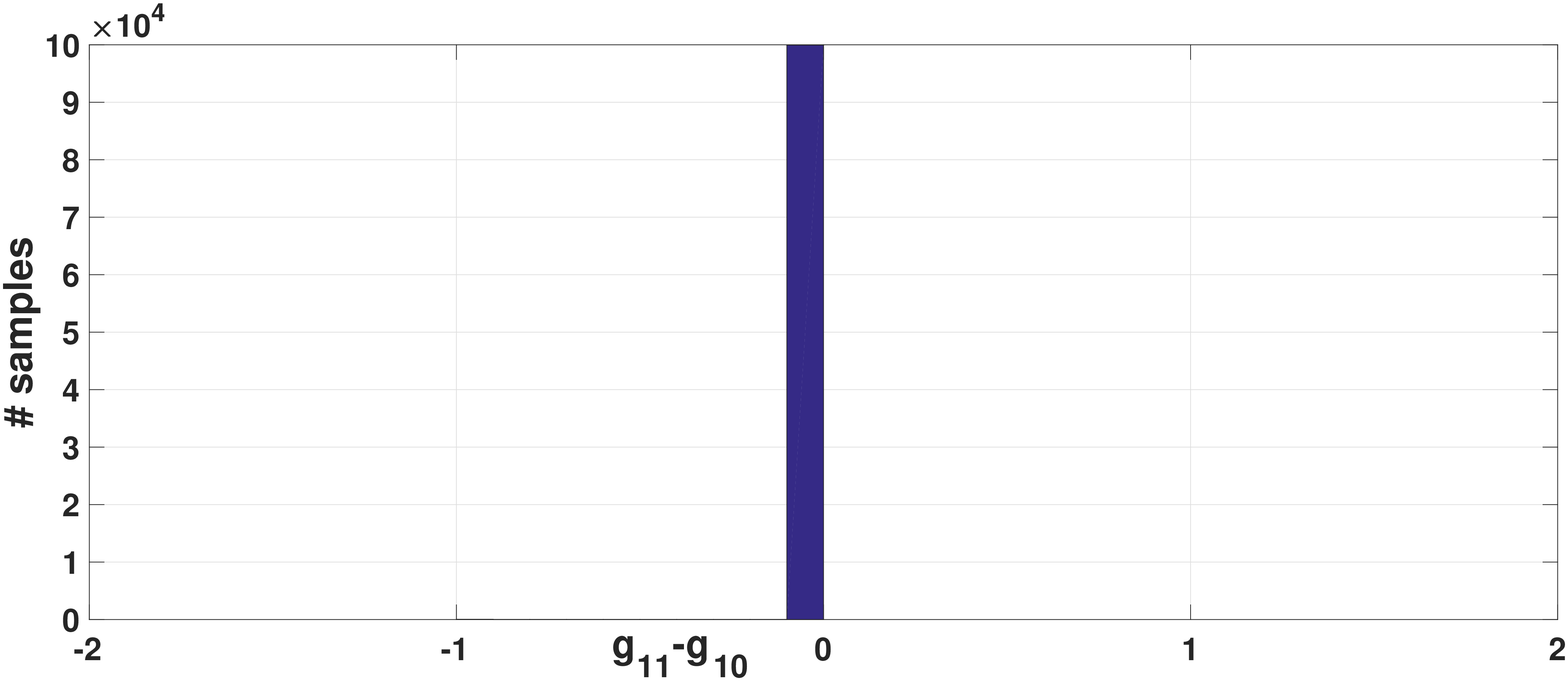}
		\caption{}
		\label{fig:sub4}
	\end{subfigure}
	\caption{Example \ref{exp:Michaelis–Menten enzyme kinetics} without IS: Histogram of  $g_{\ell}-g_{\ell-1}$ ($g_{\ell}=\overline{X}^{(3)}_{\ell}(T)$), for number of samples $M_{\ell}=10^5$. The proportion of samples $\{g_{\ell}-g_{\ell-1}=0\}$ is an increasing function of the level, $\ell$,  of the MLMC estimator, to reach almost $100\%$ for $\ell=11$. a) $\ell=3$. b)  $\ell=11$.}
	\label{fig:catastrophic_coupling_illustration_example4_X_3}
\end{figure}
\FloatBarrier
\subsection{Numerical Results of MLMC With IS}\label{sec: Numerical Results of MLMC With IS}
The MLMC estimator in combination with IS reduces the kurtosis significantly and improves the strong convergence rate from $1$ to $1+\delta$, as illustrated by Figures  \ref{fig: MLMC estimator_imp_sampling choice_1delta_05_exp1}, and  \ref{fig: MLMC estimator_imp_sampling choice_1delta_075_exp1} for Example \ref{exp:decay}, Figures  \ref{fig: MLMC estimator_imp_sampling choice_1delta_05_exp2}, \ref{fig: MLMC estimator_imp_sampling choice_1delta_075_exp2} for Example \ref{exp:Gene transcription and translation}, and Figures  \ref{fig: MLMC estimator_imp_sampling choice_1delta_05_exp3}, and  \ref{fig: MLMC estimator_imp_sampling choice_1delta_075_exp3} for Example \ref{exp:Michaelis–Menten enzyme kinetics}.  The notable reduction of the kurtosis is mainly due to the small reduction of  the proportion of identical terminal values, $g_{\ell}$ and $g_{\ell-1}$, after using IS, as can be seen in Figures  \ref{fig:catastrophic_coupling_illustration_decay_X_1_imp_sampling}, \ref{fig:catastrophic_coupling_illustration_exampl2_X_1_imp_sampling} and \ref{fig:catastrophic_coupling_illustration_exampl3_X_3_imp_sampling}.
\FloatBarrier
\begin{figure}[h!]
	\centering
	\begin{subfigure}{.48\textwidth}
		\centering
		\includegraphics[width=1\linewidth]{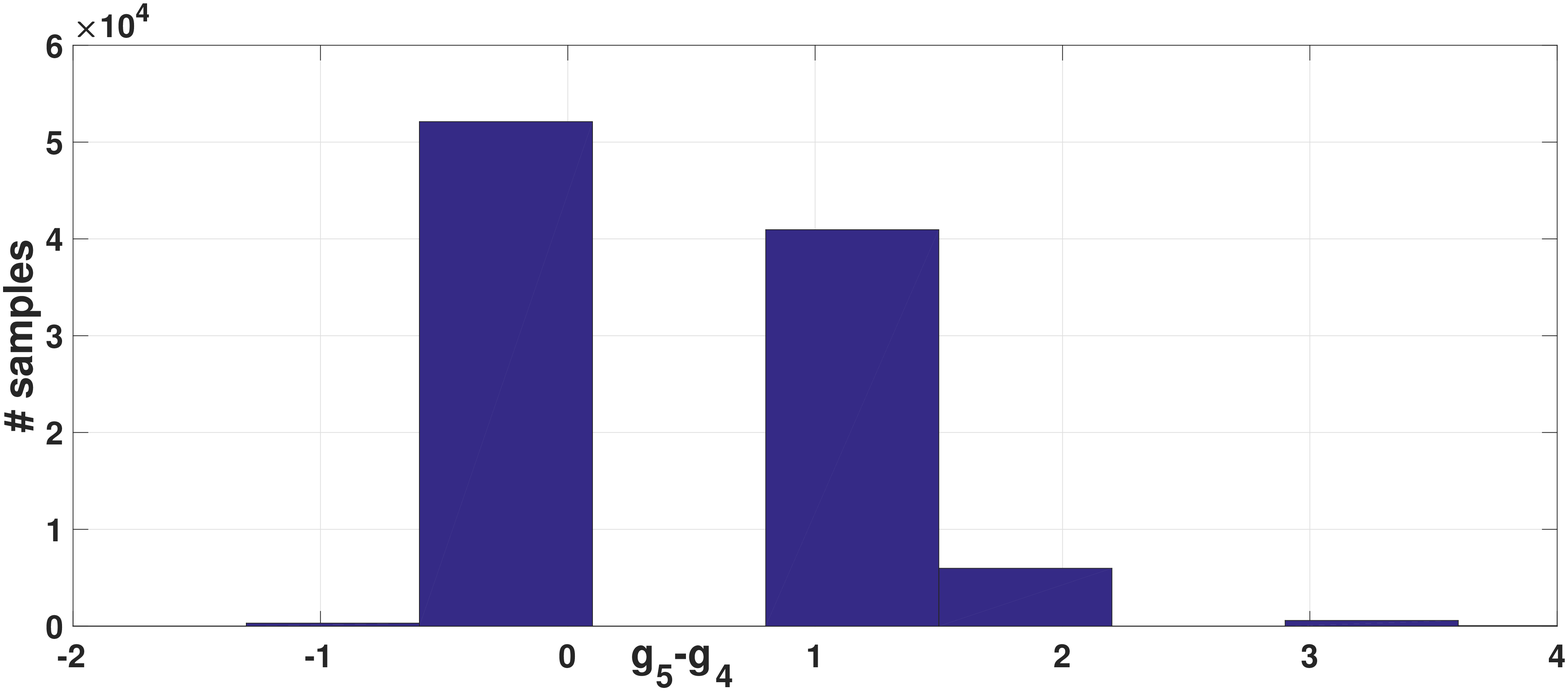}
		\caption{}
		\label{fig:sub3}
	\end{subfigure}%
	\begin{subfigure}{.48\textwidth}
		\centering
		\includegraphics[width=1\linewidth]{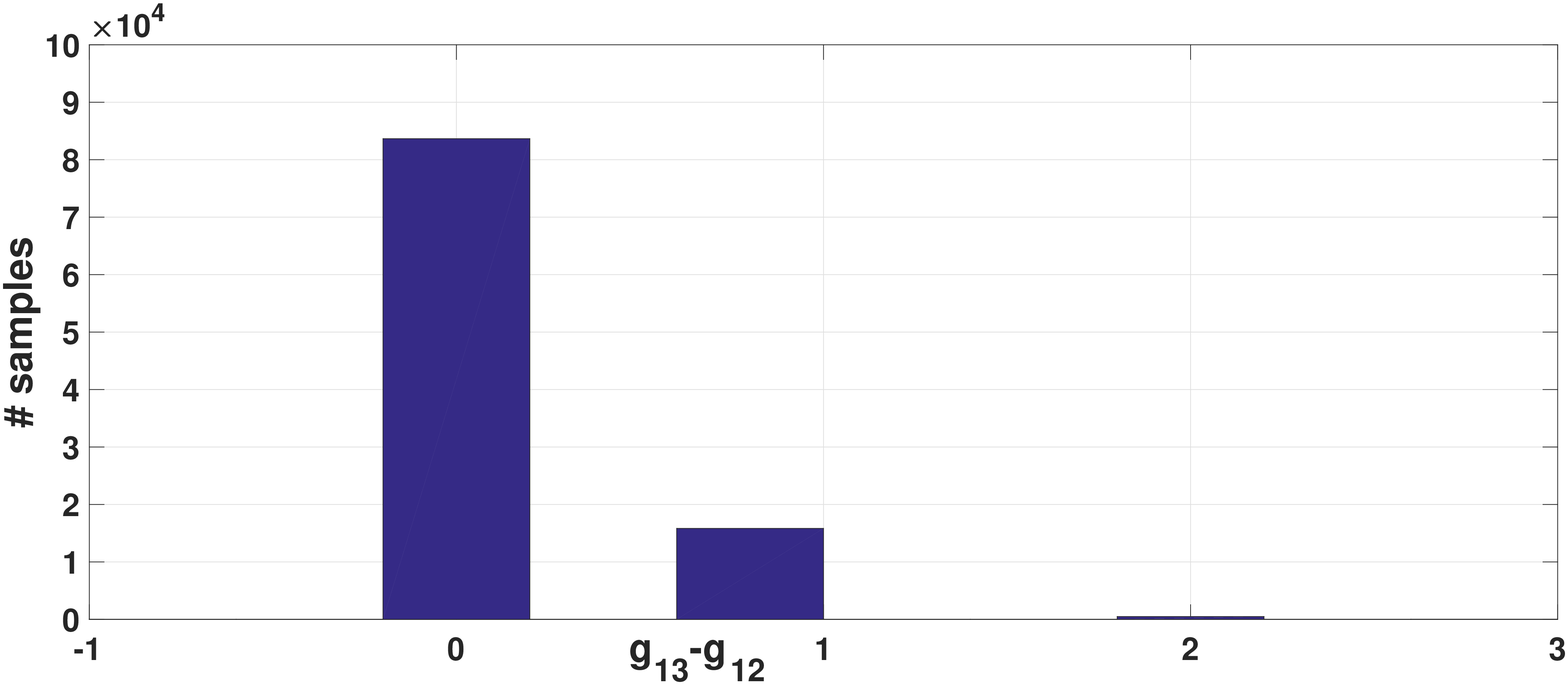}
		\caption{}
		\label{fig:sub4}
	\end{subfigure}
	\caption{Example \ref{exp:decay} with IS  ($\delta=\frac{3}{4}$): Histogram of  $g_{\ell}-g_{\ell-1}$ ($g_{\ell}=\overline{X}_{\ell}(T)$),  for number of samples $M_{\ell}=10^5$. Our IS  reduces the the proportion of samples $\{g_{\ell}-g_{\ell-1}=0\}$ (compared to the case without  IS;  see Figure \ref{fig:catastrophic_coupling_illustration_decay_X_1}) to reach around  $80\%$ for $\ell=13$. a) $\ell=5$. b)  $\ell=13$. }
	\label{fig:catastrophic_coupling_illustration_decay_X_1_imp_sampling}
\end{figure}
	\begin{figure}[h!]
\centering
\includegraphics[width=1\linewidth]{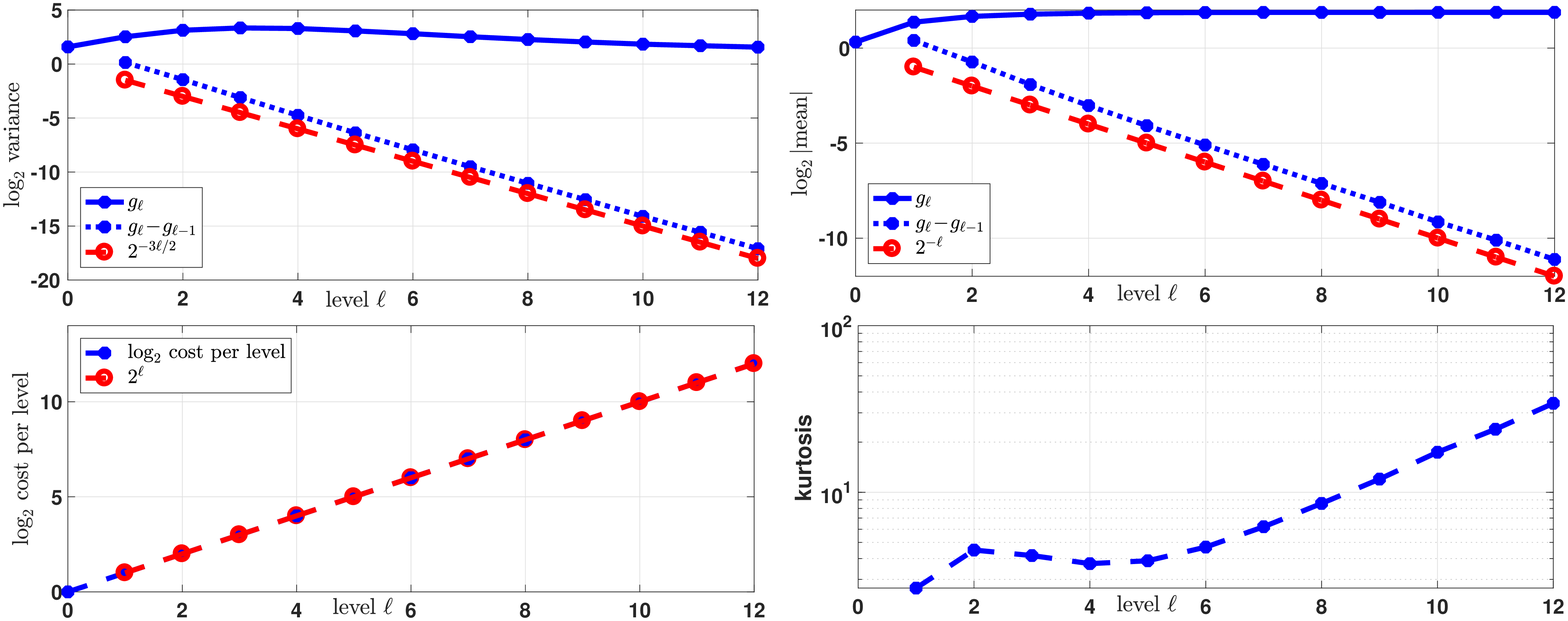}
\caption{Convergence plots of MLMC  with  IS ($\delta=1/2$) for  Example \ref{exp:decay}.}
\label{fig: MLMC estimator_imp_sampling choice_1delta_05_exp1}
\end{figure}

\FloatBarrier

	\begin{figure}[h!]
\centering
\includegraphics[width=1\linewidth]{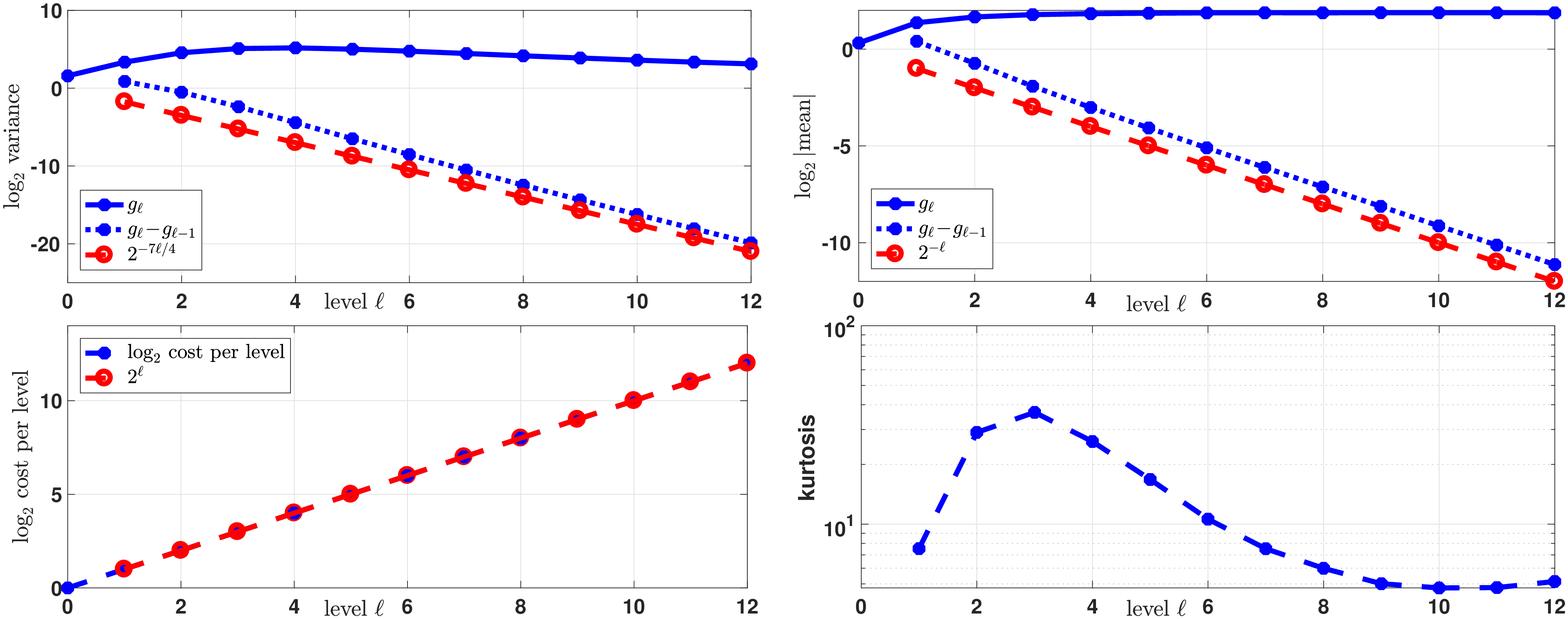}
\caption{Convergence plots of MLMC  with IS  ($\delta=3/4$) for  Example \ref{exp:decay}.}
\label{fig: MLMC estimator_imp_sampling choice_1delta_075_exp1}
\end{figure}
\FloatBarrier
\begin{figure}[h!]
	\centering
	\begin{subfigure}{.48\textwidth}
		\centering
		\includegraphics[width=1\linewidth]{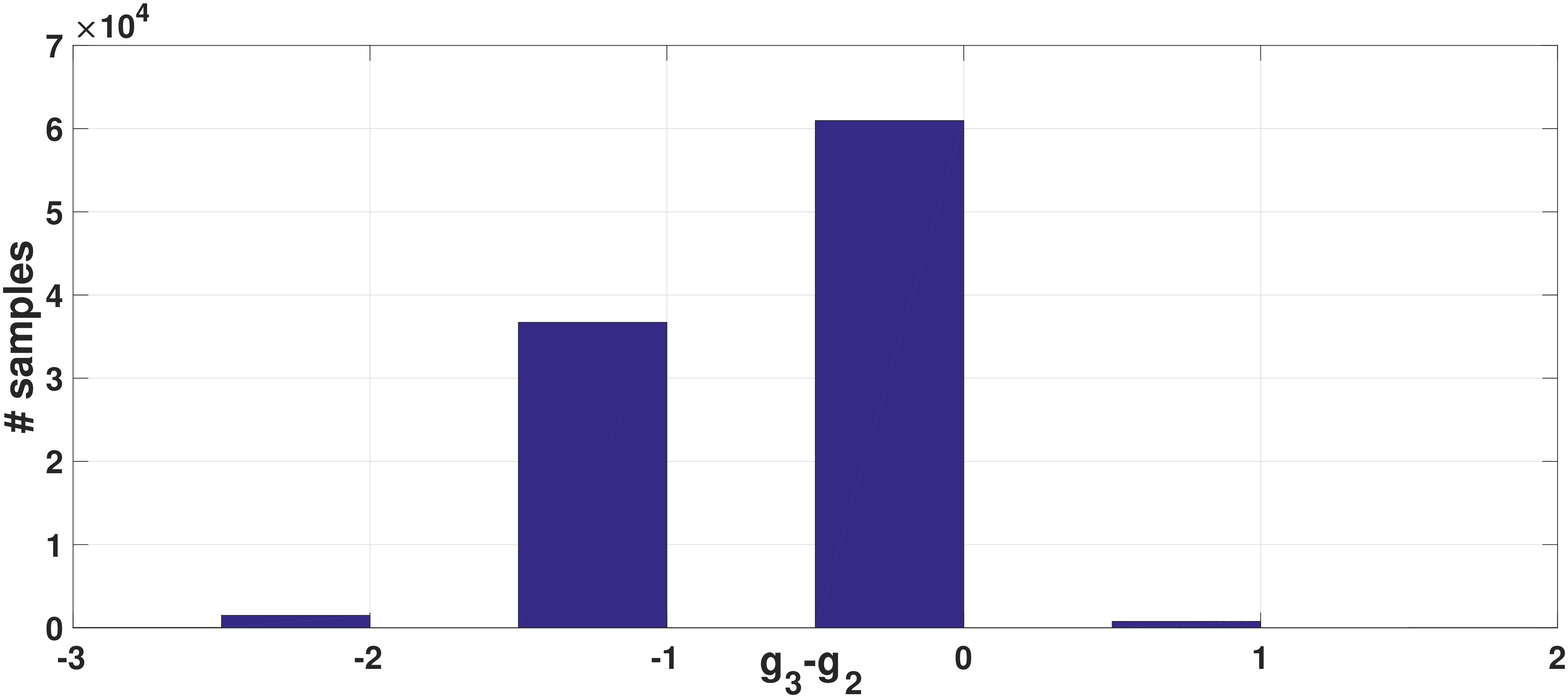}
		\caption{}
		\label{fig:sub3}
	\end{subfigure}%
	\begin{subfigure}{.48\textwidth}
		\centering
		\includegraphics[width=1\linewidth]{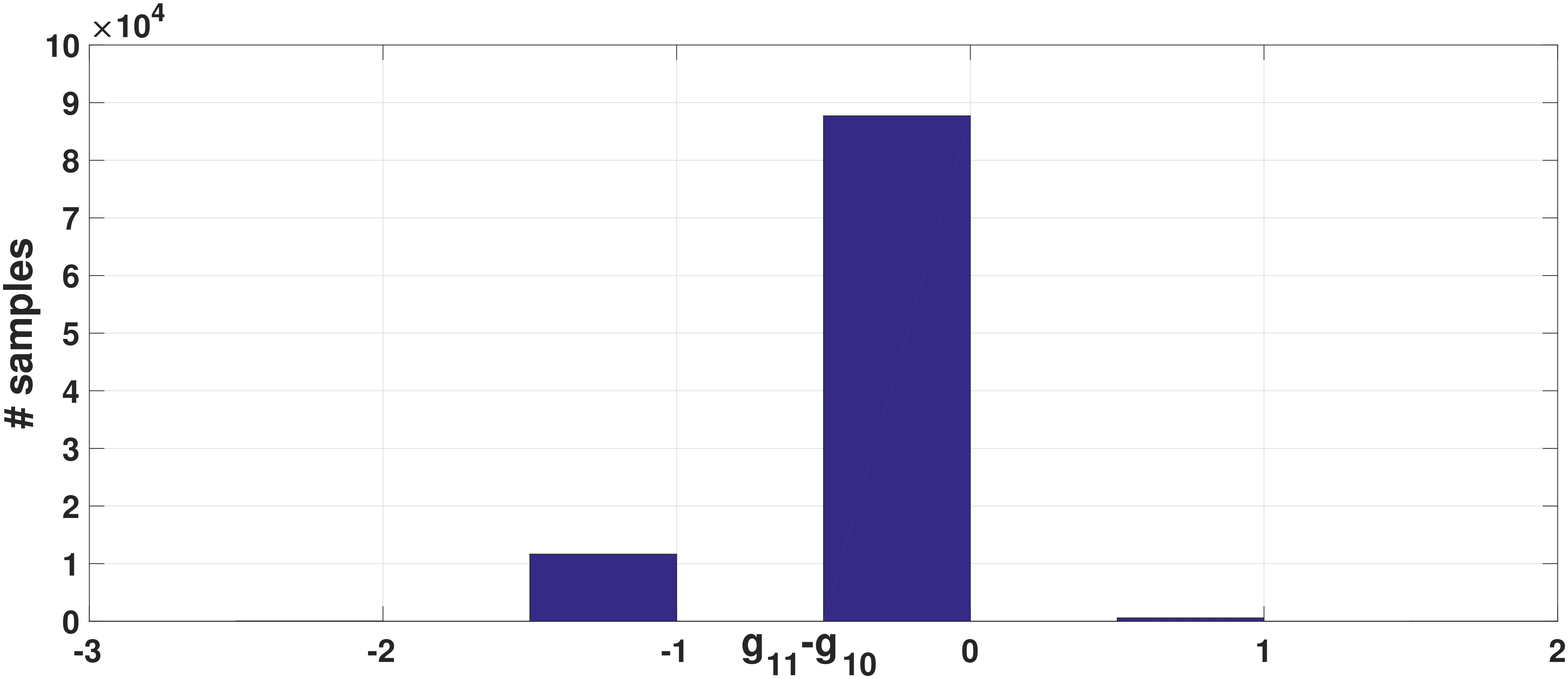}
		\caption{}
		\label{fig:sub4}
	\end{subfigure}
	\caption{Example  \ref{exp:Gene transcription and translation} with IS ($\delta=\frac{3}{4}$): Histogram of  $g_{\ell}-g_{\ell-1}$ ($g_{\ell}=\overline{X}^{(1)}_{\ell}(T)$),for number of samples $M_{\ell}=10^5$. Our IS  reduces the the proportion of samples $\{g_{\ell}-g_{\ell-1}=0\}$ (compared to the case without  IS;  see Figure \ref{fig:catastrophic_coupling_illustration_example2_X_1}) to reach around  $90\%$ for $\ell=11$.  a) $\ell=3$. b)  $\ell=11$.}
	\label{fig:catastrophic_coupling_illustration_exampl2_X_1_imp_sampling}
\end{figure}
\FloatBarrier
	\begin{figure}[h!]
\centering
\includegraphics[width=1\linewidth]{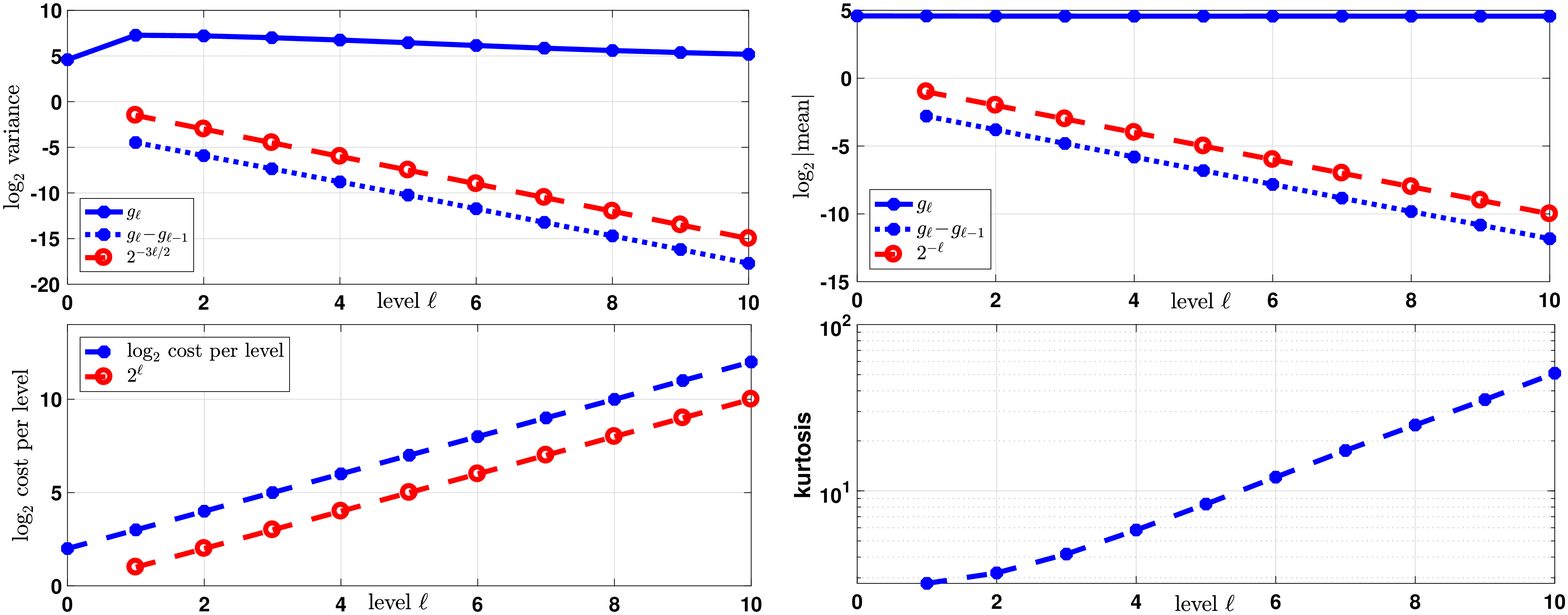}
\caption{Convergence plots of MLMC with IS ($\delta=1/2$) for Example \ref{exp:Gene transcription and translation}.}
\label{fig: MLMC estimator_imp_sampling choice_1delta_05_exp2}
\end{figure}

\FloatBarrier

\begin{figure}[h!]
\centering
\includegraphics[width=1\linewidth]{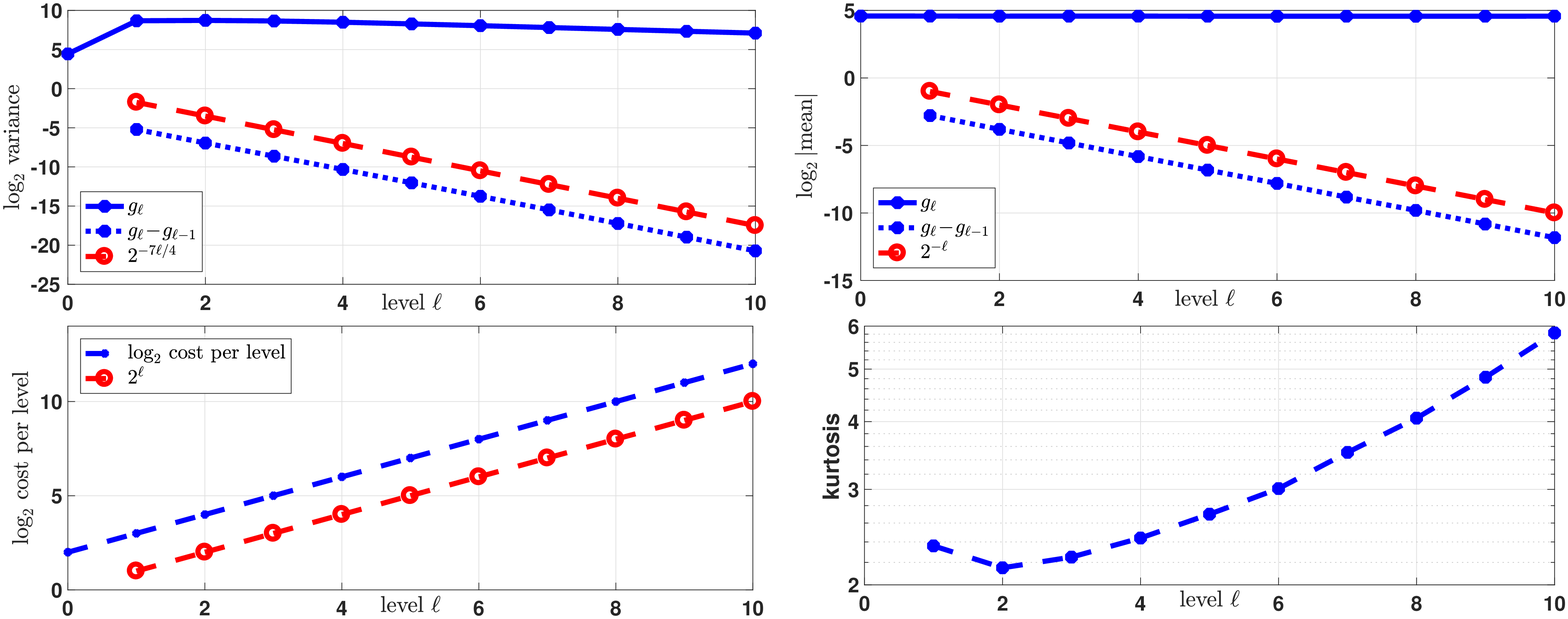}
\caption{Convergence plots  of MLMC with IS ($\delta=3/4$) for Example \ref{exp:Gene transcription and translation}.}
\label{fig: MLMC estimator_imp_sampling choice_1delta_075_exp2}
\end{figure}
\FloatBarrier
\begin{figure}[h!]
	\centering
	\begin{subfigure}{.48\textwidth}
		\centering
		\includegraphics[width=1\linewidth]{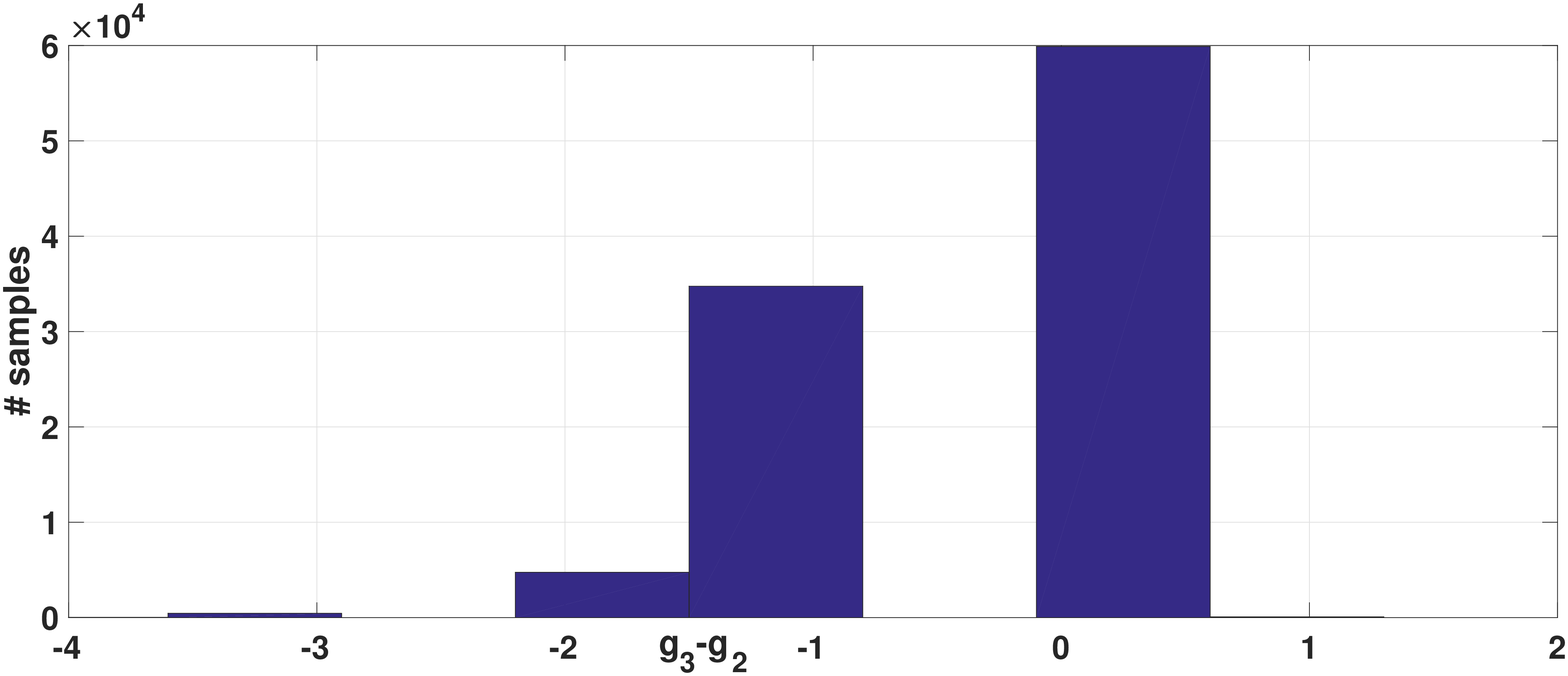}
		\caption{}
		\label{fig:sub3}
	\end{subfigure}%
	\begin{subfigure}{0.48\textwidth}
		\centering
		\includegraphics[width=1\linewidth]{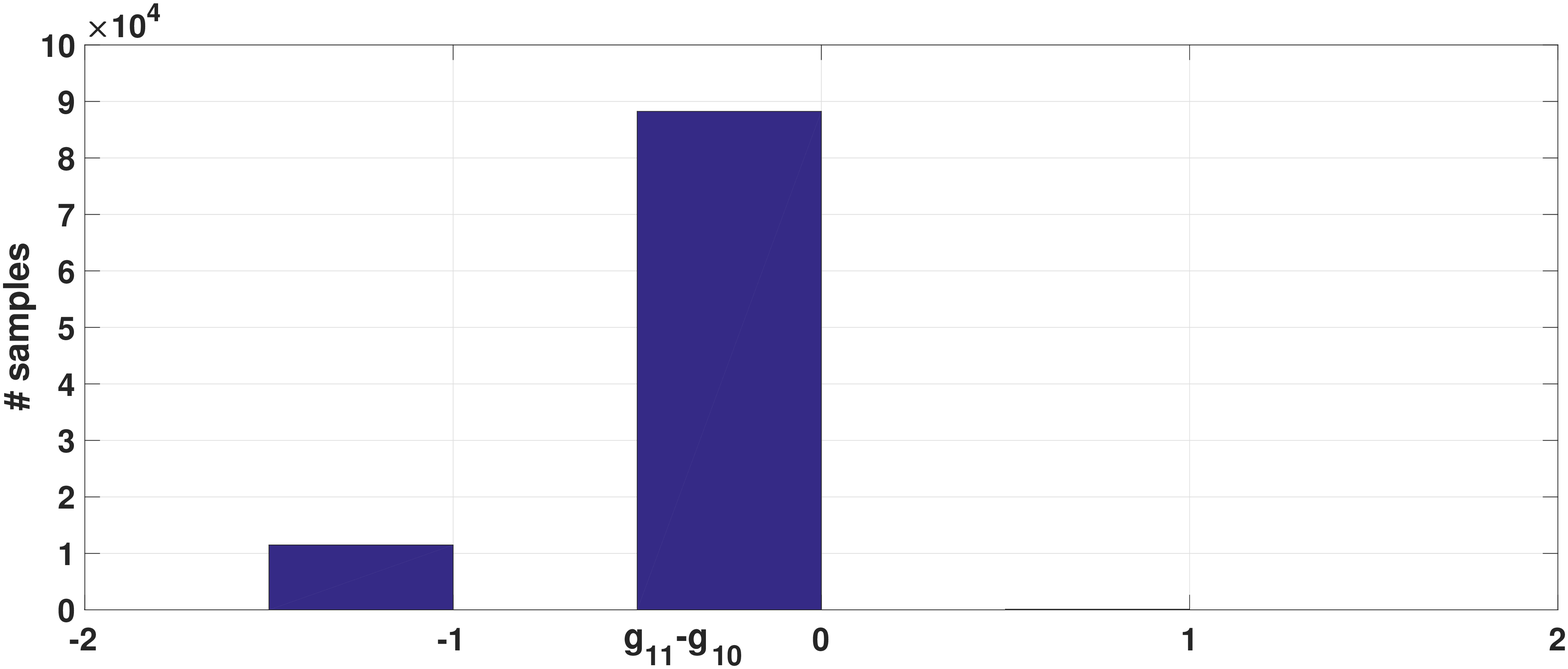}
		\caption{}
		\label{fig:sub4}
	\end{subfigure}
	\caption{Example  \ref{exp:Michaelis–Menten enzyme kinetics} with  IS ($\delta=\frac{3}{4}$): Histogram of  $g_{\ell}-g_{\ell-1}$ ($g_{\ell}=\overline{X}^{(3)}_{\ell}(T)$),  for number of samples $M_{\ell}=10^5$. Our IS reduces the the proportion of samples $\{g_{\ell}-g_{\ell-1}=0\}$ (compared to the case without  IS;  see Figure \ref{fig:catastrophic_coupling_illustration_example4_X_3}) to reach around  $90\%$ for $\ell=11$. a) $\ell=3$. b)  $\ell=11$.}	\label{fig:catastrophic_coupling_illustration_exampl3_X_3_imp_sampling}
\end{figure}
\FloatBarrier
	\begin{figure}[h!]
\centering
\includegraphics[width=1\linewidth]{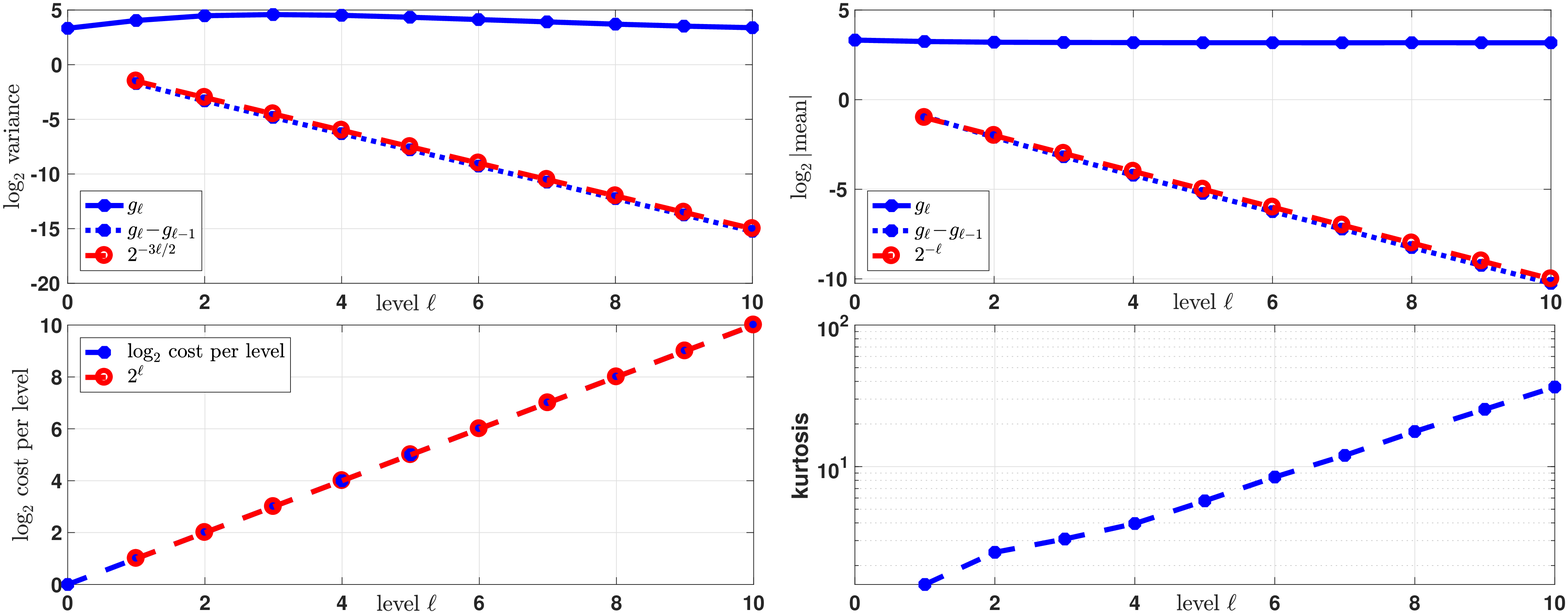}
\caption{Convergence plots  of MLMC with IS ($\delta=1/2$) for Example \ref{exp:Michaelis–Menten enzyme kinetics}.}
\label{fig: MLMC estimator_imp_sampling choice_1delta_05_exp3}
\end{figure}
\FloatBarrier
	\begin{figure}[h!]
\centering
\includegraphics[width=1\linewidth]{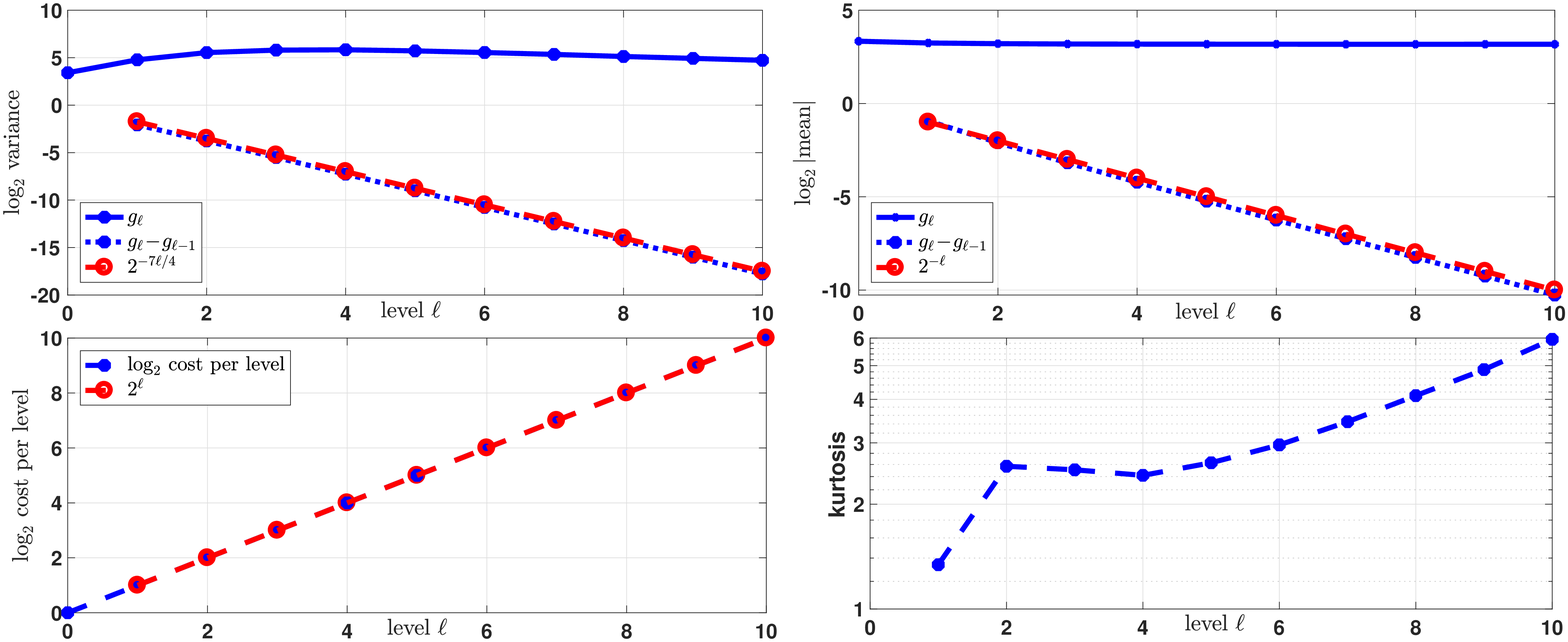}
\caption{Convergence plots  of MLMC with IS ($\delta=3/4$) for Example \ref{exp:Michaelis–Menten enzyme kinetics}.}
\label{fig: MLMC estimator_imp_sampling choice_1delta_075_exp3}
\end{figure}
\FloatBarrier

%% file: appendix.tex
\section{Proofs of Lemma \ref{lemma: moments} and Theorems \ref{thm:kurtosis_high_dim}  and \ref{thm:variance_high_dim}}\label{appendix:Proofs of Lemma}
\begin{proof}[Proof of Lemma \ref{lemma: moments}]\label{proof_lemma_one_dim}
We denote by $K=\sum_{n \in S} k_n$, $L_{\ell}\left(j\right)$ the  likelihood evaluated at $K=j$. Then, for   $p\ge 1$ and $0\le\delta<1$,  and  using relation  \eqref{eq: lik_level_ell}, we  write
\begin{small}
\begin{align}\label{eq:kurt_exapanded1}
\exptpibar{\abs{\Delta g_{\ell}}^p(T) L_{\ell}^p ; \left(\mathcal{F}_{N_{\ell}-1}, \mathcal{I}^s_{\ell}=\mathcal{S}\right)}
&= \sum_{\underset{ i \in \nset \setminus \{0\}}{\abs{\Delta g_{\ell}(T)}=K =i}} i ^p L_{\ell}^p \left(i\right)  \bar{\pi}_{\ell} \left( \abs{\Delta g_{\ell}(T)}=i, K= i ; \left(\mathcal{F}_{N_{\ell}-1}, \mathcal{I}^s_{\ell}=\mathcal{S}\right)\right)\nonumber \\
&+ \sum_{\underset{i \neq j, \: (i,j)  \in \nset^2}{\abs{\Delta g_{\ell}(T)}=i,\: K =j}} i ^p L_{\ell}^p \left(j\right)  \bar{\pi}_{\ell} \left( \abs{\Delta g_{\ell}(T)}=i, K= j ; \left(\mathcal{F}_{N_{\ell}-1}, \mathcal{I}^s_{\ell}=\mathcal{S}\right)\right)\nonumber \\
&= \sum_{\underset{ i \in \nset  \setminus \{0\}}{\abs{\Delta g_{\ell}}=K =i}} \underset{A_{i}}{\underbrace{ i ^p e^{p(\Delta t_{\ell}^{1-\delta} -\Delta t_{\ell})  \sum_{n \in \mathcal{S}} \Delta a_{\ell,n}} \Delta t_{\ell}^{i p \delta}  \bar{\pi}_{\ell} \left( \abs{\Delta g_{\ell}(T)}=i, K= i ; \left(\mathcal{F}_{N_{\ell}-1}, \mathcal{I}^s_{\ell}=\mathcal{S}\right)\right)}}\nonumber \\
&+ \sum_{\underset{\underset{i \neq j, \: (i,j)  \in \nset^2}{K =j}}{\abs{\Delta g_{\ell}(T)}=i}} \underset{B_{ij}}{\underbrace{i^p e^{p(\Delta t_{\ell}^{1-\delta} -\Delta t_{\ell})  \sum_{n \in \mathcal{S}} \Delta a_{\ell,n}} \Delta t_{\ell}^{j p \delta}    \bar{\pi}_{\ell} \left( \abs{\Delta g_{\ell}(T)}=i, K= j ; \left(\mathcal{F}_{N_{\ell}-1}, \mathcal{I}^s_{\ell}=\mathcal{S}\right)\right)}}\nonumber \\
& =\sum_{\underset{ i \in \nset  \setminus \{0\}}{\abs{\Delta g_{\ell}}=K =i}} A_i + \sum_{\underset{i \neq j, \: (i,j)  \in \nset^2}{\abs{\Delta g_{\ell}}=i,\: K =j}}  B_{ij}.
\end{align}
\end{small}
Using Assumption \ref{assumption1_high_dim} (a), we have 
\begin{small}
\begin{equation}\label{eq:A_i_terms}
0 \le \frac{\sum_{\underset{ i \in \nset  \setminus \{0,1\}}{\abs{\Delta g_{\ell}(T)}=K =i}} A_i}{A_1} \le  \underset{\underset{\Delta t_{\ell} \rightarrow 0}{\longrightarrow 0 }}{\underbrace{\sum_{i \in \nset  \setminus \{0,1\}} i ^{p}  \Delta t_{\ell}^{(i-1) p \delta}}}.
\end{equation}
\end{small}
Now, let us examine the second sum in the  right-hand side of \eqref{eq:kurt_exapanded1}. First, observe that $B_{0j}=0, \: \forall j \ge 1$ and $B_{i0}=0, \: \forall i \ge 1$. Although the first observation is clear, we need to explain the second observation, which is mainly due to the fact that  $\bar{\pi}\left( \abs{\Delta g_{\ell}(T)}=i, K= 0 ; \left(\mathcal{F}_{N_{\ell}-1}, \mathcal{I}^s_{\ell}=\mathcal{S}\right)\right)=0$, $\forall i \ge 1$. For the purpose of simplification, let us consider $g_{\ell}=\bar{X}_{\ell}$; then considering the first interval in the coarse level, and  using the coupling equation \eqref{eq: local_error_expression}, we have: i) At $t=0$: $\bar{X}_{\ell}(0)=\bar{X}_{\ell-1}(0)$ and  $\Delta a^1_{\ell-1,0}=0$. ii)  At $t=\Delta t_{\ell}$:  $\bar{X}_{\ell}(\Delta t_{\ell})=\bar{X}_{\ell-1}(\Delta t_{\ell})$  and $\Delta a^2_{\ell-1,0}=a(\bar{X}_{\ell}(\Delta t_{\ell}))-a(\bar{X}_{\ell-1}(0))$. iii) At $t=t_1=2\Delta t_{\ell}$: if $\Delta a^2_{\ell-1,0}=0$, then we simulate this step under the old measure and consequently  we will have $\bar{X}_{\ell}(t_1)=\bar{X}_{\ell-1}(t_1)$  otherwise if $\Delta a^2_{\ell-1,0} \neq 0$, then we  simulate this step under the  IS measure, but since $j=0$, then we will have $\bar{X}_{\ell}(t_1)=\bar{X}_{\ell-1}(t_1)$. Therefore, in both scenarios, we will have the same situation at the start, $t_0=0$. Therefore, we conclude that  $\bar{\pi}\left( \abs{\Delta g_{\ell}(T)}=i, K= 0 ; \left(\mathcal{F}_{N_{\ell}-1}, \mathcal{I}^s_{\ell}=\mathcal{S}\right)\right)=0$, $\forall \: i \ge 1$ and $B_{i0}=0$, $\forall \: i \ge 1$.

Then, using Assumptions \ref{assumption1_high_dim} (b) and  \ref{assumption1_high_dim} (c), we obtain
\begin{small}
\begin{align}\label{eq:B_ij terms}
0 \le \frac{\sum_{\underset{i \neq j, \: (i,j)  \in \nset^2}{\abs{\Delta g_{\ell} (T)}=i,\: K =j}}  B_{ij}}{A_1} & \le \frac{ \sum_{i \neq j, \:  1 \le i,j} i^p \Delta t_{\ell}^{j p\delta } \bar{\pi}_{\ell} \left\lbrace \abs{\Delta g_{\ell} (T)}= i ; \left(\mathcal{F}_{N_{\ell}-1}, \mathcal{I}^s_{\ell}=\mathcal{S}\right)  \right\rbrace}{\left(\Delta t_{\ell}^{p \delta} \right) e^{- \left( \Delta t_{\ell}^{1-\delta}  \Delta a_{\ell,n^\ast} \right)} \left(\Delta t_{\ell}^{1-\delta}  \Delta a_{\ell,n^\ast}\right)  \left(1+\ordo{1}\right)}\nonumber\\
& \le   \frac{ \sum_{i \neq j, \:  1 \le i,j} \eta_{i,\ell} i^p \Delta t_{\ell}^{j p\delta } \Delta t_{\ell}^{i (1-\delta)} }{\left(\Delta t_{\ell}^{p \delta} \right) e^{- \left( \Delta t_{\ell}^{1-\delta}  \Delta a_{\ell,n^\ast} \right)}  \left(\Delta t_{\ell}^{1-\delta} \Delta a_{\ell,n^\ast}\right)  \left(1+\ordo{1}\right)}\nonumber\\
&=    \underset{\underset{\Delta t_{\ell} \rightarrow 0}{\longrightarrow0}} {\underbrace{    \left(1+\ordo{1}\right)^{-1} \left( e^{ \left( \Delta t_{\ell}^{1-\delta}  \Delta a_{\ell,n^\ast} \right)} \Delta a_{\ell,n^\ast}^{-1} \sum_{i \neq j, \:  1 \le i,j} \eta_{i,\ell} i^p \Delta t_{\ell}^{p\delta ( j-1)} \Delta t_{\ell}^{(1-\delta) ( i-1)} \right) }}.
\end{align}
\end{small}
Therefore, using \eqref{eq:kurt_exapanded1}, \eqref{eq:A_i_terms} and \eqref{eq:B_ij terms}, we  conclude  Lemma \ref{lemma: moments}, that is
\begin{small}
\begin{equation*}
\exptpibar{\abs{\Delta g_{\ell}}^p(T) L_{\ell}^p ; \left(\mathcal{F}_{N_{\ell}-1}, \mathcal{I}^s_{\ell}=\mathcal{S}\right)}=  \Delta t_{\ell}^{(p -1)\delta+1}  \left(\Delta a_{\ell,n^\ast}\right)  e^{p(\Delta t_{\ell}^{1-\delta} -\Delta t_{\ell})  \sum_{n \in \mathcal{S}} \Delta a_{\ell,n}}  e^{- \left( \Delta t_{\ell}^{1-\delta}  \Delta a_{\ell,n^\ast} \right)}  \left(1+h_{p,\ell}\right),
\end{equation*}
\end{small}
such that $h_{p,\ell} \underset{\Delta t_{\ell} \rightarrow 0}{\longrightarrow 0}$.
\end{proof}
\begin{proof}[Proof of Theorem \ref{thm:kurtosis_high_dim}]\label{proof:kurtosis_one_dim}
Let $0 \le  \delta < 1$. In the first step of the proof, we want to show that 
\begin{equation*}
\kappa_{\ell}:=\frac{\exptpibar{\left(Y_{\ell}-\exptpibar{Y_{\ell}}\right)^4}}{\left(\text{Var}_{\bar{\pi}_{\ell}}\left[Y_{\ell}\right]\right)^2} \underset{\Delta t_{\ell} \rightarrow 0}{\sim}  \frac{\exptpibar{Y_{\ell}^4}}{\left(\exptpibar{Y_{\ell}^2}\right)^2} 
\end{equation*}
Let us first show that $\text{Var}_{\bar{\pi}_{\ell}}\left[Y_{\ell}\right]  \underset{\Delta t_{\ell} \rightarrow 0}{\sim} \exptpibar{Y_{\ell}^2}$. In fact,
\begin{align*}
\frac{\text{Var}_{\bar{\pi}_{\ell}}\left[Y_{\ell}\right] }{\exptpibar{Y_{\ell}^2}}=\frac{\exptpibar{Y_{\ell}^2}- \left(\exptpibar{Y_{\ell}}\right)^2}{\exptpibar{Y_{\ell}^2}}=1- \frac{ \left(\exptpibar{Y_{\ell}}\right)^2}{\exptpibar{Y_{\ell}^2}}.
\end{align*}
Therefore,  we need to show that $I_1:=\frac{ \left(\exptpibar{Y_{\ell}}\right)^2}{\exptpibar{Y_{\ell}^2}} \underset{\Delta t_{\ell}\rightarrow 0}\longrightarrow 0$. 

Due to the order one weak error convergence, there exists a constant $d_1>0$ such that  $\left(\exptpibar{Y_{\ell}}\right)\le d_1 \Delta t_{\ell}$. Therefore, using Lemma \ref{lemma: moments} and Assumption \ref{assumption2_high_dim}, we obtain
\begin{small}
\begin{align*}
0 \le I_1\le \frac{d_1^2 \Delta t_{\ell}^2}{\exptpibar{Y_{\ell}^2}}&=  \frac{d_1^2 \Delta t_{\ell}^2}{\exptpibar{\exptpibar{Y_{\ell}^2; \left(\mathcal{F}_{N_{\ell}}, \mathcal{I}^s_{\ell}=\mathcal{S}\right)}}}\\
 &=\frac{d_1^2 \Delta t_{\ell}^2}{\exptpibar{ \left(\Delta t_{\ell}^{\delta+1} \right)  e^{2(\Delta t_{\ell}^{1-\delta} -\Delta t_{\ell})  \sum_{n \in \mathcal{S}} \Delta a_{\ell,n}}  e^{- \left( \Delta t_{\ell}^{1-\delta}  \Delta a_{\ell,n^\ast} \right)} \left(\Delta a_{\ell,n^\ast}\right) \left(1+ h_{2,\ell}\right)}}\\
 &\le \frac{d_1^2 \Delta t_{\ell}^{1-\delta}}{\exptpibar{   e^{- \left( \Delta t_{\ell}^{1-\delta}  \Delta a_{\ell,n^\ast} \right)} \Delta a_{\ell,n^\ast} }}\le  \frac{d_1^2 \Delta t_{\ell}^{1-\delta}}{C_1} \underset{\Delta t_{\ell}\rightarrow 0}\longrightarrow 0.
\end{align*}
\end{small}
 Therefore, we conclude that
\begin{equation}\label{eq:var_equiv}
  \text{Var}_{\bar{\pi}_{\ell}}\left[Y_{\ell}\right]  \underset{\Delta t_{\ell} \rightarrow 0}{\sim} \exptpibar{Y_{\ell}^2}.
\end{equation}  
Now, let us  show that $\exptpibar{\left(Y_{\ell}-\exptpibar{Y_{\ell}}\right)^4} \underset{\Delta t_{\ell} \rightarrow 0}{\sim} \exptpibar{Y_{\ell}^4}$. In fact,
\begin{align*}
\frac{\exptpibar{\left(Y_{\ell}-\exptpibar{Y_{\ell}}\right)^4}}{\exptpibar{Y_{\ell}^4}}&=\frac{\exptpibar{Y_{\ell}^4}-4  \exptpibar{Y_{\ell}^3} \exptpibar{Y_{\ell}}+ 6 \exptpibar{Y_{\ell}^2} \left(\exptpibar{Y_{\ell}}\right)^2 -3  \left(\exptpibar{Y_{\ell}}\right)^4}{\exptpibar{Y_{\ell}^4}}\\
&=1-4\frac{\exptpibar{Y_{\ell}^3} \exptpibar{Y_{\ell}}}{\exptpibar{Y_{\ell}^4}}+ 6\frac{\exptpibar{Y_{\ell}^2} \left(\exptpibar{Y_{\ell}}\right)^2}{\exptpibar{Y_{\ell}^4}}- 3 \frac{\left(\exptpibar{Y_{\ell}}\right)^4}{\exptpibar{Y_{\ell}^4}}\\
&=1-4 I_2+ 6 I_3- 3 I_4
\end{align*}
Therefore,  we need to show that $I_2, I_3, I_4 \underset{\Delta t_{\ell}\rightarrow 0}\longrightarrow 0$. 

Using Lemma \ref{lemma: moments} and Assumption \ref{assumption2_high_dim}, we obtain
\begin{small}
\begin{align*}
0 \le I_4\le \frac{d_1^4 \Delta t_{\ell}^4}{\exptpibar{Y_{\ell}^4}}&=  \frac{d_1^4   \Delta t_{\ell}^4}{\exptpibar{\exptpibar{Y_{\ell}^4; \left(\mathcal{F}_{N_{\ell}}, \mathcal{I}^s_{\ell}=\mathcal{S}\right)}}}\\
 &=\frac{d_1^4 \Delta t_{\ell}^4}{\exptpibar{\left(\Delta t_{\ell}^{3\delta+1} \right) e^{4(\Delta t_{\ell}^{1-\delta} -\Delta t_{\ell})  \sum_{n \in \mathcal{S}} \Delta a_{\ell,n}}  e^{- \left( \Delta t_{\ell}^{1-\delta}  \Delta a_{\ell,n^\ast} \right)} \left( \Delta a_{\ell,n^\ast}\right) \left(1+ h_{4,\ell}\right)}}\\
 &\le \frac{d_1^4 \Delta t_{\ell}^{3(1-\delta)}}{\exptpibar{   e^{- \left( \Delta t_{\ell}^{1-\delta}  \Delta a_{\ell,n^\ast} \right)}   \Delta a_{\ell,n^\ast} }}\le \frac{d_1^4 \Delta t_{\ell}^{3(1-\delta)}}{C_1}\underset{\Delta t_{\ell}\rightarrow 0}\longrightarrow 0.
\end{align*}
\end{small}
Similarly for $I_2$, using Lemma \ref{lemma: moments} and Assumptions \ref{assumption2_high_dim} and \ref{assump:integrability_assump}, we obtain
\begin{small}
\begin{align*}
0 \le I_2\le \frac{d_1 \Delta t_{\ell} \exptpibar{\abs{Y_{\ell}}^3} }{\exptpibar{Y_{\ell}^4}}&=  \frac{d_1 \Delta t_{\ell} \exptpibar{\exptpibar{Y_{\ell}^3; \left(\mathcal{F}_{N_{\ell}}, \mathcal{I}^s_{\ell}=\mathcal{S}\right)}}}{\exptpibar{\exptpibar{Y_{\ell}^4; \left(\mathcal{F}_{N_{\ell}}, \mathcal{I}^s_{\ell}=\mathcal{S}\right)}}}\\
 &=\frac{d_1 \Delta t_{\ell}  \exptpibar{\left(\Delta t_{\ell}^{2 \delta+1} \right)  e^{3(\Delta t_{\ell}^{1-\delta} -\Delta t_{\ell})  \sum_{n \in \mathcal{S}} \Delta a_{\ell,n}} e^{- \left( \Delta t_{\ell}^{1-\delta}  \Delta a_{\ell,n^\ast} \right)} \left( \Delta a_{\ell,n^\ast}\right) \left(1+ h_{3,\ell}\right)}}{\exptpibar{  \left(\Delta t_{\ell}^{3 \delta+1} \right) e^{4(\Delta t_{\ell}^{1-\delta} -\Delta t_{\ell})  \sum_{n \in \mathcal{S}} \Delta a_{\ell,n}} e^{- \left(  \Delta a_{\ell,n^\ast} \right)} \left(\Delta t_{\ell}^{1-\delta}  \Delta a_{\ell,n^\ast}\right) \left(1+h_{4,\ell}\right)}}\\
 &\le \frac{d_1 \Delta t_{\ell}^{1-\delta}  \exptpibar{ e^{3(\Delta t_{\ell}^{1-\delta} -\Delta t_{\ell})  \sum_{n \in \mathcal{S}} \Delta a_{\ell,n}}   \Delta a_{\ell,n^\ast}\left(1+ h_{3,\ell}\right)}}{\exptpibar{   e^{- \left( \Delta t_{\ell}^{1-\delta}  \Delta a_{\ell,n^\ast} \right)}   \Delta a_{\ell,n^\ast} }}\\
 &\le C \Delta t_{\ell}^{1-\delta} \underset{\Delta t_{\ell}\rightarrow 0}{\longrightarrow} 0.
\end{align*}
\end{small}
Finally, for $I_3$, using Lemma \ref{lemma: moments} and Assumptions \ref{assumption2_high_dim} and \ref{assump:integrability_assump}, we obtain
\begin{small}
\begin{align*}
0 \le I_3\le \frac{d_1^2 \Delta t_{\ell}^2 \exptpibar{\abs{Y_{\ell}}^2} }{\exptpibar{Y_{\ell}^4}}&=  \frac{d_1^2 \Delta t_{\ell}^2 \exptpibar{\exptpibar{Y_{\ell}^2; \left(\mathcal{F}_{N_{\ell}}, \mathcal{I}^s_{\ell}=\mathcal{S}\right)}}}{\exptpibar{\exptpibar{Y_{\ell}^4; \left(\mathcal{F}_{N_{\ell}}, \mathcal{I}^s_{\ell}=\mathcal{S}\right)}}}\\
 &=\frac{d_1^2 \Delta t_{\ell}^2  \exptpibar{ \left(\Delta t_{\ell}^{ \delta+1} \right) e^{2(\Delta t_{\ell}^{1-\delta} -\Delta t_{\ell})  \sum_{n \in \mathcal{S}} \Delta a_{\ell,n}}  e^{- \left( \Delta t_{\ell}^{1-\delta}  \Delta a_{\ell,n^\ast} \right)} \left(  \Delta a_{\ell,n^\ast}\right) \left(1+ h_{2,\ell}\right)}}{\exptpibar{\left(\Delta t_{\ell}^{3 \delta+1} \right) e^{4(\Delta t_{\ell}^{1-\delta} -\Delta t_{\ell})  \sum_{n \in \mathcal{S}} \Delta a_{\ell,n}}  e^{- \left( \Delta t_{\ell}^{1-\delta}  \Delta a_{\ell,n^\ast} \right)} \left(  \Delta a_{\ell,n^\ast}\right) \left(1+ h_{4,\ell}\right)}}\\
 &\le \frac{d_1^2 \Delta t_{\ell}^{2(1-\delta)} \exptpibar{ e^{2(\Delta t_{\ell}^{1-\delta} -\Delta t_{\ell})  \sum_{n \in \mathcal{S}} \Delta a_{\ell,n}}   \Delta a_{\ell,n^\ast} \left(1+ h_{2,\ell}\right)} }{\exptpibar{  e^{- \left( \Delta t_{\ell}^{1-\delta}  \Delta a_{\ell,n^\ast} \right)}  \Delta a_{\ell,n^\ast} }}\\
 &\le \widetilde{C} \Delta t_{\ell}^{2(1-\delta)}  \underset{\Delta t_{\ell}\rightarrow 0}{\longrightarrow} 0.
\end{align*}
\end{small}
Therefore, we conclude that
\begin{equation}\label{eq:forth_moment_equiv}
 \exptpibar{\left(Y_{\ell}-\exptpibar{Y_{\ell}}\right)^4} \underset{\Delta t_{\ell} \rightarrow 0}{\sim} \exptpibar{Y_{\ell}^4}.
\end{equation}
Finally, using \eqref{eq:var_equiv}, \eqref{eq:forth_moment_equiv}, Lemma \ref{lemma: moments} and Assumptions \ref{assumption2_high_dim} and \ref{assump:integrability_assump}, we obtain 
\begin{small}
\begin{align}\label{eq:kurtosis_equiv}
\kappa_{\ell}:=\frac{ \exptpibar{\left(Y_{\ell}-\expt{Y_{\ell}}\right)^4}}{\left(\text{Var}_{\bar{\pi}_{\ell}}\left[Y_{\ell}\right]\right)^2} &\underset{\Delta t_{\ell} \rightarrow 0}{\sim}  \frac{ \exptpibar{Y_{\ell}^4}}{\left( \exptpibar{Y_{\ell}^2}\right)^2}\nonumber\\
 &=\frac{\exptpibar{ \exptpibar{Y_{\ell}^4(T)  ; \left(\mathcal{F}_{N_{\ell}}, \mathcal{I}^s_{\ell}=\mathcal{S}\right)}}}{\left( \exptpibar{\exptpibar{Y_{\ell}^2(T)  ; \left(\mathcal{F}_{N_{\ell}}, \mathcal{I}^s_{\ell}=\mathcal{S}\right)}}\right)^2} \nonumber\\
&= \Delta t_{\ell}^{\delta-1} \frac{\exptpibar{ e^{4(\Delta t_{\ell}^{1-\delta} -\Delta t_{\ell})  \sum_{n \in \mathcal{S}} \Delta a_{\ell,n}}  e^{- \left( \Delta t_{\ell}^{1-\delta}  \Delta a_{\ell,n^\ast} \right)}   \Delta a_{\ell,n^\ast} \left(1+ h_{4,\ell}\right)}}{\left(\exptpibar{ e^{2(\Delta t_{\ell}^{1-\delta} -\Delta t_{\ell})  \sum_{n \in \mathcal{S}} \Delta a_{\ell,n}}  e^{- \left( \Delta t_{\ell}^{1-\delta}  \Delta a_{\ell,n^\ast} \right)}  \Delta a_{\ell,n^\ast} \left(1+ h_{2,\ell}\right)}\right)^2} \nonumber\\
&=\Ordo{\Delta t_{\ell}^{\delta-1}}.
\end{align}
\end{small}
\end{proof}
\begin{proof}[Proof of Theorem \ref{thm:variance_high_dim}]
Let $0 <  \delta < 1$. Then, using \eqref{eq:var_equiv}, Lemma \ref{lemma: moments} and Assumptions \ref{assumption2_high_dim} and \ref{assump:integrability_assump}, we obtain  
\begin{align}\label{eq:var_cond}
\text{Var}_{\bar{\pi}_{\ell}}\left[ Y_{\ell} \right]  \underset{\Delta t_{\ell} \rightarrow 0}{\sim} \exptpibar{Y_{\ell}^2}&= \exptpibar{\exptpibar{ Y_{\ell}^2; \left(\mathcal{F}_{N_{\ell}-1}, \mathcal{I}^s_{\ell}=\mathcal{S}\right)}}  \nonumber\\
&=\exptpibar{ \left(\Delta t_{\ell}^{\delta+1} \right)  e^{2(\Delta t_{\ell}^{1-\delta} -\Delta t_{\ell})  \sum_{n \in \mathcal{S}} \Delta a_{\ell,n}}  e^{- \left( \Delta t_{\ell}^{1-\delta}  \Delta a_{\ell,n^\ast} \right)} \left(\Delta a_{\ell,n^\ast}\right) \left(1+ h_{2,\ell}\right)}\nonumber\\
&=\Ordo{\Delta t_{\ell}^{1+\delta}}.
\end{align}
\end{proof}
\section{Numerical Evidence of   Assumption  \ref{assumption1_high_dim}}\label{appendix:Numerical evidence of   Assumptions}
In Figures \ref{fig:number_k_n>0_decay}, \ref{fig:number_k_n>0_example2} and \ref{fig:number_k_n>0_example4}, we plot the histograms, for Examples \ref{exp:decay}, \ref{exp:Gene transcription and translation} and \ref{exp:Michaelis–Menten enzyme kinetics}, with $\delta=0.5$,  corresponding to   $\# \{ \text{IS steps : s.t.}\:  \bar{K}:=\sum_{j \in \mathcal{J}_1} \sum_{n \in \mathcal{S}_j}k^j_n>0\}$, that is the number of times where we perform IS and  succeeded to  separate the two paths. These Figures show that our assumption \ref{assumption1_high_dim} (c)  is valid since for small values of $\Delta t_{\ell}$, we have at most one jump created by IS such that it separates the two paths. 

\FloatBarrier
\begin{figure}[h!]
	\centering
	\begin{subfigure}{0.4\textwidth}
		\centering
		\includegraphics[width=0.9\linewidth]{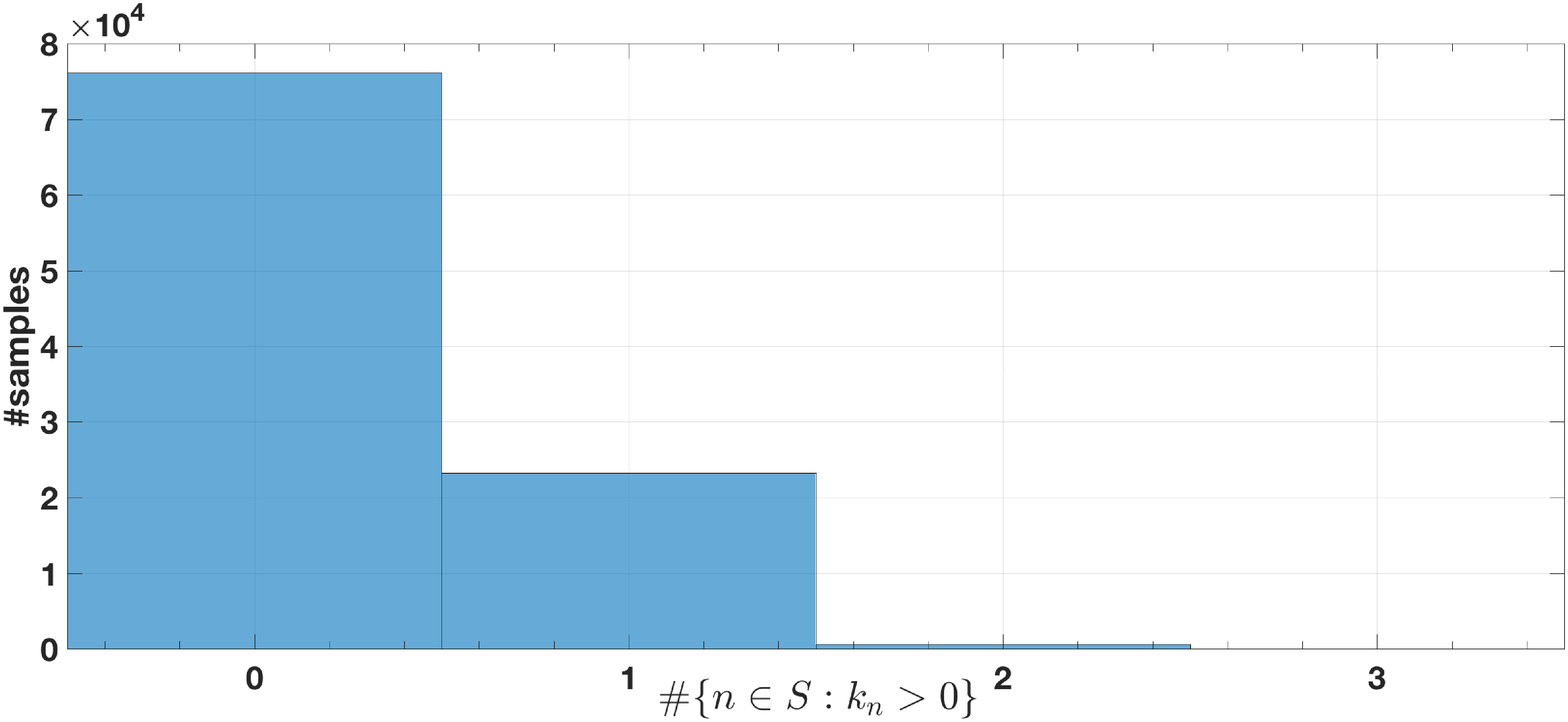}
		\caption{}
		\label{fig:sub4}
	\end{subfigure}
	\begin{subfigure}{0.4\textwidth}
		\centering
		\includegraphics[width=0.9\linewidth]{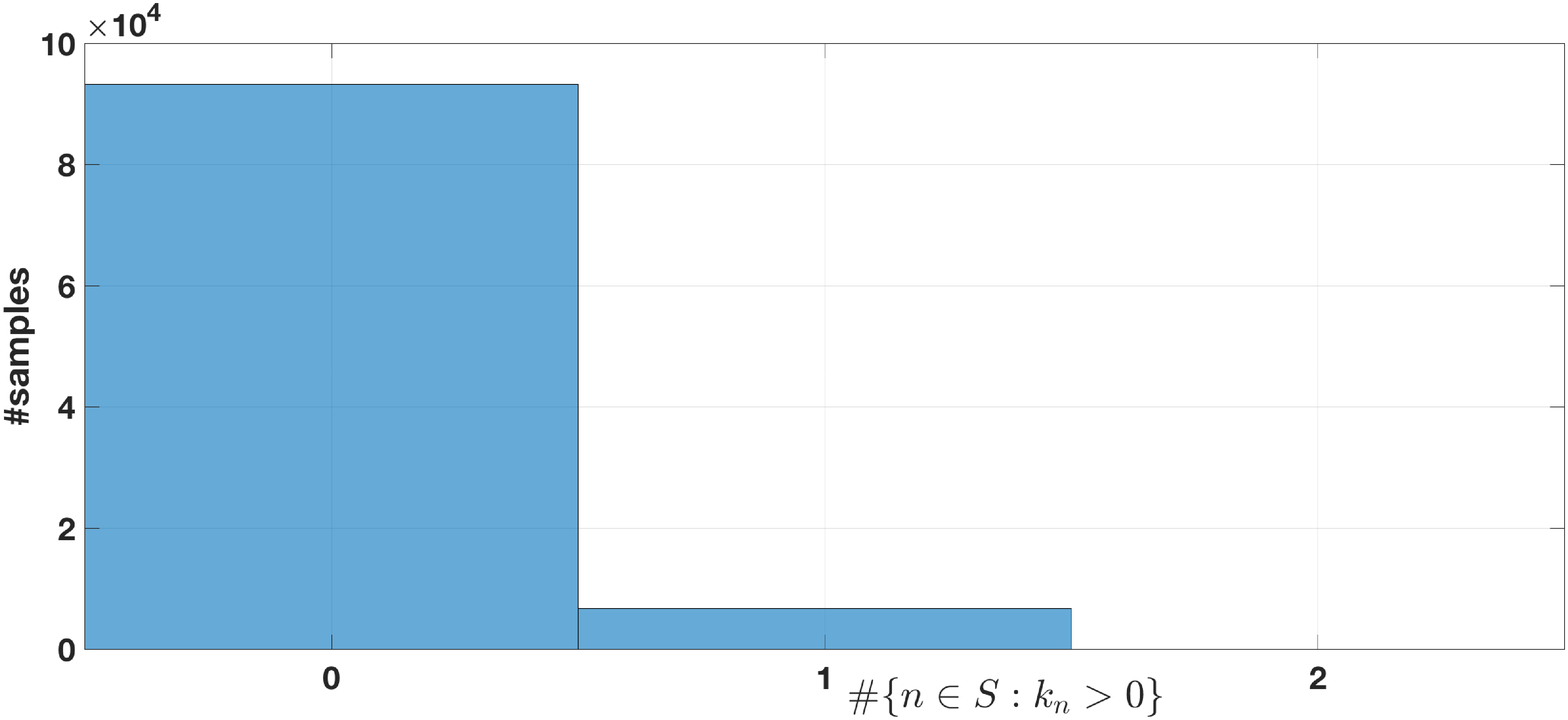}
		\caption{}
		\label{fig:sub4}
	\end{subfigure}
	\caption{ Example \ref{exp:decay} with IS (with $\delta=0.5$): Histogram of  $\# \{n \in \mathcal{S}: k_n>0\}$, for number of samples $M_{\ell}=10^5$. a)  $\ell=6$. b) $\ell=10$ }
	\label{fig:number_k_n>0_decay}
\end{figure}

\FloatBarrier
\begin{figure}[h!]
	\centering
	\begin{subfigure}{0.4\textwidth}
		\centering
		\includegraphics[width=0.9\linewidth]{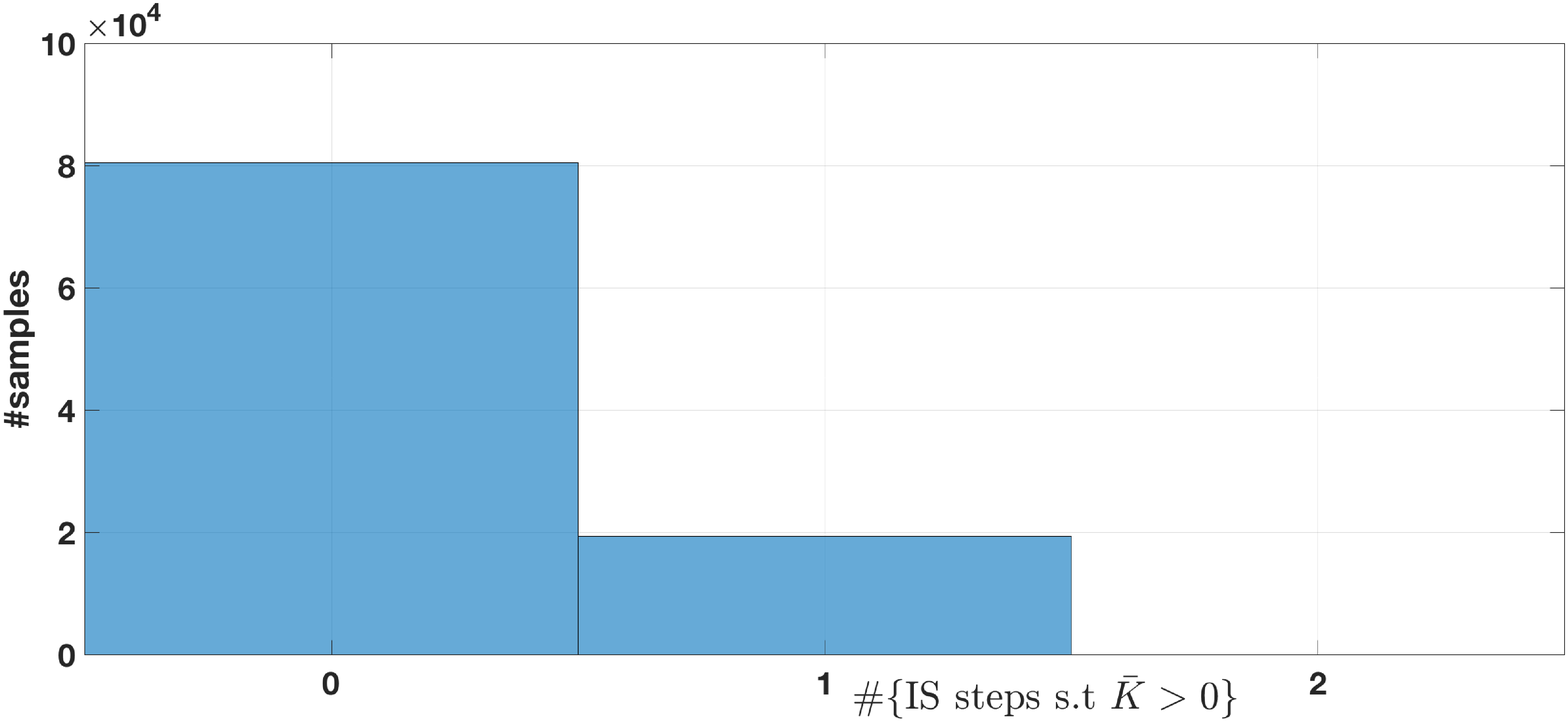}
		\caption{}
		\label{fig:sub3}
	\end{subfigure}
	\begin{subfigure}{0.4\textwidth}
		\centering
		\includegraphics[width=0.9\linewidth]{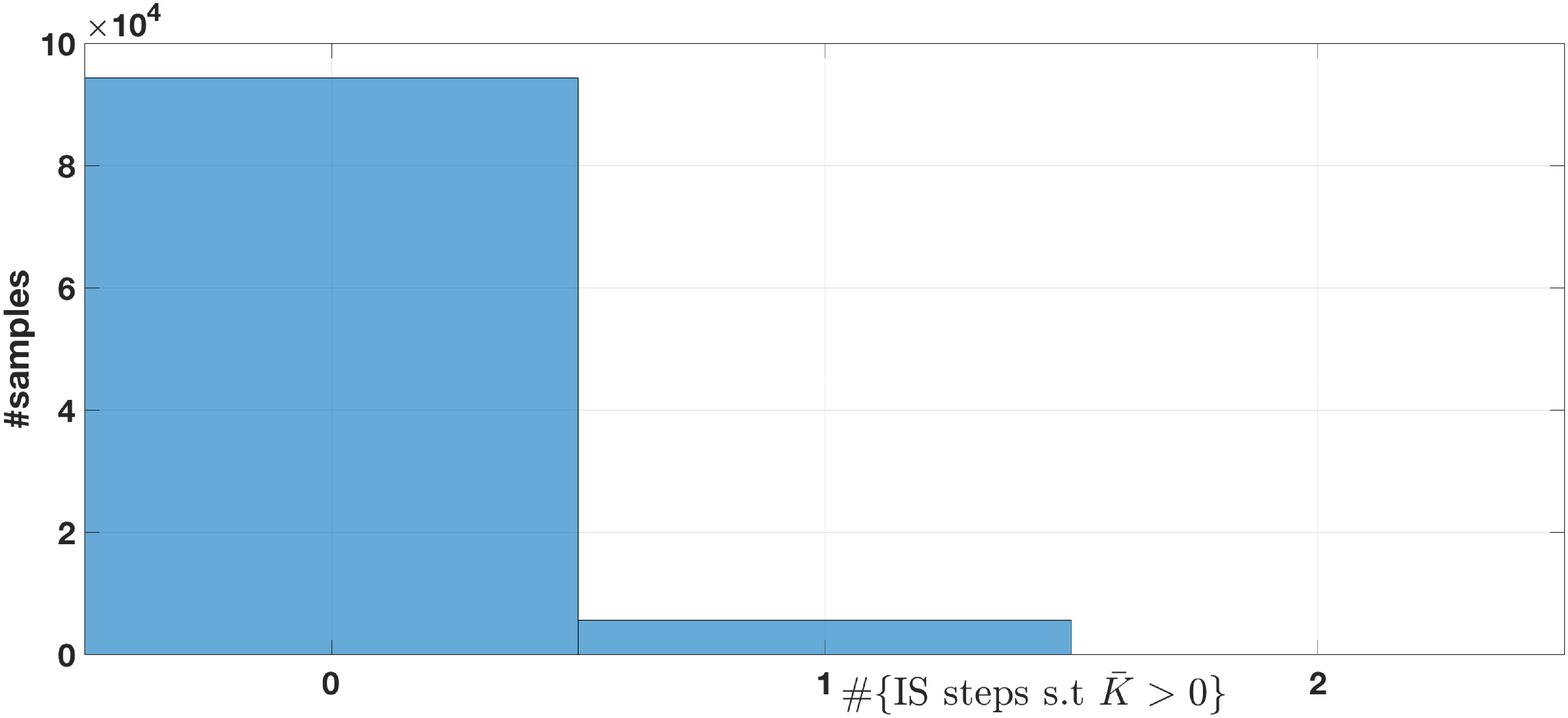}
		\caption{}
		\label{fig:sub4}
	\end{subfigure}
	\caption{ Example \ref{exp:Gene transcription and translation} with IS  (with $\delta=0.5$): Histogram of  $\# \{ \text{IS steps : s.t.}\:  \bar{K}:=\sum_{j \in \mathcal{J}_1}  \sum_{n \in \mathcal{S}_j}k^j_n>0\}$, for number of samples $M_{\ell}=10^5$.  a) $\ell=4$. b)  $\ell=8$.  }
	\label{fig:number_k_n>0_example2}
\end{figure}
\FloatBarrier
\begin{figure}[h!]
	\centering
	\begin{subfigure}{0.4\textwidth}
		\centering
		\includegraphics[width=0.9\linewidth]{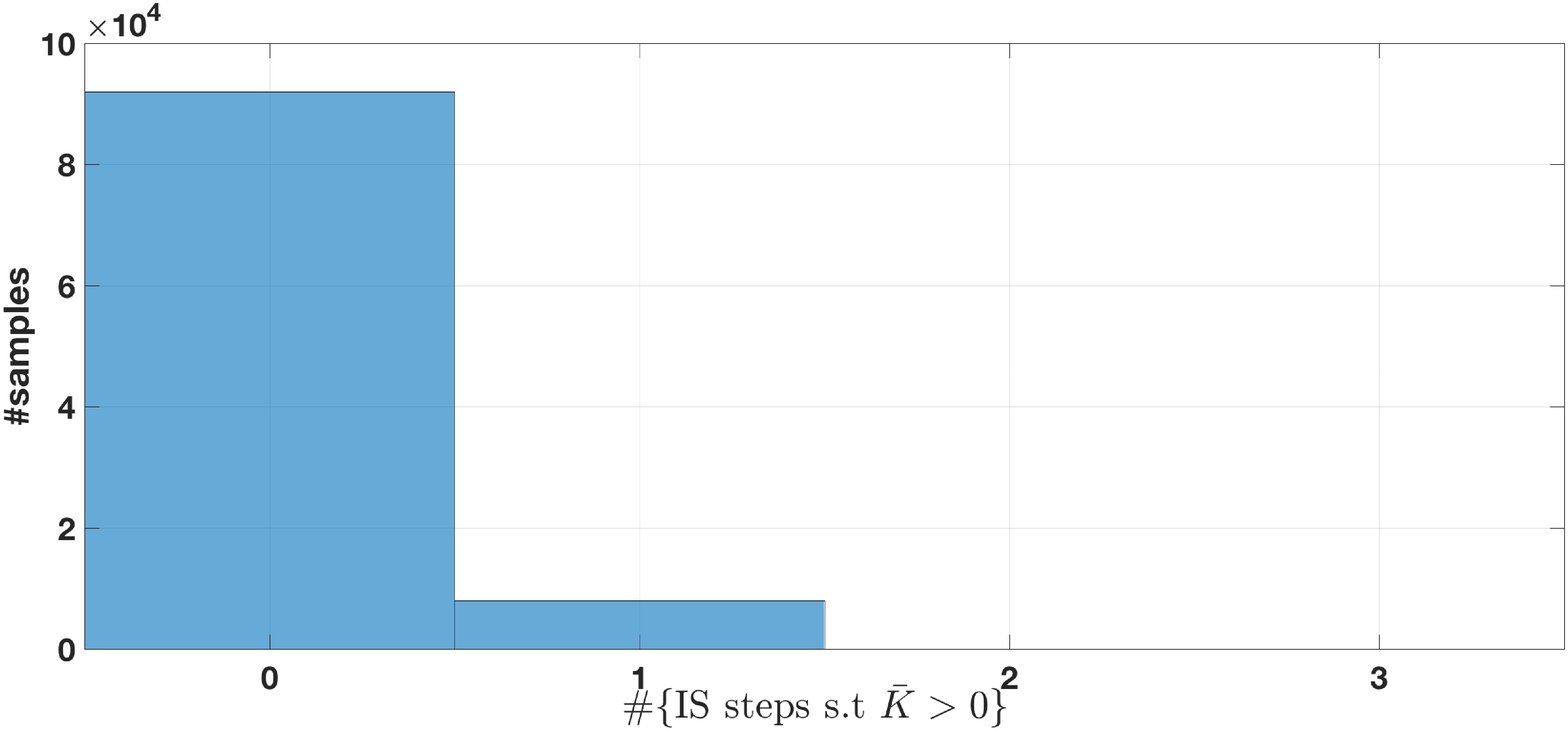}
		\caption{}
		\label{fig:sub4}
	\end{subfigure}
	\begin{subfigure}{0.4\textwidth}
		\centering
		\includegraphics[width=0.9\linewidth]{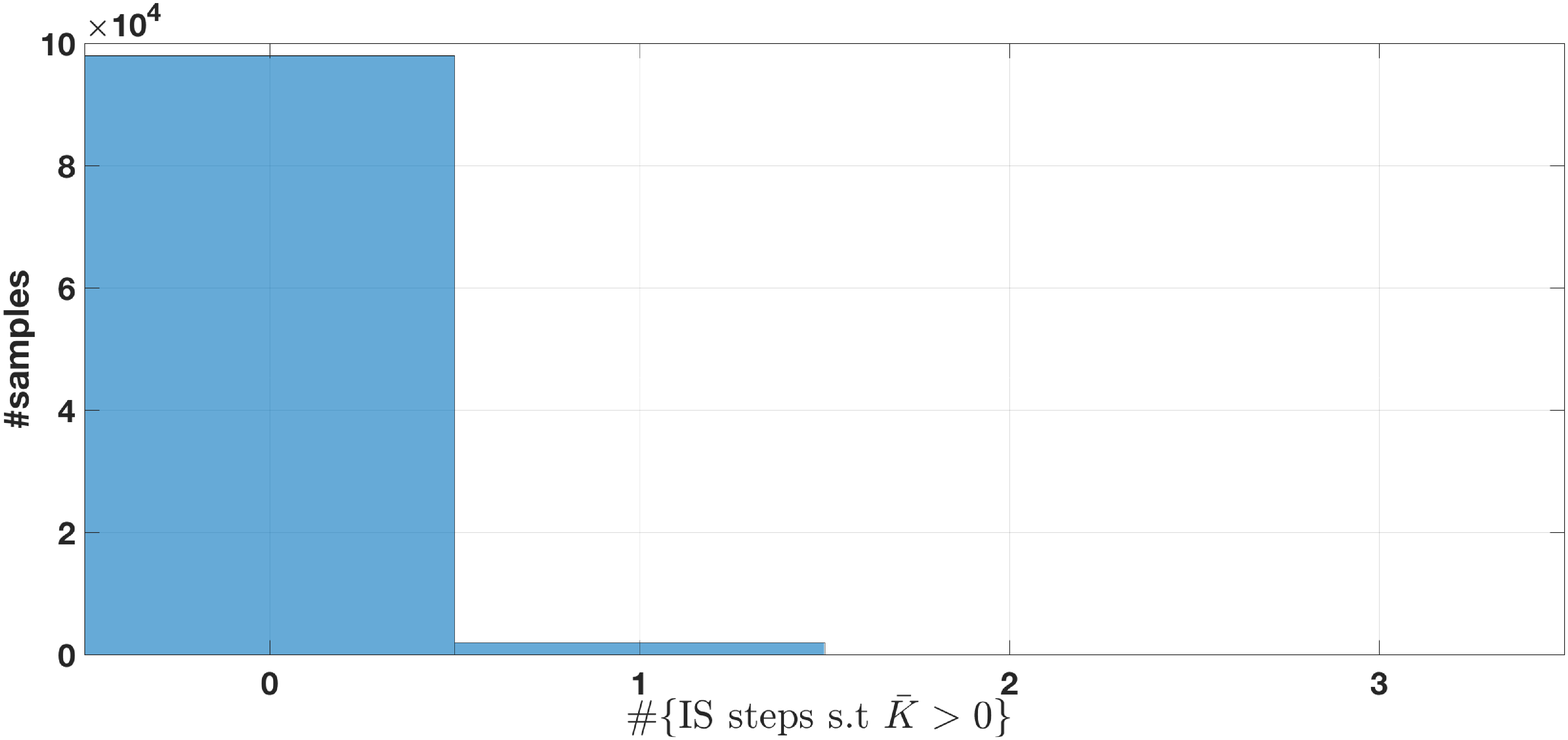}
		\caption{}
		\label{fig:sub4}
	\end{subfigure}
	\caption{ Example \ref{exp:Michaelis–Menten enzyme kinetics}  with IS  (with $\delta=0.5$): Histogram of  $\# \{ \text{IS steps : s.t.}\:  \bar{K}:=\sum_{j \in \mathcal{J}_1}  \sum_{n \in \mathcal{S}_j}k^j_n>0\}$, for number of samples $M_{\ell}=10^5$.  a)  $\ell=6$. b) $\ell=10$ }
	\label{fig:number_k_n>0_example4}
\end{figure}
\FloatBarrier